\documentclass[11pt]{article}
\usepackage[utf8]{inputenc}

\usepackage[margin=1.0in]{geometry}

\usepackage{graphics,latexsym,amsfonts,epsfig,verbatim,amsmath,algorithm,algorithmic,color}
\usepackage{multirow}
\usepackage{url}
\usepackage{todonotes}
\usepackage{graphicx}
\usepackage{mathrsfs}
\usepackage{booktabs}
\usepackage{bm}
\usepackage{tikz}

\newtheorem{theorem}{Theorem}
\newtheorem{proposition}{Proposition}

\newtheorem{remark}{Remark}
\usepackage{natbib}
\def\newblock{\ }%
\usepackage{mathtools}
\usepackage{subcaption}

\usepackage{array} 
\usepackage{paralist} 
\usepackage{verbatim} 

\usepackage{fancyhdr} 
\allowdisplaybreaks

\graphicspath{{figures/}}

\newcommand{\p}{\bm{p}}
\newcommand{\rb}{\bm{r}}
\newcommand{\ut}{\tilde{\bm{u}}}
\newcommand{\E}{\mathbb{E}}
\newcommand{\argmax}{\mathop{\rm arg~max}\limits}
\newcommand{\argmin}{\mathop{\rm arg~min}\limits}

\begin{document}
\title{Distributionally Robust Partially Observable Markov Decision Process with Moment-based Ambiguity}
	\author{Hideaki Nakao\thanks{Department of Industrial and Operations Engineering, University of Michigan at Ann Arbor, USA;}
    ~~~Ruiwei Jiang\thanks{Department of Industrial and Operations Engineering, University of Michigan at Ann Arbor, USA;}
    ~~~Siqian Shen\thanks{Corresponding author; Department of Industrial and Operations Engineering, University of Michigan at Ann Arbor, USA.
    Email: {\tt siqian@umich.edu}.}}
\date{} 
\maketitle

\begin{abstract}
We consider a distributionally robust Partially Observable Markov Decision Process (DR-POMDP), where the distribution of the transition-observation probabilities is unknown at the beginning of each decision period, but their realizations can be inferred using side information at the end of each period after an action being taken. We build an ambiguity set of the joint distribution using bounded moments via conic constraints and seek an optimal policy to maximize the worst-case (minimum) reward for any distribution in the set. We show that the value function of DR-POMDP is piecewise linear convex with respect to the belief state and propose a heuristic search value iteration method for obtaining lower and upper bounds of the value function. We conduct numerical studies and demonstrate the computational performance of our approach via testing instances of a dynamic epidemic control problem. Our results show that DR-POMDP can produce more robust policies under misspecified distributions of transition-observation probabilities as compared to POMDP, but has less costly solutions than robust POMDP. The DR-POMDP policies are also insensitive to varying parameter in the ambiguity set and to noise added to the true transition-observation probability values obtained at the end of each decision period.
\end{abstract}

\section{Introduction}\label{sec:introduction}
Partially Observable Markov Decision Processes (POMDPs) are useful for modeling sequential decision making problems, where a decision maker (DM) is only able to obtain partial information about the present state of a system of interest. Similar to the Markov Decision Processes (MDPs), the transition probabilities in between the states of the system depend on the current state and the action chosen by the DM. In addition, POMDPs are accompanied with a set of observation outcomes that are realized probabilistically given the DM's action and the state into which the system has transitioned. Different from MDPs where the DM is able to directly observe the current state of the system, in POMDPs the DM can only view an observation instead of the true state. Applications of POMDPs include clinical decision making, inventory control, machine repair, epidemic intervention and many more \cite{cassandra1998survey,hauskrecht2000planning,treharne2002adaptive}. 

A general objective in sequential decision making is to devise a policy of taking dynamic actions to maximize (minimize) the expected value of the cumulative reward (cost). In MDPs, the DM gains a reward (or pays a cost) for each action made on a state of the system. In POMDPs, since the DM has no access to the true state, she is uncertain about the reward (cost) received. Instead, the DM retains her belief of the present state based on past actions and observations, and anticipates an expected value of the reward (or the expected cost) based on the belief. The DM's belief is represented by a probability mass associated with each state of the system, which is a sufficient statistic of the history of past actions and observations \cite[Chapter 6.6]{kumar2015stochastic}. Since a policy is a function of the past actions and observations, this property is useful to compactly represent an increasing sequence of information.

In POMDPs, a critical assumption is that the exact transition and observation probabilities are known to the DM for each action-state combination. In practice, there may exist estimation errors about either the transition or observation probability values, to handle which, \cite{rasouli2018robust} builds an uncertainty set of probabilities and develops an exact algorithm for the problem of maximizing the expected reward in the worst-case realization of the unknown probabilities in POMDPs. We will numerically compare actions of robust POMDP (see \cite{osogami2015robust}) with decision policies of DR-POMDP and POMDP in Section~\ref{sec:comp}.

In this paper, using bounded moments, we construct an ambiguity set of the unknown joint distribution of the transition-observation probabilities, in which the true joint distribution lies with high probability. We consider a distributionally robust optimization framework of POMDPs (called DR-POMDP) to seek an optimal policy against the worst-case distribution in the ambiguity set, when realizations of the transition and observation probabilities in each decision period are generated from this distribution. Moreover, we allow transition-observation probabilities to vary in different decision periods, and assume that at the end of each period, the DM can gather side information to infer the true values of the transition-observation probabilities realized in that period, even these values were unknown to the DM when decisions were made. Admittedly, it is rather restrictive to have this assumption where the transition-observation probabilities can be observed retrospectively. However, there exist a wide range of applications where the underlying dynamics are understood and can be simulated to produce unknown parameters (i.e., transition-observation probabilities) once values of some exogenous parameters are gained after the decisions are made. For example, Mannor et al. \cite{mannor2016robust} justify the electric power system as one case where the system performance can be reliably simulated when environmental factors, such as wind and solar radiation levels, are known. 
In Section \ref{sec:Example}, we provide a few examples to further illustrate and justify this assumption and in Section \ref{sec:comp}, we conduct numerical tests on dynamic epidemic control problem instances, which satisfy the assumption. 

In distributionally robust optimization (DRO), we seek solutions to optimize the worst-case objective given by possible distributions contained in an ambiguity set. Compared with robust optimization that accounts for the worst-case objective outcome given by all possible realizations of uncertain parameters in an uncertainty set, optimal solutions to DRO models are less conservative  and can be adjusted through the amount of data/information we have. 
Ref.\ \cite{delage2010distributionally} develops a moment-based ambiguity set, considering a set of distributions with an ellipsoidal condition on the mean and a conic constraint on the second-order moment, to derive tractable reformulations of several distributionally robust convex programs. Standardization of ambiguity sets via conic representable sets is proposed by \cite{wiesemann2014distributionally}. Ref.\ \cite{zymler2013distributionally} considers tractable reformulations of DR chance-constrained programs using moment-based ambiguity set. Other types of ambiguity sets used in DRO models bound the $\phi$-divergence \cite{ben2013robust, jiang2016data} or Wasserstein distance \cite{esfahani2015data, gao2016distributionally} in between possible distributions to a nominal distribution. In this paper, we also use a moment-based ambiguity set where the moment information is bounded via conic constraints. We establish the Bellman equation for DR-POMDP and prove the piecewise-linear-convex property of the value function, using which we further develop efficient computational algorithms and demonstrate the efficacy of the DR-POMDP model by testing  epidemic control problem instances with diverse parameter settings.

The remainder of the paper is organized as follows. 
In Section \ref{sec:lit}, we review the most relevant POMDP, robust MDP/POMDP, and DRO literature. In Section \ref{sec:Example}, we formally present DR-POMDP and provide a few examples to show possible applications. In Section \ref{sec:dr-or}, we formulate the Bellman equation and show that the value function is piecewise linear convex under general moment-based ambiguity sets described in \cite{yu2016distributionally}. In Section \ref{sec:approx2}, we develop an approximation algorithm for DR-POMDP based on a distributionally robust variant of the heuristic value search iteration algorithm. In Section \ref{sec:comp}, we demonstrate the computational results of solving DR-POMDP on randomly generated instances of a dynamic epidemic control problem, and compare it with POMDP and robust POMDP through different out-of-sample tests. Section \ref{sec:conclusion} concludes the paper and presents future research directions. 

\section{Literature Review}
\label{sec:lit}
Although strong modeling connections exist in between MDP and POMDP, techniques applied to solve MDP models where the states are discrete, are not directly applicable to solving POMDP since belief states are continuous. Ref.\ \cite{smallwood1973optimal} shows that the value function of POMDP is piecewise linear convex (PWLC) with respect to the belief state, and derives an exact algorithm to find an optimal policy. The exact algorithm, which keeps a set of vectors for characterizing the value function, is intractable as the search space increases exponentially over periods. Ref.\ \cite{pineau2003point} proposes a point-based value iteration (PBVI) algorithm by only keeping characterizing vectors for a subset of belief states, and thus maintains a lower bound of the true value function that aims to maximize the reward. The PBVI algorithm is polynomial in the number of states, observations, and actions, and the error induced by taking a subset of belief states is shown to be convergent if the subset  is sampled densely in the reachable set of belief states. Ref.\ \cite{smith2004heuristic} develops a heuristic search value iteration (HSVI) algorithm to derive an upper bound of the value function via finding the reachable set through 
simulation. Ref.\ \cite{smith2004heuristic} shows that HSVI is guaranteed to terminate after the gap between the upper and lower bounds converges within a certain threshold.

The research on robust MDP is motivated by possible estimation errors of transition matrices and how they may have a significant impact to the solution quality (see, e.g., \cite{abbad1992perturbation, abbad1990algorithms}). In \cite{wiesemann2013robust}, the authors show probabilistic guarantees for solutions to robust MDPs by building an uncertainty set using fully observable history. By construction, their robust  policy achieves or exceeds its worst-case performance with a certain confidence. 
Ref.\ \cite{nilim2005robust} considers robust control for a finite-state, finite-action MDP, where uncertainty on the transition matrices is described by particular uncertainty sets such as likelihood regions or entropy bounds, and the authors present a robust dynamic programming algorithm for solving the problem. Ref.\ \cite{iyengar2005robust} analyzes a robust formulation for discrete-time dynamic programming where the transition probabilities are uncertain and ambiguously known, and shows that it is equivalent to stochastic zero-sum games with perfect information. Ref.\ \cite{delage2010percentile} argues that robust MDP models may produce over-conservative solutions, as they do not incorporate the distributional information of uncertain parameters. Then \cite{xu2012distributionally} presents a distributionally robust MDP model, where the ambiguity set is characterized by a sequence of nested sets, each having a confidence level to guarantee that the true value is in the set with a certain probability. Ref.\ \cite{yu2016distributionally} generalizes the distributionally robust MDP to include multi-modal distributions and the information of mean and variance. Ref.\ \cite{yang2017convex} proposes a distributionally robust MDP model by building an ambiguity set of distributions on transition probability using a Wasserstein ball centered around a nominal distribution. The use of Wasserstein ball ambiguity set results in a Kantorovich-duality-based convex reformulation for distributionally robust MDP.

Ref.\ \cite{saghafian2018ambiguous} presents a modeling framework of ambiguous POMDP (called APOMDP), which generalizes the robust POMDP  in \cite{rasouli2018robust}. APOMDP optimizes over the $\alpha$-maxmin expected utility, resulting in a policy that can achieve the intermediate performance of the worst case and the best case in the uncertainty set of parameters. Ref.\ \cite{saghafian2018ambiguous} describes conditions under which the value function of APOMDP is PWLC. Meanwhile,  \cite{rasouli2018robust} considers a general setting of robust POMDP, where the DM may not be able to obtain the exact transition-observation probabilities even after taking actions at the end of each period. In this case, the sufficient statistic is no longer a single belief state, but a collection of belief states, and the expected reward up to the current period must be taken into account to realize a policy that is robust in terms of the entire cumulative expected reward. The authors also derive an exact algorithm for robust POMDP where the uncertainty set is discrete. Here we note that robust POMDP with a continuous uncertainty set is computationally challenging even in a very simple setting. Moreover, \cite{osogami2015robust} formulates a robust counterpart for POMDP, where the transition-observation matrix is assumed to lie in a fixed support within the probability simplex. The realized transition-observation probability values are assumed to be observable to the DM at the end of each decision period, similar to the setting in this paper. While the value function for the standard POMDP can be described by a PWLC function, the value function of the robust POMDP is not necessarily piecewise linear, as there are possibly infinitely many supporting hyperplanes. The authors derive an efficient algorithm based on PBVI to approximate the exact solution, and discusses a method to conduct a robust belief update.

\section{Problem Description}
\label{sec:Example}
Figure \ref{fig:events} depicts the sequence of events that occur during one decision period. In a distributionally robust setting, we consider another agent (the ``nature"), who chooses a distribution $\mu$ of the transition-observation probabilities from a pre-assumed ambiguity set. The DM expects that the nature may access to the same information as the DM and acts adversarially against the DM's action $a$ taken at the beginning of each period. Therefore, the distribution $\mu$ is expected to lead to the worst-case  expected reward. Next, the joint transition-observation probability $\p$ is realized from the distribution $\mu$. The state makes a transition according to $\p$, and the observation outcome $z$ is shown. Finally, the DM obtains the values of $z$ and $\p$ at the end of the period. 

\begin{figure}[htbp!]
%
\centering
\includegraphics[width=0.95\textwidth]{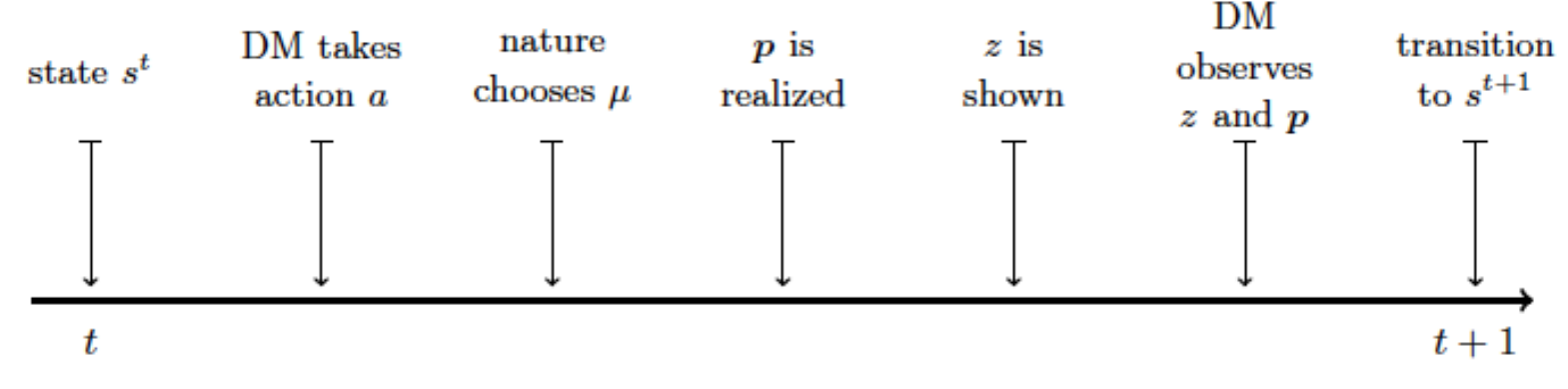}
\caption{Sequence of events during one decision period in a DR-POMDP}
\label{fig:events}
\end{figure}

We denote $\mathcal{S}$ as the set of states, $\mathcal{A}$ as the set of actions, and $\mathcal{Z}$ as the set of observation outcomes. For all $(s,s',z,a)\in\mathcal{S}^2\times\mathcal{Z}\times\mathcal{A}$, we define $p_{as}(s',z)=\mbox{Pr}(s',z|s,a)$ as the probability of transitioning between $(s,s')$ and observing $z$, given action $a$. For $(s,a)\in\mathcal{S}\times\mathcal{A}$, let $r_{as}$ be the reward for taking action $a$ at state $s$. 
%
For all $s\in\mathcal{S},\ a\in\mathcal{A}$, we define a vector of probabilities $\p_{as}=\left(p_{as}(s',z),\ (s',z)\in\mathcal{S}\times\mathcal{Z}\right)^{\top}$ and assume that the Cartesian product $(\p_{as},r_{as})$ is a member of a set $\mathcal{X}_{as}\subseteq\Delta(\mathcal{S}\times\mathcal{Z})\times\mathbb{R}$, where $\Delta(\cdot)$ is a probability simplex of  set $\cdot$. We denote $\p_{a}=(p_{as}(s',z),\ (s,s',z)\in\mathcal{S}^2\times\mathcal{Z})^{\top}$ and $\rb_{a}=(r_{as},\ s\in\mathcal{S})^{\top}$ for all $a\in\mathcal{A}$. 
We assume that $(\p_{as},r_{as})$ follows a distribution $\mu_{as}$, which is unknown but is included in an ambiguity set $\mathcal{D}_{as}\subseteq\mathcal{P}(\mathcal{X}_{as})$, where $\mathcal{P}(\cdot)$ represents a set of all probability distributions with support $\cdot$. 
Furthermore, the set of distributions is rectangular with respect to the set of actions $\mathcal{A}$ and the set of states $\mathcal{S}$, i.e., the overall ambiguity set is $\mathcal{D}=\bigotimes_{\substack{a\in\mathcal{A}\\s\in\mathcal{S}}}D_{as}$. This assumption is analogous to the $(s,a)$-rectangularity in \cite{wiesemann2013robust}. The above conditions increase the conservativeness of the model in general. In the online supplement \cite{online-sup} \ref{sec:arectangular}, we discuss a relaxation of the $a$-rectangularity assumption for DR-POMDP.

Below we describe several examples in which the above settings of DR-POMDP can be justified, and therefore our approach can be applied to optimize corresponding policies. The key is to justify whether the DM can obtain the true value of $\p$ using side information at the end of each decision period. In Section \ref{sec:comp}, we also numerically show that our approach can produce quite stable reward in out-of-sample simulation tests even we add noise to the true $\p$-value obtained at the end of each period and thus the assumption is relatively weak. 

First, consider dynamic epidemic surveillance and control. During a flu season, 
the number of weekly visits of patients who show influenza-like illness (ILI) symptoms is reported to the public. The number of ILI patients divided by the total population, called the ILI rate, is frequently used to estimate the prevalence of an epidemic. For example, \cite{rath2003automated} studies a two-state MDP model (i.e., epidemic vs.\ non-epidemic) and shows that the ILI rate follows a Gaussian and an exponential distribution for the epidemic and non-epidemic state, respectively; \cite{le1999monitoring} uses ILI rate to predict influenza epidemics through a hidden Markov model. The hidden states correspond to the current epidemic level, which is unobservable to the DM due to incubation period and patient arrival latency. Different epidemic levels also cause different probabilities of the population visiting healthcare providers, which will then be reflected in ILI rate. 

Arguably, the transition probabilities and ILI rates are dependent on government control policies, such as restricting travels, stopping mass gatherings, and so on. These decisions often have to be made before knowing the true transition matrix and observation probabilities between ILI rate and the true epidemic state. The DR-POMDP seeks a policy to minimize the worst-case expected cost (e.g., the total infected count, death toll, etc.) and at the end of each decision period, side information such as humidity, antigenic evolution of the virus, and population travels in the past period can be used to infer the true transition and ILI-rate observation probabilities \cite[see, e.g.,][]{du2017evolution}. Note that the side information is not available at the beginning of each decision period when the DM takes an action, but can be collected at the end of each period.

Another example arises in clinical decision-making such as deciding prostate cancer treatment plans \cite{zhang2018partially}, where different treatment plans can probabilistically vary cancer conditions (i.e., states) of a patient. The true state of a cancer patient is hard to know but can be inferred probabilistically from belief states. Using DR-POMDP, a doctor's objective is to provide treatment and inspection as needed in order to minimize the maximum expected quality-adjusted life years for each patient under ambiguously known transition-observation probabilities. According to \cite{zhang2018partially}, the detection of prostate-specific antigen (PSA), has a varying accuracy rate depending on the patient's condition. After treatment in each period, the doctor can utilize the PSA information to infer the true transition and observation probabilities happening to the patient and update her belief to make treatment plans for the next period.

One can also consider planning production or maintaining inventory in highly seasonal industries such as agriculture \cite{treharne2002adaptive}, where system states correspond to market trends in each decision period. The trend makes a transition according to a probability mass function that is unknown to the DM and each trend is associated with a certain distribution of demand that the DM aims to satisfy. For a certain product, the market transition probability and the demand distribution are correlated with climate factors, such as temperature and precipitation, which are uncertain to the DM when she makes a production plan and thus using DR-POMDP, the goal is to minimize the maximum demand loss due to distributional ambiguity. After each period, the DM observes the realized temperature and precipitation and also the true demand, to identify the true value of $\p$. 

\section{Optimal Policy for DR-POMDP}
\label{sec:dr-or}
We derive an optimal policy for DR-POMDP when the DM can obtain the value of transition-observation probability at the end of each decision period. In Section \ref{sec:DRBellman}, we formulate DR-POMDP as an optimization problem and construct the Bellman equation to derive the optimal policy. In Section \ref{sec:DRBellmanSol}, we show that the value function satisfying the Bellman equation is PWLC. Finally, in Section \ref{sec:infinte}, we consider the infinite-horizon case, and demonstrate that the value function converges under the Bellman update operation.

\subsection{Distributionally Robust Bellman Equation}\label{sec:DRBellman}
We formulate a dynamic game involving two players: The DM selects $a\in\mathcal{A}$ and then the nature selects  $\mu_a=\bigotimes_{s\in\mathcal{S}}\mu_{as}$ from the ambiguity set $D_a=\bigotimes_{s\in\mathcal{S}}\mathcal{D}_{as}$ to minimize the expected reward given the DM's action $a$. Let $a^t$, $\p_{a^t}^t$, $z^t$ be the action, transition-observation probability outcome, and observation during decision period $t$. We denote $\mathcal{H}^t$ as the set of all possible histories up to period $t$, and denote $h^t=\left(a^1,\p_{a^1}^1,z^1,\ldots,a^{t-1},\p_{a^{t-1}}^{t-1},z^{t-1}\right)$ as a history in $\mathcal{H}^t$. The DM's objective is to find an optimal policy of selecting an action $a\in\mathcal{A}$ based on the history from $t=1$ to $T$, i.e., finding the best policy $\pi=(\pi^1,\ldots,\pi^{T-1})$ with $\pi^t: \mathcal{H}^t\rightarrow\mathcal{A}$. We denote the set of all such policies as $\Pi$, and define an extended history $\tilde{h}^t=\left(a^1,\p_{a^1}^1,z^1,\ldots,a^{t-1},\p_{a^{t-1}}^{t-1},z^{t-1},a^t\right)\in\tilde{\mathcal{H}}^t$, on which the nature bases its decision for choosing $\mu_{a^t}$. The nature's objective is to find the best policy (from the nature's perspective) $\gamma=(\gamma^1,\ldots,\gamma^{T-1})$, with $\gamma^t:\tilde{\mathcal{H}}^t\rightarrow\mathcal{D}_{a^t}$ to minimize the expected reward. Similarly, we denote the set of all the nature's policies as $\Gamma$.

Rasouli and Saghafian \cite{rasouli2018robust} point out that the sufficient statistic for robust POMDP is no longer a single belief state, but a set of belief states. Moreover, they discuss that the set of belief states by itself cannot be used to construct an optimal policy since there exists uncertainty for the reward accumulated in the past, associated with each of the belief states. 
Because of the uncertainty in the expected reward, the DM must consider a belief state that achieves the smallest expected reward both in the past and the future, posing great challenge for optimization. We claim that a similar observation holds true for the distributionally robust case.  However, when the DM can obtain the value of transition-observation probability at the end of each decision period, the ambiguity of the belief state, as well as the expected reward diminishes and the single belief state becomes a sufficient statistic for DR-POMDP, which can also be used to characterize the optimal policy.

Let the belief state in period $t$ be $(b_{s}^t,\ s\in\mathcal{S}) =\bm{b}^t\in\Delta(\mathcal{S})$. Given action $a$, transition-observation probability $\p_{a}$, and observation outcome $z$, the sufficient statistic for the history $h^{t+1}=(h^t,a,\p_{a},z)$, or the belief state in period $t+1$ is given by
\begin{align}\label{eq:bayesian}
\bm{b}^{t+1}=\bm{f}(\bm{b},a,\p_a,z)=\frac{\sum_{s\in\mathcal{S}}{\bm{J}_z}\p_{as}b_s}{\sum_{s\in\mathcal{S}}\bm{1}^{\top}{\bm{J}_z}\p_{as}b_s},
\end{align}
where $\bm{1}$ represents a vector of ones having the length $|\mathcal{S}|$; ${\bm{J}_z}\in\mathbb{R}^{|\mathcal{S}|\times(|\mathcal{S}|\times|\mathcal{Z}|)}$ is a matrix of zeros and ones that projects the vector $\p_{as}$ to a vector $\p_{asz}=\left(p_{as}(s',z),\ s'\in\mathcal{S}\right)^{\top}$, whose entries correspond to the outcome $z$. That is, $\p_{asz} = \bm{J}_z\p_{as}, \ \forall a, \ s, \ z$. Note that the belief state cannot be updated using \eqref{eq:bayesian} and will not be a sufficient statistic of the history of past actions and observations if we do not have the true values of $\p_{as}$.

With slight abuse of notation, let $\pi$ be a policy that maps belief states to the actions, i.e., $\pi^t:\ \Delta(\mathcal{S})\rightarrow\mathcal{A}$ for all $t\in\{1,\ldots,T-1\}$. Similarly, let $\gamma^t:\ \Delta(\mathcal{S})\times \mathcal{A}\rightarrow\mathcal{D}_{a^t}$ for all $t\in\{1,\ldots,T-1\}$. Note that the nature's policy is dependent on the belief state since the nature acts adversarial to the DM. 

\begin{remark}
Note that the deterministic policy is optimal since the nature is able to access to the same information as the DM, plus the action that the DM has performed. This does not hold true when the nature is not able to perfectly access to the DM's immediate action.
\end{remark}

Given the nature's choice of distribution $\mu_a$, the expected value of the instantaneous reward given belief state $\bm{b}$ and action $a$ is denoted as $\E_{(\p_a,\rb_a)\sim\mu_{a}}\left[\bm{b}^{\top}\bm{r}_a\right]$, where ``$\sim$" expresses the relation between random variables and probability distributions. Let $\beta\in(0,1]$ be a discount factor. The objective of the DM is to find a policy to maximize the minimum cumulative discounted expected reward given all possible policies (i.e., distributions of transition-observation probabilities) by the nature. That is, DR-POMDP aims to solve 
\begin{subequations}\label{eq:optimization}
\begin{align}
\max_{\pi\in\Pi}\min_{\gamma\in\Gamma}\quad&\E\left[\sum_{t=1}^{T-1}\beta^t {\bm{b}^{t}}^{\top} \rb_{a^t}^t\right]\\
\mbox{s.t.}\quad& a^t=\pi^t(\bm{b}^t), &&\forall t\in\{1,\ldots,T-1\}\\
&\mu^t_{a^t}=\gamma^t(\bm{b}^t,a^t),&&\forall t\in\{1,\ldots,T-1\}\\
&(\p^t_{a^t},\rb^t_{a^t})\sim\mu^t_{a^t},&&\forall t\in\{1,\ldots,T-1\}\\
&(s^{t+1},z^t)\sim\p^t_{a^ts^t},&&\forall t\in\{1,\ldots,T-1\}\\
&\bm{b}^{t+1}= \bm{f}(\bm{b}^t,a^t,\p_{a^t}^t,z^t),&&\forall t\in\{1,\ldots,T-1\}
\end{align}
\end{subequations}
where the terminal reward is zero without loss of generality. The initial belief state is given as $\bm{b}$. Alternatively, we denote the problem \eqref{eq:optimization} as
\begin{align}\label{eq:simplified}
\max_{\pi\in\Pi}\min_{\gamma\in\Gamma}\E \left[\sum_{t=1}^{T-1}\beta^t{\bm{b}^{t}}^{\top} \rb_{a^t}^t\Biggr|\ \bm{b}^1=\bm{b}\right].
\end{align}
Here we omit all the constraints in \eqref{eq:optimization} for presentation simplicity.

To solve \eqref{eq:optimization}, we propose to use dynamic programming, and derive the Bellman equation below.  
\begin{proposition}\label{prop:dp}
Denote $\pi^{t:T-1}=(\pi^t,\pi^{t+1},\ldots,\pi^{T-1})$ and $\gamma^{t:T-1}=$\hfill \linebreak $(\gamma^t,\gamma^{t+1},\ldots,\gamma^{T-1})$ as sequences of policies from $t$ to $T-1$. Let  $\Pi^{t:T-1}$ and $\Gamma^{t:T-1}$ be the sets of all policies $\pi^{t:T-1}$ and $\gamma^{t:T-1}$, respectively.  Consider the value function in period $t$ as 
\begin{align}\label{eq:vf_source}
V^t(\bm{b})=\max_{\pi^{t:T-1}\in\Pi^{t:T-1}}\min_{\gamma^{t:T-1}\in\Gamma^{t:T-1}}\E\left[\sum_{n=t}^{T-1}\beta^{n-t}{\bm{b}^{n}}^{\top} \rb_{a^n}^n\Biggr|\ \bm{b}^t=\bm{b}\right].
\end{align}
Then,
\begin{align}\label{eq:vf2}
V^t(\bm{b})&=\max_{a\in\mathcal{A}}\min_{\mu_a\in\mathcal{D}_a}\E_{(\p_{a},\rb_a)\sim\mu_{a}}\Biggl[\sum_{s\in\mathcal{S}}b_s\Biggl\{r_{as}+\beta\sum_{z\in\mathcal{Z}}\bm{1}^{\top}{\bm{J}_z}\p_{as}V^{t+1}\left(\bm{f}\left(\bm{b},a,\p_a,z\right)\right)\Biggr\}\Biggr].
\end{align}
\end{proposition}

\begin{proof}
We first isolate the term associated with period $t$ inside the expectation of \eqref{eq:vf_source} as follows.
\begin{align*}
V^t(\bm{b})&=\max_{\pi^{t:T-1}\in\Pi^{t:T-1}}\min_{\gamma^{t:T-1}\in\Gamma^{t:T-1}} \E\left[{\bm{b}^{t}}^{\top} \rb_{a^t}^t+\beta\sum_{n=t+1}^{T-1}\beta^{n-(t+1)}{\bm{b}^{n}}^{\top} \rb_{a^n}^n\Biggr|\ \bm{b}^t=\bm{b}\right].
\end{align*}
Given $a^t=\pi^t(\bm{b}),\ \p^t_a=\p_{\pi^t(\bm{b})},\ z^t=z$, the probability of observing $z$ is 
\begin{align*}
\sum_{s\in\mathcal{S}}b_s\bm{1}^{\top}{\bm{J}_z}\p_{\pi_t(\bm{b})s}.
\end{align*}
Thus, we can calculate the expectation conditioned on the values of $a^t,\ \p_a^t,\ z^t$ in the value function as: 
\scriptsize
\begin{eqnarray*}
V^t(\bm{b})&= & \max_{\pi^{t:T-1}\in\Pi^{t:T-1}}\min_{\gamma^{t:T-1}\in\Gamma^{t:T-1}}\E_{(\p_{\pi_t(\bm{b})},\rb_{\pi_t(\bm{b})})\sim\mu_{\pi^t(\bm{b})}}\Biggl[\sum_{s\in\mathcal{S}}b_sr^t_{\pi^t(\bm{b})s}\\
& & \quad+\sum_{z\in\mathcal{Z}}\sum_{s\in\mathcal{S}}b_s\bm{1}^{\top}{\bm{J}_z}\p_{\pi_t(\bm{b})s}{\E}\Biggl[\sum_{n=t+1}^{T-1}\beta^{n-(t+1)}{\bm{b}^{n}}^{\top} \rb_{a^n}^n \Biggr|\ \bm{b}^t=\bm{b},a^t=\pi^t(\bm{b}),\p^t_a=\p_{\pi^t(\bm{b})},z^t=z\Biggr]\Biggr\}\Biggr]\\
&= &\max_{\pi^{t:T-1}\in\Pi^{t:T-1}}\min_{\gamma^{t:T-1}\in\Gamma^{t:T-1}}\E_{(\p_{\pi_t(\bm{b})},\rb_{\pi_t(\bm{b})})\sim\mu_{\pi^t(\bm{b})}}\Biggl[\sum_{s\in\mathcal{S}}b_s\Biggl\{r^t_{\pi^t(\bm{b})s}\\
& & \quad+\beta\sum_{z\in\mathcal{Z}}\bm{1}^{\top}{\bm{J}_z}\p_{\pi_t(\bm{b})s}{\E}\Biggl[\sum_{n=t+1}^{T-1}\beta^{n-(t+1)}{\bm{b}^{n}}^{\top} \rb_{a^n}^n \Biggr|\ \bm{b}^{t+1}=\bm{f}\left(\bm{b},\pi^t(\bm{b}),\p_{\pi^t(\bm{b})},z\right)\Biggr]\Biggr\}\Biggr],
\end{eqnarray*}
\normalsize
where the second equality is due to rearranging the terms and the fact that $\bm{b}$ is an information state.
Because policies beyond period $t$ do not affect $(\p_{a^t}^t,\rb_{a^t}^t)$, we have
\scriptsize
\begin{align*}
V^t(\bm{b})&=\max_{a\in\mathcal{A}}\min_{\mu_a\in\mathcal{D}_a}\E_{(\p_{a},\rb_a)\sim\mu_{a}}\Biggl[\sum_{s\in\mathcal{S}}b_s\Biggl\{r_{as}+\beta\sum_{z\in\mathcal{Z}}\bm{1}^{\top}{\bm{J}_z}\p_{as}\\
&\quad\times\max_{\pi^{t+1:T-1}\in\Pi^{t+1:T-1}}\min_{\gamma^{t+1:T-1}\in\Gamma^{t+1:T-1}}{\E}\Biggl[\sum_{n=t+1}^{T-1}\beta^{n-(t+1)}{\bm{b}^{n}}^{\top} \rb_{a^n}^n \Biggr|\ \bm{b}^{t+1}=\bm{f}\left(\bm{b},a,\p_a,z\right)\Biggr]\Biggr\}\Biggr]\\
&=\eqref{eq:vf2}.
\end{align*}
\normalsize
The final equality follows the definition of $V^{t+1}$. This completes the proof. 
\end{proof}
Following Proposition \ref{prop:dp}, the policies optimal to \eqref{eq:simplified} can be determined by recursively solving \eqref{eq:vf2} from period $T$ to $t=1$.

Now define two functions: 
\small
\begin{align}
\label{eq:U}U^t(\bm{b},a,\mu_{a})&=\E_{(\p_{a},\rb_a)\sim\mu_{a}}\Biggl[\sum_{s\in\mathcal{S}}b_s\Biggl\{r_{as}+\beta\sum_{z\in\mathcal{Z}}\bm{1}^{\top}{\bm{J}_z}\p_{as}V^{t+1}\left(\bm{f}\left(\bm{b},a,\p_a,z\right)\right)\Biggr\}\Biggr],\\
\label{eq:Q}Q^t(\bm{b},a)&=\min_{\mu_a\in\mathcal{D}_a}U^t(\bm{b},a,\mu_{a}).
\end{align}
\normalsize
The solution to the Bellman equation provides the optimal action given belief state $\bm{b}$. That is, an optimal action for the DM in period $t$ is
\begin{align*}
\argmax_{a\in\mathcal{A}}Q^t(\bm{b},a),
\end{align*}
whereas the optimal distribution chosen by the nature, under belief state $\bm{b}$ and the DM's action $a$, is 
\begin{align*}
\argmin_{\mu_a\in\mathcal{D}_a}U^t(\bm{b},a,\mu_{a}).
\end{align*}

\subsection{Properties of Distributionally Robust Bellman Equation \eqref{eq:vf2}}\label{sec:DRBellmanSol}
We consider an ambiguity set based on mean absolute deviation of transition-observation probabilities as described below. We refer the readers to the online supplement \cite{online-sup} \ref{app:generalAS} for a more general ambiguity set that can also involve ambiguity in the reward, and the mean values are on an affine manifold with conic representable support. The same property here holds for DR-POMDP with the general ambiguity set and we omit the details for presentation simplicity. 

Suppose that the expected value of the deviation of the transition-observation probability from its mean value $\bar{\p}_{as}$ is at most $\bm{c}_{as}$. Then for all $a\in\mathcal{A}$ and $s\in\mathcal{S}$, the unknown distribution $\mu_{as}$ satisfies  $\E_{\p_{as}\sim\mu_{as}}\left[|\p_{as}-\bar{\p}_{as}|\right]\leq\bm{c}_{as}$, which is reformulated as: 
\begin{align*}
&\E_{(\p_{as},\ut_{as})\sim\tilde{\mu}_{as}}\left[\ut_{as}\right]=\bm{c}_{as},\\
&\tilde{\mu}_{as}\left(
\begin{tabular}{cc}
$\ut_{as}\geq\p_{as}-\bar{\p}_{as}$, & $\bm{1}^{\top}\p_{as}=1$\\
$\ut_{as}\geq\bar{\p}_{as}-\p_{as}$, & $\p_{as}\geq 0$
\end{tabular}
\right)=1.
\end{align*}

Here, $\ut_{as}\in\mathbb{R}^{|\mathcal{S}|\times|\mathcal{Z}|}$ denotes a vector of auxiliary variables, and $\tilde{\mu}_{as}$ is a joint distribution of $(\p_{as}, \ut_{as})$. This notation is introduced to differentiate from $\mu_{as}$, which represents the true distribution of $\p_{as}$. The ambiguity set for distribution $\tilde{\mu}_{as}$ is therefore
\begin{align}\label{eq:ambiguityset}
\tilde{\mathcal{D}}_{as}=\left\{\tilde{\mu}_{as}
\begin{pmatrix}
\p_{as} \\ \ut_{as}
\end{pmatrix}
\middle\vert
\begin{array}{ll}
\E_{(\p_{as},\ut_{as})\sim\tilde{\mu}_{as}}\left[\ut_{as}\right]=\bm{c}_{as}\\
\tilde{\mu}_{as}\left(\mathcal{X}_{as}\right)=1
\end{array}
\right\}, 
\end{align}
while the support $\tilde{\mathcal{X}}_{as}$ for $(\p_{as}, \ut_{as})$ is given by 
\begin{align}\label{eq:support}
\tilde{\mathcal{X}}_{as}=\left\{
\begin{pmatrix}
\p_{as} \\ \ut_{as}
\end{pmatrix}\in
\begin{matrix}
\mathbb{R}^{|\mathcal{S}|\times|\mathcal{Z}|}_+\\
\mathbb{R}^{L}
\end{matrix}
\middle\vert\ 
\begin{matrix}
\ut_{as}\geq\p_{as}-\bar{\p}_{as}\\
\ut_{as}\geq\bar{\p}_{as}-\p_{as}\\
\bm{1}^{\top}\p_{as}=1
\end{matrix}
\right\}.
\end{align}
 
For ambiguity sets and supports respectively defined in terms of \eqref{eq:ambiguityset} and \eqref{eq:support}, we show that the value function is convex with respect to the belief state $\bm{b}$ for each decision period. 
\begin{theorem}
\label{thm:pwl}
For all $a\in\mathcal{A}$ and $s\in\mathcal{S}$, let the ambiguity set and support be \eqref{eq:ambiguityset} and \eqref{eq:support}, respectively. For all $t\in\{1,\ldots,T\}$, there exists  a set $\Lambda^t$ of slopes such that the value function can be expressed as follows.
\begin{align}\label{eq:pwlc2}
V^t(\bm{b})=\max_{\bm{\alpha}\in\Lambda^t}\bm{\alpha}^{\top}\bm{b}.
\end{align}
\end{theorem}
A detailed proof of Theorem \ref{thm:pwl} is shown in the online supplement \cite{online-sup} \ref{sec:proofs}. Following this result, having provided the values of $a$ and $\bm{\alpha}_{az}$, the inner minimization in \eqref{eq:valuefunc} can be solved efficiently using linear programming. The issue, however, is that there are possibly infinitely many elements in  $\mbox{Conv}\left(\Lambda^{t+1}\right)$, and even if there are finitely many, the number of supporting hyperplanes $\bm{\alpha}$ inside $\Lambda^t$ increases exponentially as the value functions are calculated from period $t=T$ to $t=1$. We describe in Section \ref{sec:approx2} a heuristic search value iteration (HSVI) algorithm for efficiently computing optimal policies in DR-POMDP. 

\subsection{Case of Infinite Horizon}\label{sec:infinte}
We show that the PWLC property of the value function can be extended to the case with infinite horizon. We prove the result by following the Banach fixed point theorem (see, e.g., \cite{puterman2014markov}), and show that by repeatedly updating the value function in \eqref{eq:vf2}, it converges to a unique function corresponding to the optimal value $V^*$ of the infinite-horizon DR-POMDP problem.
\begin{theorem}\label{thm:fixedpoint2}
The operator $\mathcal{L}$ defined as 
\footnotesize
\begin{align}\label{eq:bellmaninf}
\mathcal{L}V(\bm{b})=\max_{a\in\mathcal{A}}\min_{\mu_{a}\in\tilde{\mathcal{D}}_{a}}&\E_{(\p_a,\rb_a)\sim\mu_{a}}\left[\sum_{s\in\mathcal{S}}b_s\left(r_{as}+\beta\sum_{z\in\mathcal{Z}}\bm{1}^{\top}{\bm{J}_z}\p_{as}V\left(\bm{f}(\bm{b},a,\p_a,z)\right)\right)\right]
\end{align}
\normalsize
is a contraction for $0<\beta<1$.
\end{theorem}

We refer the readers to a detailed proof provided in \ref{sec:proofs} in the online supplement \cite{online-sup}. 
Theorem \ref{thm:fixedpoint2} suggests that by employing the exact algorithm discussed in the finite horizon case, starting from any initial value function, the value function $V$ converges to an optimal function $V^*$ with rate $\beta$ by iteratively performing the Bellman operator $\mathcal{L}$. Therefore, we can use the same solution approach to be discussed in Section \ref{sec:approx2} for handling both finite-horizon and infinite-horizon cases of DR-POMDP.

\section{Solution Method}
\label{sec:approx2}
We present a variant of the HSVI algorithm proposed in \cite{smith2004heuristic} (originally for solving POMDP) for efficiently computing upper and lower bounds for DR-POMDP. We maintain a set of finite number of hyperplanes $\Lambda_{\underline{V}}$, where the resulting PWLC function $\underline{V}$ bounds the true value function from below. We also maintain a set of points $\Upsilon_{\overline{V}}$ whose elements are $(\bm{b},v)$, which is a combination of a belief $\bm{b}$ and an upper bound $v$ of the true value function at the belief $\bm{b}$. Therefore, the resulting PWLC function $\overline{V}$ bounds the value function from above. The upper bound $v$ corresponding to a belief $\bm{b}$ is obtained through 
sampling. The sampling follows a greedy strategy to close the gap between the upper bound $\overline{V}$ and the lower bound $\underline{V}$ for the belief points that are reachable from the initial belief. 

\begin{algorithm}
\caption{Heuristic Search Value Iteration (HSVI)}
\label{alg:main}
\begin{algorithmic}[1]
	\STATE {\bf Input:} initial belief state $\bm{b}^0$, tolerance $\epsilon$
	\STATE {\bf Initialize:} $\overline{V}$, $\underline{V}$ (see details in Section \ref{subsubsec:init}) \label{step:initial}
	\WHILE{$\overline{V}(\bm{b}^0)-\underline{V}(\bm{b}^0)>\epsilon$ or time limit is reached}
		\STATE $DR\mbox{-}BoundExplore(\bm{b}^0,0)$ (see details in Algorithm \ref{alg:explore})\label{step:explore}
	\ENDWHILE
	\STATE {\bf Output:} $\overline{V}$, $\underline{V}$
\end{algorithmic}
\end{algorithm}

Algorithm \ref{alg:main} presents the main algorithmic steps in HSVI, where the details of Step \ref{step:explore} are later provided in Algorithm \ref{alg:explore}. During Step \ref{step:explore}, one sample path of DM, the nature's action and the observation outcomes are greedily selected, and then the bounds are updated using Bellman equations. Figure \ref{fig:vf} demonstrates how the lower bound of the value function can be described as the maximum of the lower bounding hyperplanes, and the upper bound can be described as a convex hull of the upper bounding points. Figure \ref{fig:vf2} illustrates an example of how newly discovered bounding hyperplanes and points can be used to locally update the bounds. 

\begin{figure}[htbp!]
\centering
\includegraphics[width=0.65\textwidth]{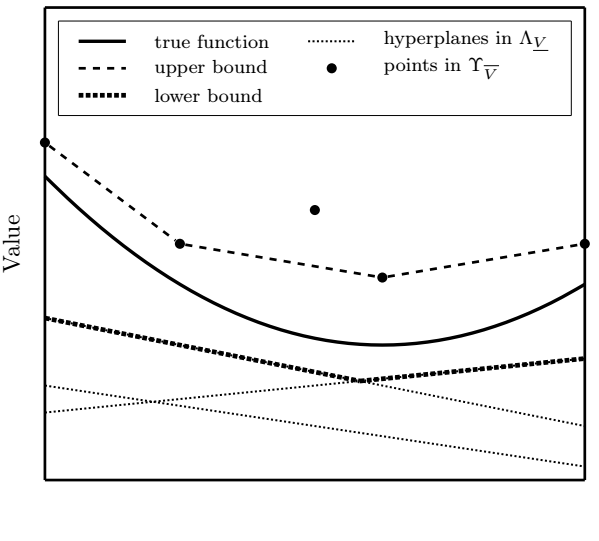}
\caption{An example of upper- and lower-bounds of a value function}
\label{fig:vf}
\end{figure}

\begin{figure}[htbp!]
\centering
\includegraphics[width=0.65\textwidth]{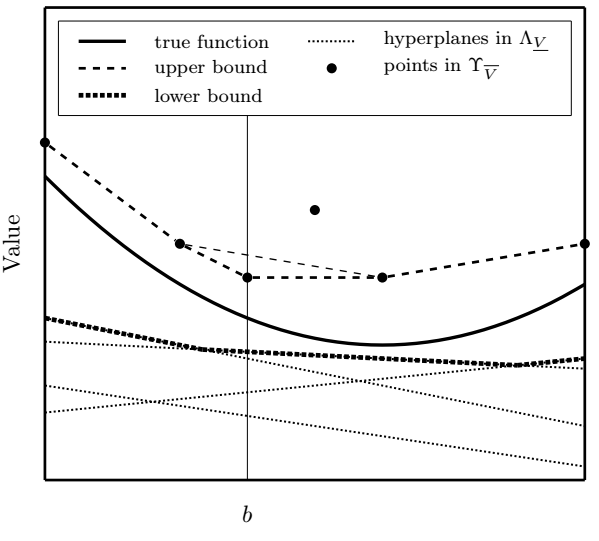}
\caption{An example of updated upper- and lower-bounds}
\label{fig:vf2}
\end{figure}

In Section \ref{subsubsec:init}, we explain how the upper and lower bounds of the value function are initialized (i.e., the details for Step \ref{step:initial}), and in Section \ref{subsubsec:heuristic}, we present an exploration strategy to close the gap to a pre-determined tolerance level. Finally, in Section \ref{subsubsec:local}, we discuss how the value functions are updated given a belief state $\bm{b}$.

\subsection{Initialization}\label{subsubsec:init}
Recall the ambiguity set and support defined in \eqref{eq:ambiguityset} and \eqref{eq:support}, respectively. In the initialization step, we compute the lower bound for the true value function by taking the best action for obtaining the worst-case expected reward in each decision period.   That is, for each action $a$, we solve
\begin{align*}
\underline{R}_a&=\sum_{t=0}^{\infty}\beta^t\min_{s\in\mathcal{S}}\min_{\mu_{as}\in\mathcal{D}_{as}}\E_{(\p_{as},r_{as})\sim\mu_{as}}\left[r_{as}\right]=\frac{1}{1-\beta}\min_{s\in\mathcal{S}}\min_{\mu_{as}\in\mathcal{D}_{as}}\E_{(\p_{as},r_{as})\sim\mu_{as}}\left[r_{as}\right].
\end{align*}
In the case of mean absolute deviation based ambiguity set \eqref{eq:ambiguityset}, the second minimization is trivial as $r_{as}$ is fixed. The minimum value for all $s\in\mathcal{S}$ is computed by enumeration. We then define an initial lower bounding hyperplane $\alpha_s'=\max_{a\in\mathcal{A}}\underline{R}_a,\ \forall s\in\mathcal{S}$ and set $\Lambda_{\underline{V}}=\left\{\bm{\alpha}'\right\}$, where $\bm{\alpha}'=\left(\alpha_s', s\in\mathcal{S}\right)^{\top}$.

The upper bound for the true value function is obtained by considering full observability of the system and computing the MDP for the best-case scenario in the ambiguity set. Let $\bm{V}^{MDP}\in\mathbb{R}^{|\mathcal{S}|}$ be a value function for the distributionally-optimistic MDP. It satisfies
\begin{align*}
V^{MDP}_s&=\max_{a\in\mathcal{A}}\max_{\mu_{as}\in\mathcal{D}_{as}}\E_{(\p_{as},r_{as})\sim\mu_{as}}\left[r_{as}+\beta \bm{V}^{MDP\top}\sum_{z\in\mathcal{Z}}{\bm{J}_z}\p_{as}\right], && \forall s\in\mathcal{S}.
\end{align*}
To solve this, we take a linear programming approach by formulating
\footnotesize 
\begin{subequations}
\label{eq:LP:initial}
\begin{align}
    \min_{\bm{V}^{MDP}}\quad&\bm{1}^{\top}\bm{V}^{MDP}\\
    \mbox{s.t.}\quad&V^{MDP}_s\geq\max_{\mu_{as}\in\mathcal{D}_{as}}\E_{(\p_{as},r_{as})\sim\mu_{as}}\left[r_{as}+\beta \bm{V}^{MDP\top}\sum_{z\in\mathcal{Z}}{\bm{J}_z}\p_{as}\right], \ \forall a\in\mathcal{A},s\in\mathcal{S}.
\end{align}
\end{subequations}\normalsize
In the case of ambiguity set \eqref{eq:ambiguityset}, model \eqref{eq:LP:initial} becomes
\begin{subequations}
\begin{align}
\min_{\bm{\rho},\bm{\kappa}, \bm{V}^{MDP}}\quad&\bm{1}^{\top}{\bm{V}^{MDP}}\\
\mbox{s.t.}\quad&{V^{MDP}_s-}\bm{c}_{as}^{\top}\bm{\rho}_{as}-{\bar{\bm{p}}_{as}^{\top}\bm{\kappa}_{as}^1+\bar{\bm{p}}_{as}^{\top}\bm{\kappa}_{as}^2-\sigma_{as}\geq r_{as}},&&\forall s\in\mathcal{S},\ a\in\mathcal{A}\\
&\beta\sum_{z\in\mathcal{Z}}{\bm{J}_z}^{\top}{\bm{V}^{MDP}-\bm{\kappa}_{as}^1+\bm{\kappa}_{as}^2-\bm{1}\sigma_{as}}\leq 0,&&\forall s\in\mathcal{S},\ a\in\mathcal{A}\\
&{\bm{\kappa}_{as}^1+\bm{\kappa}_{as}^2-\bm{\rho}_{as}=0},&&\forall s\in\mathcal{S},\ a\in\mathcal{A}\\
&{\bm{\kappa}_{as}^1,\kappa_{as}^2\in \mathbb{R}_+^{|\mathcal{S}|\times|\mathcal{A}|},\sigma_{as}\in\mathbb{R}\ \bm{\rho}_{as}\in\mathbb{R}^{|\mathcal{S}|\times|\mathcal{A}|},}&& \forall s\in\mathcal{S},\ a\in\mathcal{A}\\
&{\bm{V}^{MDP}}\in\mathbb{R}^{|\mathcal{S}|}.
\end{align}
\end{subequations}
After the optimal solution is discovered, we initialize $\Upsilon_{\overline{V}}=\left\{\left(\bm{e}_s, V^{MDP}_s\right),\ \forall s\in\mathcal{S}\right\}$, where $\bm{e}_s$ is a column vector with 1 in the element corresponding to $s$ and zero elsewhere.
Overall, the initialization step consists of solving a polynomial number of convex optimization problems. 

To obtain $\underline{V}(\bm{b})$, we solve
\begin{align*}
\max\left\{\bm{\alpha}^{\top}\bm{b}\ |\ \forall \bm{\alpha}\in\Lambda_{\underline{V}}\right\}
\end{align*}
by enumerating all the values of $\bm{\alpha}^{\top}\bm{b}$. To obtain $\overline{V}(\bm{b})$, we consider a convex combination of points $(\bm{b}^i,v^i)\in\Upsilon_{\overline{V}}$, and find a point $(\bm{b},v)$ so that $v$ is the smallest attainable value. That is, we let $w^i$ be a weight corresponding to a point  $(\bm{b}^i,v^i)$ and solve
\small
\begin{align}\label{eq:UB}
v=\min\left\{\sum_{i\in[|\Upsilon_{\overline{V}}|]}w^iv^i\ \Biggl|\ \sum_{i\in[|\Upsilon_{\overline{V}}|]}w^i\bm{b}^i=\bm{b},\ \sum_{i\in[|\Upsilon_{\overline{V}}|]}w^i=1,\ w^i\geq 0,\ \forall i\in[|\Upsilon_{\overline{V}}|]\right\},
\end{align}
\normalsize
where $[N]$ denotes the set $\left\{1,\ldots,N\right\}$ for some integer $N$.


\subsection{Forward Exploration Heuristics}
\label{subsubsec:heuristic} 
The forward heuristics follow from the HSVI algorithm from \cite{smith2004heuristic}, where the selection of a suboptimal action leads to lowering the upper bound of the value function, eventually being replaced by another action having higher upper bound. Then, the scenario of the observation is chosen such that the expected value of the gap is the highest in the child node. This process is repeated until the discounted value of the gap is smaller than a tolerance. 
The algorithmic steps described in this section are based on a greedy 
sampling strategy to close the gap between the upper and lower bounds of the value function. Samples in the simulation are branched by the DM's actions $a$, the nature's distribution choices $\mu_a$, and their outcomes $z$ and $\p_a$. 

We consider the following function: 
\begin{align*}
U_{V}(\bm{b},a,\mu_{a})&=\E_{(\p_{a},\rb_a)\sim\mu_{a}}\Biggl[\sum_{s\in\mathcal{S}}b_s\Biggl\{r_{as}+\beta\sum_{z\in\mathcal{Z}}\bm{1}^{\top}{\bm{J}_z}\p_{as}V\left(\bm{f}\left(\bm{b},a,\p_a,z\right)\right)\Biggr\}\Biggr].
\end{align*}
We can obtain $U_{\overline{V}}$ and $U_{\underline{V}}$ by letting $V = \overline{V}$ and $V = \underline{V}$, respectively. 

First, we select the DM and nature's decision pair $(a^*,\mu_{a^*}^*)$. The gap between $U_{\overline{V}}$ and $U_{\underline{V}}$ at belief state $\bm{b}$ is
\begin{small}
\begin{align}
    U_{\overline{V}}&(\bm{b},a^{*},\mu_{a^*}^*)-U_{\underline{V}}(\bm{b},a^{*},\mu_{a^*}^*)\nonumber\\
    =\quad&\E_{(\p_{a^*},\rb_{a^*})\sim\mu_{a^*}}\left[\sum_{s\in\mathcal{S}}b_s\left(r_{a^*s}+\beta\sum_{z\in\mathcal{Z}}\bm{1}^{\top}{\bm{J}_z}\p_{a^*s}\overline{V}\left(\bm{f}(\bm{b},a^*,\p_{a^*},z)\right)\right)\right]\nonumber\\
    & \hspace{20pt} - \E_{(\p_{a^*},\rb_{a^*})\sim\mu_{a^*}}\left[\sum_{s\in\mathcal{S}}b_s\left(r_{a^*s}+\beta\sum_{z\in\mathcal{Z}}\bm{1}^{\top}{\bm{J}_z}\p_{a^*s}\underline{V}\left(\bm{f}(\bm{b},a^*,\p_{a^*},z)\right)\right)\right]\nonumber\\
    =\quad&\beta\E_{(\p_{a^*},\rb_{a^*})\sim\mu_{a^*}^*}\left[\sum_{s\in\mathcal{S}}b_s\sum_{z\in\mathcal{Z}}\bm{1}^{\top}{\bm{J}_z}\p_{a^*s}\left(\overline{V}\left(\bm{f}(\bm{b},a^*,\p_{a^*},z)\right)-\underline{V}\left(\bm{f}(\bm{b},a^*,\p_{a^*},z)\right)\right)\right].\label{eq:u-bound}
\end{align}
\end{small}
Here we describe a greedy strategy to select the branches. For a given action $a$, we define $\mu_a^*=\mbox{argmin}_{\mu_{a}\in\tilde{\mathcal{D}}_{a}}U_{\underline{V}}(\bm{b},a,\mu_{a})$. Then, we let $a^* = \mbox{argmax}_{a\in\mathcal{A}}U_{\overline{V}}(\bm{b},a,\mu_{a}^*)$. We therefore have
\begin{align}
    \overline{V}(\bm{b})-\underline{V}(\bm{b}) &= \max_{a\in\mathcal{A}}\min_{\mu_a\in\tilde{\mathcal{D}}_a}U_{\overline{V}}(\bm{b},a,\mu_a)-\max_{a\in\mathcal{A}}\min_{\mu_a\in\tilde{\mathcal{D}}_a}U_{\underline{V}}(\bm{b},a,\mu_a)\nonumber\\
    &\leq\max_{a\in\mathcal{A}}U_{\overline{V}}(\bm{b},a,\mu_a^*)-\max_{a\in\mathcal{A}}U_{\underline{V}}(\bm{b},a,\mu_a^*)\nonumber\\
    &\leq U_{\overline{V}}(\bm{b},a^{*},\mu_{a^*}^*)-U_{\underline{V}}(\bm{b},a^{*},\mu_{a^*}^*).\label{eq:v-u}
\end{align}
This greedy strategy ensures that a suboptimal decision pair $(a^*,\mu_{a^*}^*)$ gets replaced by better ones as updating the value functions reduces the gap.

To achieve the gap $\epsilon$ at the initial state $\bm{b}_0$, the condition for the gap at depth level $t$ starting from the initial one is only $\epsilon\beta^{-t}$, which can readily be seen from \eqref{eq:u-bound} and \eqref{eq:v-u}. We define the difference of the gap and the required condition as the excess uncertainty, which is
\begin{align*}
\mbox{excess}(\bm{b},t)=\overline{V}\left(\bm{b}\right)-\underline{V}\left(\bm{b}\right)-\epsilon\beta^{-t}.
\end{align*}
Using \eqref{eq:v-u} and applying the identity \eqref{eq:u-bound}, we have
\small
\begin{align}\label{eq:excess}
\mbox{excess}(\bm{b},t) \leq\beta\E_{(\p_{a^*},\rb_{a^*})\sim\mu_{a^*}^*}\left[\sum_{s\in\mathcal{S}}b_s\sum_{z\in\mathcal{Z}}\bm{1}^{\top}{\bm{J}_z}\p_{a^*s}\mbox{excess}(\bm{f}(\bm{b},a^*,\p_{a^*},z),t+1)\right].
\end{align}
\normalsize
Next, we greedily choose $(z^*, p_{a^*}^*)$ so that the quantity associated to the pair in right-hand side of \eqref{eq:excess} has the maximum expected value, i.e., 
\small
\begin{align}
(z^*,\p_{a^*}^*) \in\argmax_{z\in\mathcal{Z},\ \p_{a^*}\in\mathcal{X}_{a^*}}\mu_{a^*}^*(\p_{a^*})\times\sum_{s\in\mathcal{S}}b_s\bm{1}^{\top}{\bm{J}_z}\p_{a^*s}^*\mbox{excess}(\bm{f}(\bm{b},a^*,\p_{a^*},z),t+1).
\end{align}
\normalsize
Note that because the worst-case distribution under ambiguity set \eqref{eq:ambiguityset} is a point mass distribution, obtaining $\p_{a^*}^*$ is trivial. Algorithm \ref{alg:explore} describes the detailed algorithmic steps. In the HSVI approach, Algorithm \ref{alg:explore} is called recursively to make decisions on which branch to choose in the next depth level $t+1$. After the simulation is terminated, the updates on the lower and upper bounds are made for the belief states that are discovered through the simulation.

\begin{algorithm}
\caption{DR-BoundExplore$(\bm{b},t)$}
\label{alg:explore}
\begin{algorithmic}[1]
	\STATE {\bf Input:} belief state $\bm{b}$, depth level $t$
	\IF{$\overline{V}(\bm{b})-\underline{V}(\bm{b})>\epsilon\beta^{-t}$}
	    \STATE ($\mu_{a}^*,\ \forall a \in \mathcal{A})\gets\mbox{argmin}_{\mu_{a}\in\mathcal{D}_{a}}U_{\underline{V}}(\bm{b},a,\mu_{a})$
		\STATE $a^*\gets\mbox{argmax}_{a\in\mathcal{A}}U_{\overline{V}}(\bm{b},a,\mu_{a}^*)$
		
		\STATE $z^*,\p_{a^*}^*$ $\gets\mbox{argmax}_{z\in\mathcal{Z},\ \p_{a^*}\in\mathcal{X}_{a^*}}\mu_{a^*}^*(\p_{a^*})\times\sum_{s\in\mathcal{S}}b_s\bm{1}^{\top}{\bm{J}_z}\p_{a^*s}^*\times\mbox{excess}(\bm{f}(\bm{b},a^*,\p_{a^*},z),t+1)$
		\STATE $DR\mbox{-}BoundExplore(\bm{f}(\bm{b},a^*,\p_{a^*}^*,z^*),t+1)$
		\STATE $\Lambda_{\underline{V}}\gets \Lambda_{\underline{V}}\cup DR\mbox{-}backup(\bm{b},\Lambda_{\underline{V}})$ (see the details in Algorithm \ref{alg:backup})
		\STATE $\Upsilon_{\overline{V}}\gets \Upsilon_{\overline{V}}\cup DR\mbox{-}update(\bm{b},\Upsilon_{\overline{V}})$ (see the details in Algorithm \ref{alg:update})	
	\ENDIF
\end{algorithmic}
\end{algorithm}

\subsection{Local Updates}
\label{subsubsec:local}
In this section, we describe the details of $DR\mbox{-}backup$ and $DR\mbox{-}update$ steps in Algorithm \ref{alg:explore}. We first illustrate how the lower bound is updated in $DR\mbox{-}backup$. For each $a\in\mathcal{A}$, we solve the two inner maximization problems in \eqref{eq:maxproblem} provided $a$ and $\bm{b}$, where we set $\Lambda^{t+1} = \Lambda_{\underline{V}}$. The convex hull of $\Lambda_{\underline{V}}$ is therefore,
\begin{align}\label{eq:convLambda}
\mbox{Conv}\left(\Lambda_{\underline{V}}\right)=\left\{\sum_{i\in[|\Lambda_{\underline{V}}|]}w^i\bm{\alpha}^{i}\Biggr|\ \sum_{i\in[|\Lambda_{\underline{V}}|]}w^i=1,\ \bm{\alpha}^{i}\in\Lambda_{\underline{V}},\ w^i\geq0,\ i\in[|\Lambda_{\underline{V}}|]\right\}.
\end{align}
Thus, we combine the two inner maximization problems in \eqref{eq:maxproblem} as
\footnotesize
\begin{subequations}\label{eq:LBprob}
\begin{align}
\max_{\bm{\rho}_a, \bm{\kappa}_{a}^1,\bm{\kappa}_{a}^2,\bm{\sigma}_{a}} \quad&{\sum_{s\in\mathcal{S}}\bm{c}_{as}^{\top}\bm{\rho}_{as}+\sum_{s\in\mathcal{S}}b_s r_{as}+\sum_{s\in\mathcal{S}}\left(-\bar{p}_{as}^\top\bm{\kappa}_{as}^1+\bar{p}_{as}^\top\bm{\kappa}_{as}^2+\sigma_{as}\right)}\\
\label{eq:LBphat}\mbox{s.t.}\quad&{\beta b_{s}\sum_{z\in\mathcal{Z}}\sum_{i\in[|\Lambda_{\underline{V}}|]}w_{az}^i{\bm{J}_z}^{\top}\bm{\alpha}_{az}^i+\bm{\kappa}_{as}^1-\bm{\kappa}_{as}^2-\bm{1}\sigma_{as}\geq0},&&\forall s\in\mathcal{S}\\
&\sum_{i\in[|\Lambda_{\underline{V}}|]}w_{az}^i=1,&&\forall z\in\mathcal{Z}\\
&w_{az}^i\in\mathbb{R}_+, &&\forall i\in[|\Lambda_{\underline{V}}|],\ z\in\mathcal{Z}\\
&\eqref{eq:dualut},\eqref{eq:dualvar},\eqref{eq:dualrho}.\nonumber
\end{align}
\end{subequations}
\normalsize 
We denote the optimal solutions to \eqref{eq:LBprob} using a superscript $\star$, and let the optimal dual solutions associated with constraints \eqref{eq:LBphat} 
be $\hat{\p}_{as}^{\star}$. 
For each action $a\in\mathcal{A}$, we can generate a lower bounding hyperplane
\begin{align}\label{eq:newalpha}
\bm{\alpha}'=\left(r_{as}+\beta\sum_{z\in\mathcal{Z}}\bm{\alpha}_{az}^{\star\top}{\bm{J}_z}\hat{\p}_{as}^{\star},\ s\in\mathcal{S}
\right)^{\top},
\end{align}
where $\bm{\alpha}_{az}^{\star}=\sum_{i\in[N]}w_{az}^{i\star}\bm{\alpha}_{az}^i$.
We present the detailed algorithmic steps in Algorithm~\ref{alg:backup}. 

\begin{algorithm}
\caption{DR-backup$(\bm{b},\Lambda_{\underline{V}})$}
\label{alg:backup}
\begin{algorithmic}[1]
	\STATE {\bf Input:} belief $\bm{b}$, lower bounding hyperplanes $\Lambda_{\underline{V}}$
			\FOR{$\forall a\in\mathcal{A}$}
			\STATE solve \eqref{eq:LBprob} for action $a$
			\STATE $\mathcal{L}(a)\gets\bm{\alpha}'$ (calculated using \eqref{eq:newalpha})
		\ENDFOR
		\STATE {\bf Output:} $\mbox{argmax}_{\bm{\alpha}\in\mathcal{L}}\bm{\alpha}^{\top}\bm{b}$
\end{algorithmic}
\end{algorithm}

Next, we discuss how to update the upper bound and describe the algorithmic steps of $DR\mbox{-}update$ in Algorithm \ref{alg:update}. Combining \eqref{eq:maxproblem} and the dual representation of \eqref{eq:UB}, for each $a\in\mathcal{A}$, we solve 
\footnotesize
\begin{subequations}\label{eq:UBprob}
\begin{align}
{\max_{\bm{\rho}_a, \bm{\kappa}_{a}^1,\bm{\kappa}_{a}^2,\bm{\sigma}_{a}}}\quad&{\sum_{s\in\mathcal{S}}\bm{c}_{as}^{\top}\bm{\rho}_{as}+\sum_{s\in\mathcal{S}}b_s r_{as}+\sum_{s\in\mathcal{S}}\left(-\bar{p}_{as}^\top\bm{\kappa}_{as}^1+\bar{p}_{as}^\top\bm{\kappa}_{as}^2+\sigma_{as}\right)}\\
\label{eq:UBphat}\mbox{s.t.}\quad&\beta {b_{s}\sum_{z\in\mathcal{Z}}{\bm{J}_z}^{\top}\varphi_{az}+\beta b_s\sum_{z\in\mathcal{Z}}\psi_{az}{\bm{J}_z}^{\top}\bm{1}+\bm{\kappa}_{as}^1-\bm{\kappa}_{as}^2-\bm{1}\sigma_{as}\geq0},&&\forall s\in\mathcal{S}\\
&\bm{b}^{i\top}\varphi_{az}+\psi_{az}\leq v_i, &&\forall z\in\mathcal{Z}, i\in[|\Upsilon_{\overline{V}}|]\\
&\varphi_{az}\in\mathbb{R}^{|\mathcal{S}|},\ \psi_{az}\in\mathbb{R}, &&\forall z\in\mathcal{Z}, i\in[|\Upsilon_{\overline{V}}|]\\
&\eqref{eq:dualut},\eqref{eq:dualvar},\eqref{eq:dualrho}.\nonumber
\end{align}
\end{subequations}
\normalsize
Here $\varphi_{az}$ and $\psi_{az}$ are the dual variables associated with the two sets of constraints, $\sum_{i\in[|\Upsilon_{\overline{V}}|]}w^i\bm{b}^i=\bm{b}$, $\sum_{i\in[|\Upsilon_{\overline{V}}|]}w^i=1$, respectively. The maximum objective value among all $a\in\mathcal{A}$ is added to $\Upsilon_{\overline{V}}$.

\begin{algorithm}
\caption{DR-update$(\bm{b},\Upsilon_{\overline{V}})$}
\label{alg:update}
\begin{algorithmic}[1]
	\STATE {\bf Input:} belief $\bm{b}$, upper bounding points $\Upsilon_{\overline{V}}$
		\FOR{$\forall a\in\mathcal{A}$}
		\STATE $\mathcal{Q}(a)\gets$(optimal objective value of \eqref{eq:UBprob} for action $a$)
		\ENDFOR
\STATE {\bf Output:} $(\bm{b},\max_{a \in \mathcal A}\{\mathcal{Q}(a)\})$
\end{algorithmic}
\end{algorithm}

\begin{remark}
The complexity of the related algorithm presented in \cite{smith2004heuristic} is based on the finiteness of the scenario tree up to a tolerance level $\epsilon$. In the DR-HSVI algorithm, the scenario tree is not finite as the nature is able to choose from a continuous ambiguity set of distributions, and therefore the scenario tree has an infinite number of elements. Later we numerically demonstrate the convergence of the DR-HSVI algorithm in Section \ref{sec:comp} for different combinations of parameter choices. 
\end{remark}

\section{Numerical Studies}
\label{sec:comp}
We test DR-POMDP policies for dynamic epidemic control (Sections \ref{subsec:setup} and \ref{subsec:dynamic}), and compare the  results of a two-state epidemic control problem with the ones given by POMDP and robust POMDP (Section \ref{subsec:policy}). We vary parameter choices to test the robustness and sensitivity of DR-POMDP policies (i) under various types of ambiguity sets used in the in-sample tests (Sections \ref{sec:varyingASsize}, \ref{sec:varyingASnum}) and (ii) given certain noise added to the transition-observation probability value obtained at the end of each decision period in out-of-sample tests (Sections \ref{sec:sensitivity1}, \ref{sec:sensitivity2}).  In Sections \ref{sec:state-time} and \ref{sec:uncertainty-time}, we increase the sizes of the two-state influenza epidemic control instances in Section \ref{subsec:setup}, demonstrate the algorithmic convergence, and present computational time results of using POMDP and DR-POMDP for solving larger-scale epidemic control instances. 

\subsection{Two-state Influenza Epidemic Control Problem}
\label{subsec:setup}
We study the problem of influenza epidemic control mentioned in Section \ref{sec:Example}. 
In the base setting, we consider two states, epidemic (E) and non-epidemic (N), and four actions as $a \in $ \{Level 0, Level 1, Level 2, Inspection\}. Here Level 0 corresponds to the minimum disease prevention and intervention plan, e.g., doing nothing, while Level 2 corresponds to the most restrictive strategy. The ``Inspection'' action refers to the same disease-control strategy as the Level 0 action, except that the DM pays extra cost to improve the observation of disease spread to obtain more accurate ILI rate. 

For actions $a\in\{0,1,2\}$, the transition probability matrix is given by
\begin{align}\label{eq:TRepidemic}
    \begin{pmatrix}
        0.99 - 0.1a & 0.01 + 0.1a\\
        0.3 - 0.1a  & 0.7 + 0.1a
    \end{pmatrix}.
\end{align}
When $a=0$ (i.e., the DM does nothing), the above transition probabilities follow studies on influenza epidemics (see, e.g., \cite{le1999monitoring}). The setting of the matrix \eqref{eq:TRepidemic} indicates that higher-level actions (i.e., more restrictive control strategies) will lead to greater chances that an epidemic state turns into non-epidemic and that a non-epidemic state remains itself. The transition probability for $a=$ `Inspection' (`I') is the same as the one for $a=0$. The observation outcome is the ILI rate, calculated as the number of ILI patients per 1000 population. For actions $a\in\{0,1,2\}$, we follow  \cite{rath2003automated} and assume that the ILI rate follows a Gaussian distribution with mean value $\mu_{E} = 2-0.5a$ and variance $\mbox{Var}_E = 30 - \mu_{E}^2$ for $s=$ `Epidemic' (`E'), and with mean $\mu_{N} = 0.2-0.05a$ and variance $\mbox{Var}_N = 2 - \mu_{N}^2$ for $s=$ `Non-epidemic' (`N'). 
We discretize the observation outcome into five levels as $\{(-\infty,0],(0,1/3],(10/3,20/3],(20/3,10],(10,\infty)\}$. For $a=$ `I',  the probabilities of observing the five outcomes are $\{0.01,0.1/3,0.1/3,0.1/3,0.89\}$ when $s=$ `E', and the ILI rate follows the same distribution as the one of $a=0$ if $s=$ `N', to model the situation where more careful inspection action can result in more ILI patients showing up. The rewards for each action-state combination are presented in Table \ref{tab:rewards}, reflecting the negative number of total infections minus the effort paid for different actions in different states. 

\begin{table}[htbp]
\centering
\caption{Reward setting for each state-action pair}
\label{tab:rewards}
\begin{tabular}{c|r r r r} 
State/Action & Level 0 & Level 1 & Level 2 & Inspection \\ 
 \hline
Epidemic & $-100$ & $-50$ & $-25$ & $-110$\\
Non-epidemic & $0$ & $-20$ & $-40$ & $-20$\\ 
\hline
\end{tabular}
\end{table}

When implementing the HSVI algorithm in Section \ref{sec:approx2} for solving DR-POMDP, we set the discount factor $\beta=0.95$ and the gap tolerance $\epsilon=1.0$. The computation is terminated when the gap between the upper and lower bounds is less than $\epsilon$, at the initial states $b^0_E=0.5,\ b^0_N=0.5$. We code the algorithm in Python and execute all the tests on a computer with Intel Core i5 CPU running at 2.9 GHz and 8 GB of RAM. We solve all the linear programming models using the Gurobi solver. Note that the complexity of computing the lower bound is linear in the number of elements in $\Lambda_{\underline{V}}$, and the complexity of computing the upper bound is polynomial in the size of set $\Upsilon_{\overline{V}}$ as we need to solve linear programs. Both $|\Lambda_{\underline{V}}|$ and $|\Upsilon_{\overline{V}}|$ increase monotonically, but most elements in the two sets are dominated by others. We follow a heuristic to prune all the dominated elements whenever the number of elements increases by 10\%.

\subsubsection{Policy Comparison}
\label{subsec:policy}

We compare DR-POMDP policies with the ones by POMDP and robust POMDP via cross testing.  We randomly generate ten samples of the transition probability for Level 2 action (i.e., $a=2$) and epidemic state (i.e., $s=$ `E'), by keeping all the values the same as the base setting in \eqref{eq:TRepidemic} but letting the probability $p_{2}(N|E)=0.99-0.1 \times 2+0.1 \times x$, where $x$ follows a standard Normal distribution. (We make sure that $0 \leq p_{2}(N|E) \leq 1$ and re-sample if not.) For all three approaches, the mean value of the ten samples is used as the nominal transition probability. For robust POMDP, the maximum L1 norm from the mean defines an uncertainty set centered around the nominal probability. For DR-POMDP, we use the  mean absolute deviation to define the ambiguity set.

\begin{table}[htbp]
\centering
\caption{Estimated median values of the cross-tested rewards}
\label{tab:median}
\begin{tabular}{c|rrr}
\hline
& \multicolumn{3}{c}{Nature's policy}                                                                              \\ \hline
DM's policy                & \multicolumn{1}{c}{POMDP(std)} & \multicolumn{1}{c}{DR-POMDP(std)} &  \multicolumn{1}{c}{Robust(std)} \\ \hline
POMDP           & $\mathbf{-541.22}\ (1.08)$ &	$-609.63\ (0.93)$	& $-597.06\ (2.19)$	\\
DR-POMDP        & $-559.02\ (0.95)$ &	$-589.93\ (0.92)$	& $\mathbf{-594.30}\ (1.31)$ \\
Robust          & $-570.16\ (1.44)$ &	$\mathbf{-585.99}\ (1.22)$   & $-597.75\ (1.18)$	\\
\hline                                    
\end{tabular}
\end{table}

\begin{table}[htbp]
\centering
\caption{Estimated five-percentile values of the cross-tested rewards}
\label{tab:5perc}
\begin{tabular}{c|rrr}
\hline
& \multicolumn{3}{c}{Nature's policy}                                                                              \\ \hline
DM's policy             & \multicolumn{1}{c}{POMDP(std)} & \multicolumn{1}{c}{DR-POMDP(std)} &  \multicolumn{1}{c}{Robust(std)} \\ \hline
POMDP           & $\mathbf{-656.99}\ (2.39)$ &	$-696.34\ (1.34)$   & $-711.14\ (1.43)$	\\
DR-POMDP        & $-669.26\ (2.35)$ &   $\mathbf{-677.87}\ (1.95)$   & $-705.61\ (1.60)$	\\
Robust          & $-689.26\ (1.78)$ &	$-691.77\ (2.07)$	& $\mathbf{-698.93}\ (2.19)$	\\
\hline                                 
\end{tabular}
\end{table}

We implement the DM's optimal polices given by different approaches in out-of-sample environments where the nature follows the settings of POMDP, DR-POMDP, and robust POMDP to realize the transition probabilities in each period. 
The number of simulated instances is 5000 each.
We report the estimated value of the median and the 5-percentile values of the reward in each case in Tables \ref{tab:median} and \ref{tab:5perc}, respectively using Harrell-Davis quantile estimator \cite{harrell1982new}. We also include the standard deviation of the estimator. Note that the 5-percentile of the reward is equivalent to the 95-percentile of the cost, indicating the tail (worse) performance of different policies. Therefore, Tables \ref{tab:median} and \ref{tab:5perc} indicate that POMDP has the smallest reward when the nature agrees with the DM to pick the nominal transition probabilities at each decision period, but it can lead to much worse reward (both in terms of the mean value and tail performance) if the transition probabilities are realized as the worst-case (in robust POMDP) or from the worst-case distribution (in DR-POMDP). On the other hand, the performance of DR-POMDP solutions is quite stable and robust under all out-of-sample circumstances but the tail performance is worse than the mean results. Lastly, the robust POMDP policy yields worse mean value and tail performance when the true environment is POMDP or DR-POMDP.

\subsubsection{Results of Varying Ambiguity Set Sizes}
\label{sec:varyingASsize}
We first only consider an ambiguity in the transition-observation probabilities of Level 0 action and epidemic state. We build the ambiguity set based on the mean absolute deviation such that $\E_{\p_{as}\sim\mu_{as}}\left[|\p_{as}-\bar{\p}_{as}|\right]\leq\bm{c}_{as}$ for $a=0$ and $s=$ `E', where $\bar{\p}_{as}\in\Delta(\mathcal{S}\times\mathcal{Z})$ is the mean value of given probability samples and $\bm{c}_{as}\in\mathbb{R}^{|\mathcal{S}\times\mathcal{Z}|}$. We let $\bm{c}_{as}$ be $c\cdot\bm{1}$ for some $c\in\mathbb{R}$ and vary the values of $c$ in our tests to vary the size of the ambiguity set. 

\begin{figure}[htbp]
\centering
\begin{subfigure}[b]{0.49\textwidth}
   \includegraphics[width=\textwidth]{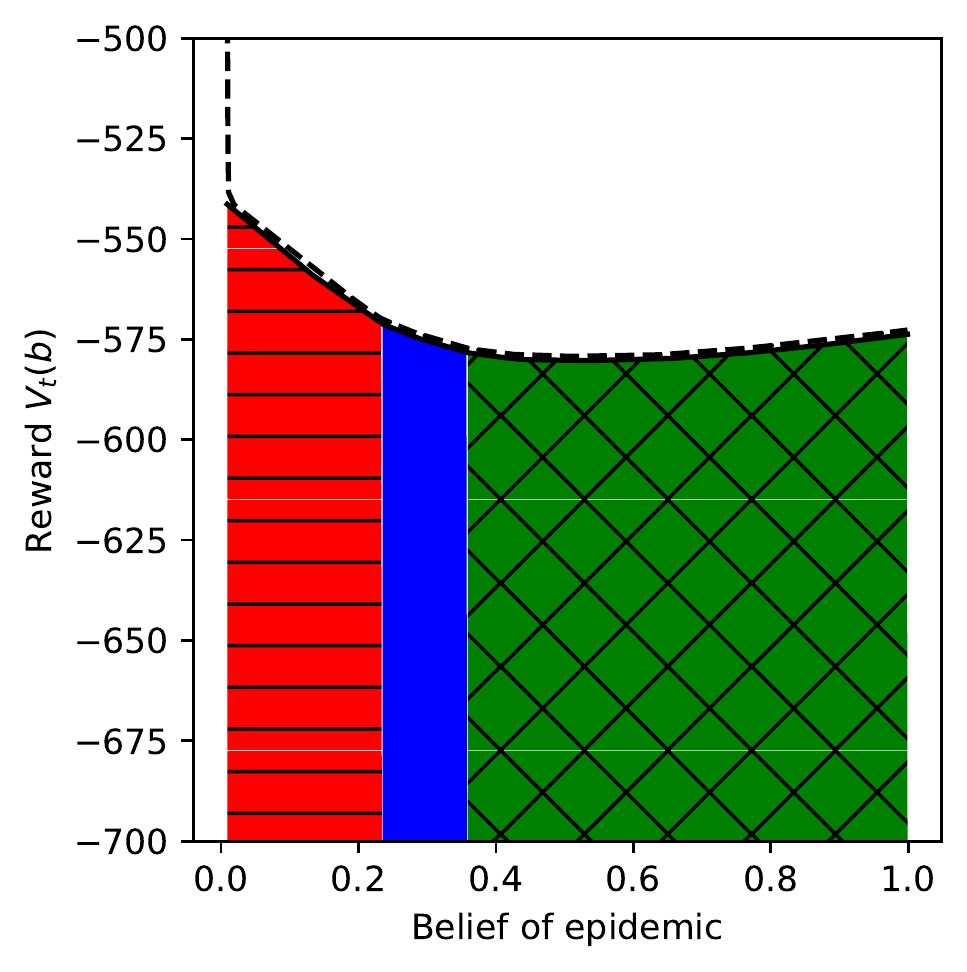}
   \caption{POMDP ($c=0.00$)}
   \label{fig:c0} 
\end{subfigure}
\hfill
\begin{subfigure}[b]{0.49\textwidth}
   \includegraphics[width=\textwidth]{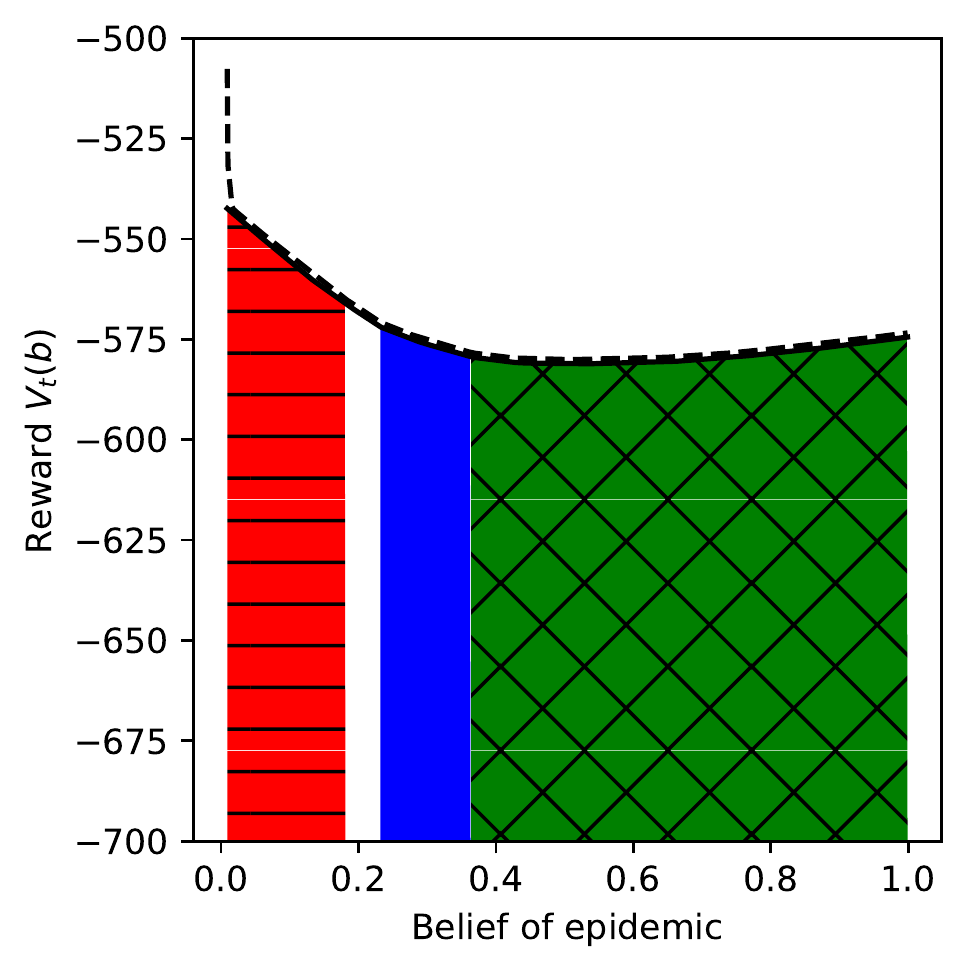}
   \caption{DR-POMDP ($c=0.03$)}
   \label{fig:c3}
\end{subfigure}
\vskip\baselineskip
\begin{subfigure}[b]{0.49\textwidth}
   \includegraphics[width=\textwidth]{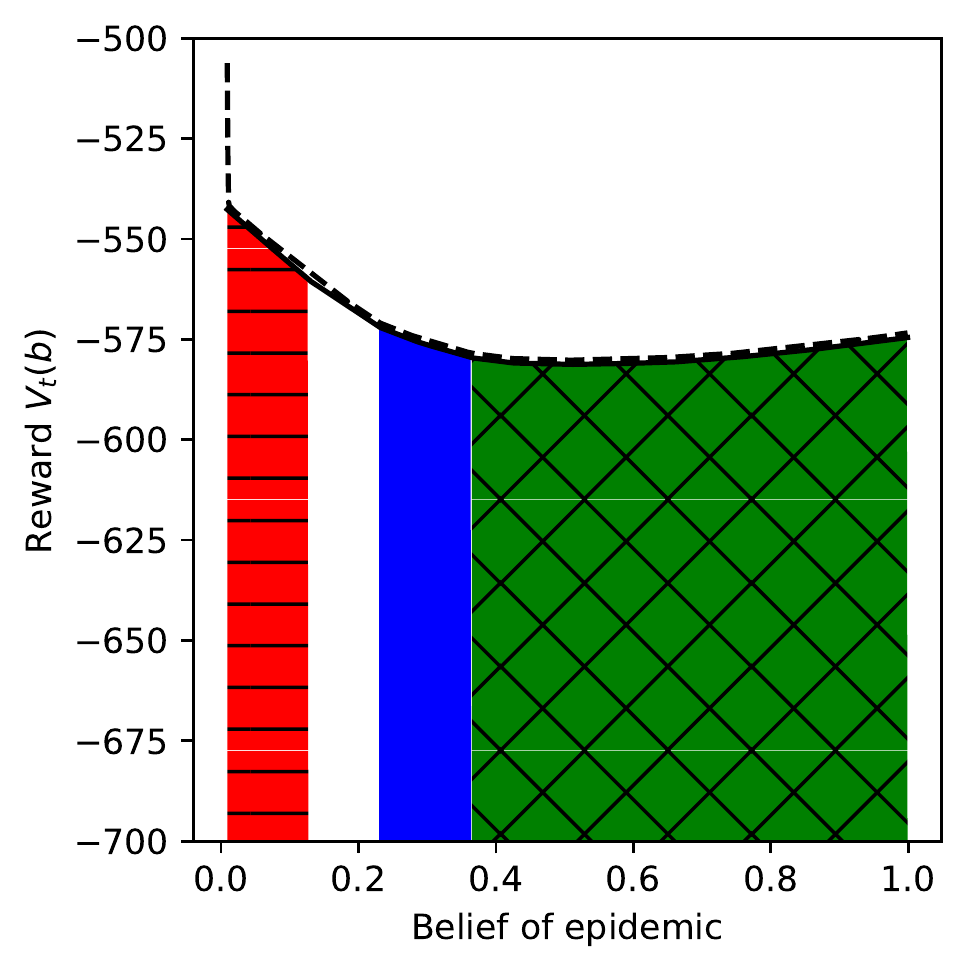}
   \caption{DR-POMDP ($c=0.06$)}
   \label{fig:c6}
\end{subfigure}
\hfill
\begin{subfigure}[b]{0.49\textwidth}
   \includegraphics[width=\textwidth]{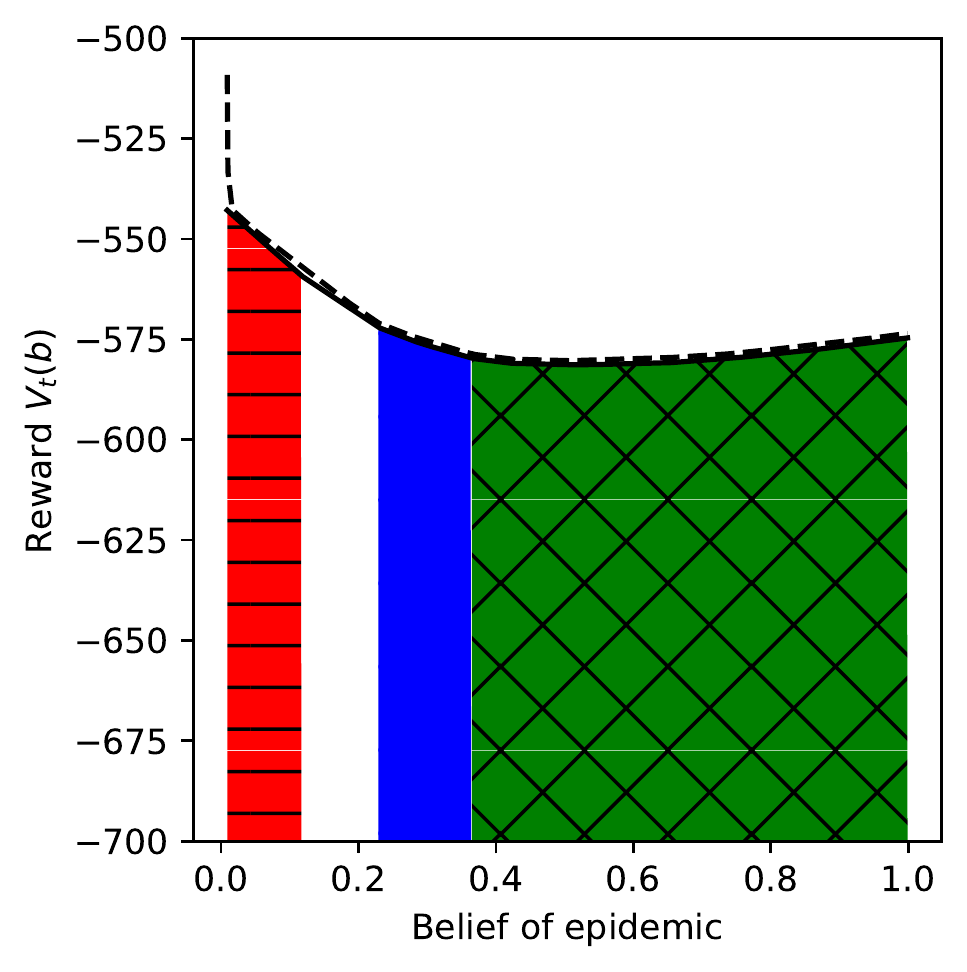}
   \caption{DR-POMDP ($c=0.09$)}
   \label{fig:c9}
\end{subfigure}
\caption{Value functions for different ambiguity-set sizes. Solid line: lower bound, dashed line: upper bound. Corresponding actions:  Level 0 -- (red, horizontal), Level 1 -- (blue, dot), Level 2 -- (green, cross),  Inspection -- (white, diagonal)}
\label{fig:AmbiguitySize}
\end{figure}

We vary $c=0.03, 0.06, 0.09$ for DR-POMDP and also compute the POMDP policy using $\bar{\p}_{as}$ as the transition-observation probabilities for all $a$ and $s$, which corresponds to a special case of DR-POMDP with $c=0.00$. Figure \ref{fig:AmbiguitySize} depicts the upper bound (dashed line) and the lower bound (solid line) of the value functions of POMDP and DR-POMDP, as well as optimal actions corresponding to different beliefs of the epidemic. The region of the belief in red (horizontal shade) corresponds to Level 0 action, blue (dotted shade) to Level 1 action, green (cross shade) to Level 2 action, and white (diagonal shade) to Inspection action. Because the ambiguity is in the transition-observation probabilities related to $a=0$, in all the subfigures, as compared to POMDP, the DR-POMDP policy relies less on Level 0 action and replaces it with the `Inspection' action when the belief of epidemic is relatively higher. When the belief increases further, both DR-POMDP and POMDP agree on implementing Level 1 or Level 2 action. As the ambiguity set size increases (i.e., $c$ increases), the DR-POMDP policy becomes more conservative and shifts to the `Inspection' action earlier, even in relatively low belief of epidemic.

\subsubsection{Results of Multiple Ambiguities}
\label{sec:varyingASnum}

Next, we increase the number of action-state pairs that have distributional ambiguity in the transition-observation probabilities. We use $c=0.05$ for all ambiguity sets and vary the number of action-state pairs among $\{2,3,4,5\}$. In Figure \ref{fig:num2}, action-state pairs (Level 0, E) and (Level 0, N) have ambiguous probability distributions and then we add pairs (Level 1, E), (Level 1, N), and (Level 2, E) one by one in the subsequent Figures \ref{fig:num3}, \ref{fig:num4}, \ref{fig:num5}. 

\begin{figure}[htbp]
\centering
\begin{subfigure}[b]{0.49\textwidth}
   \includegraphics[width=\textwidth]{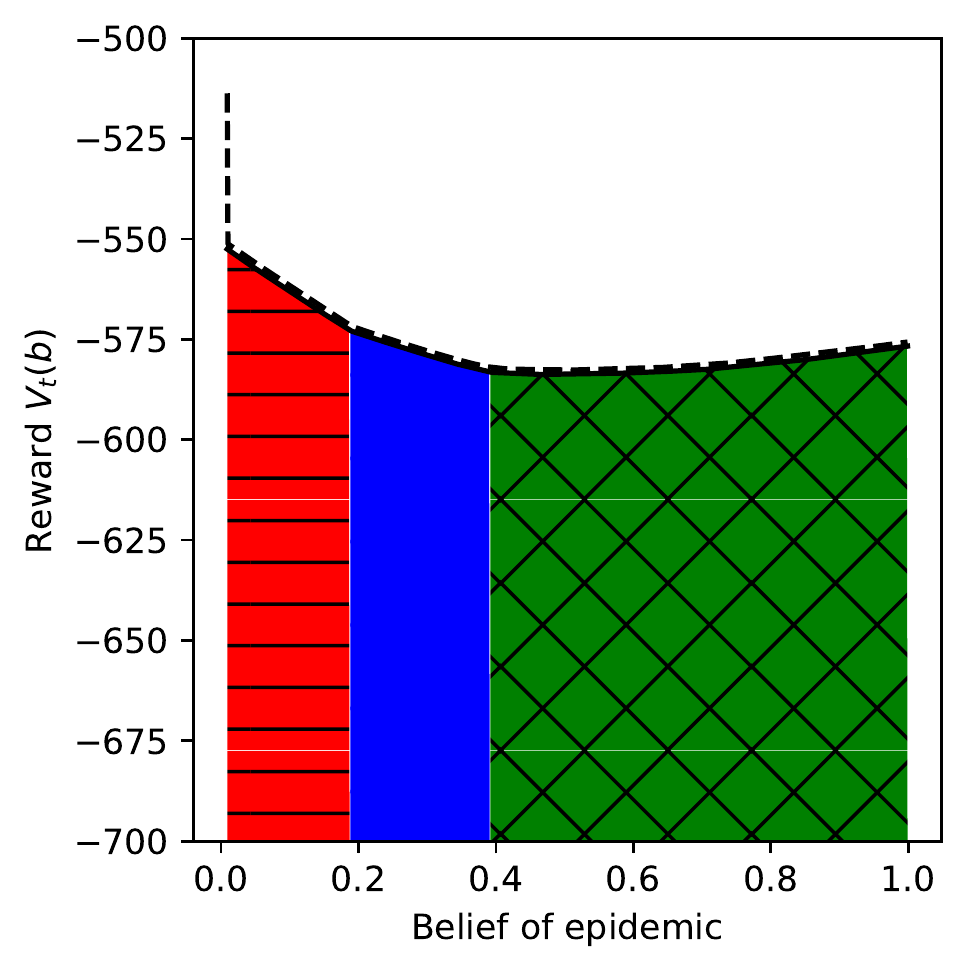}
   \caption{\{(Level 0, E), (Level 0, N)\}}
   \label{fig:num2} 
\end{subfigure}
\hfill
\begin{subfigure}[b]{0.49\textwidth}
   \includegraphics[width=\textwidth]{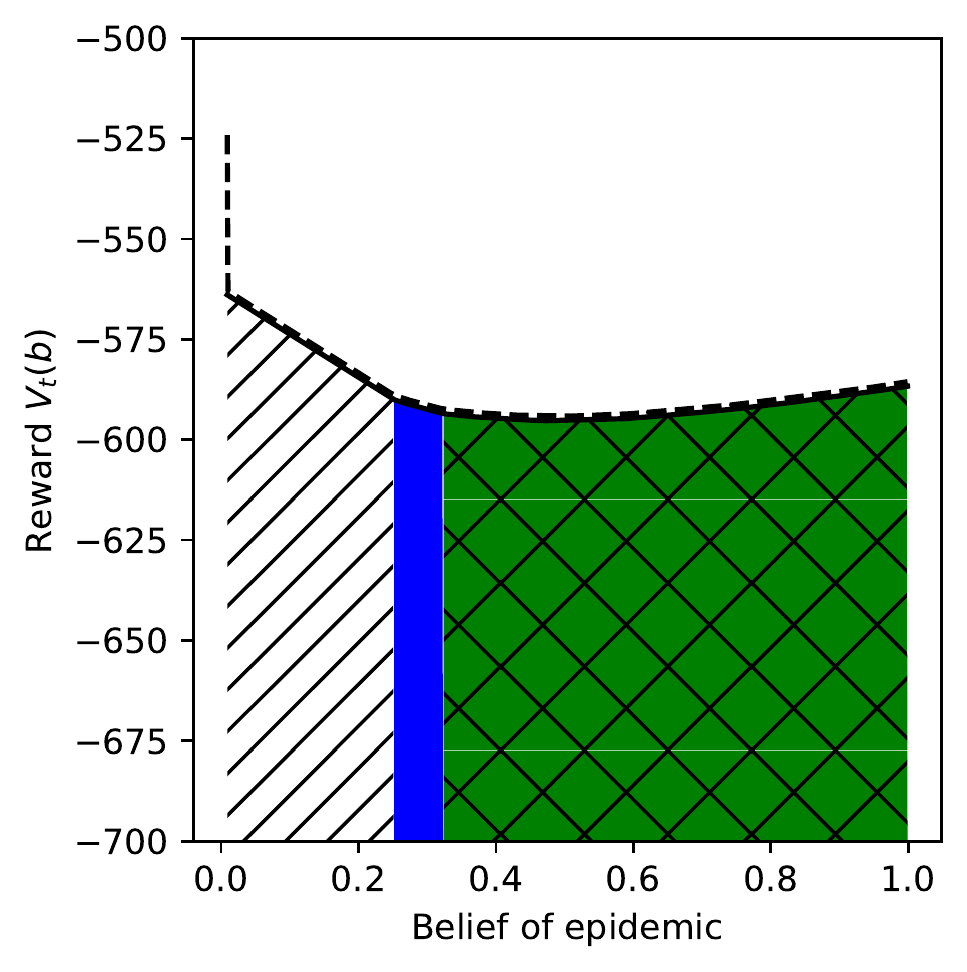}
   \caption{\{(Level 0, E), (Level 0, N), (Level 1, E)\}}
   \label{fig:num3}
\end{subfigure}
\vskip\baselineskip
\begin{subfigure}[b]{0.49\textwidth}
   \includegraphics[width=\textwidth]{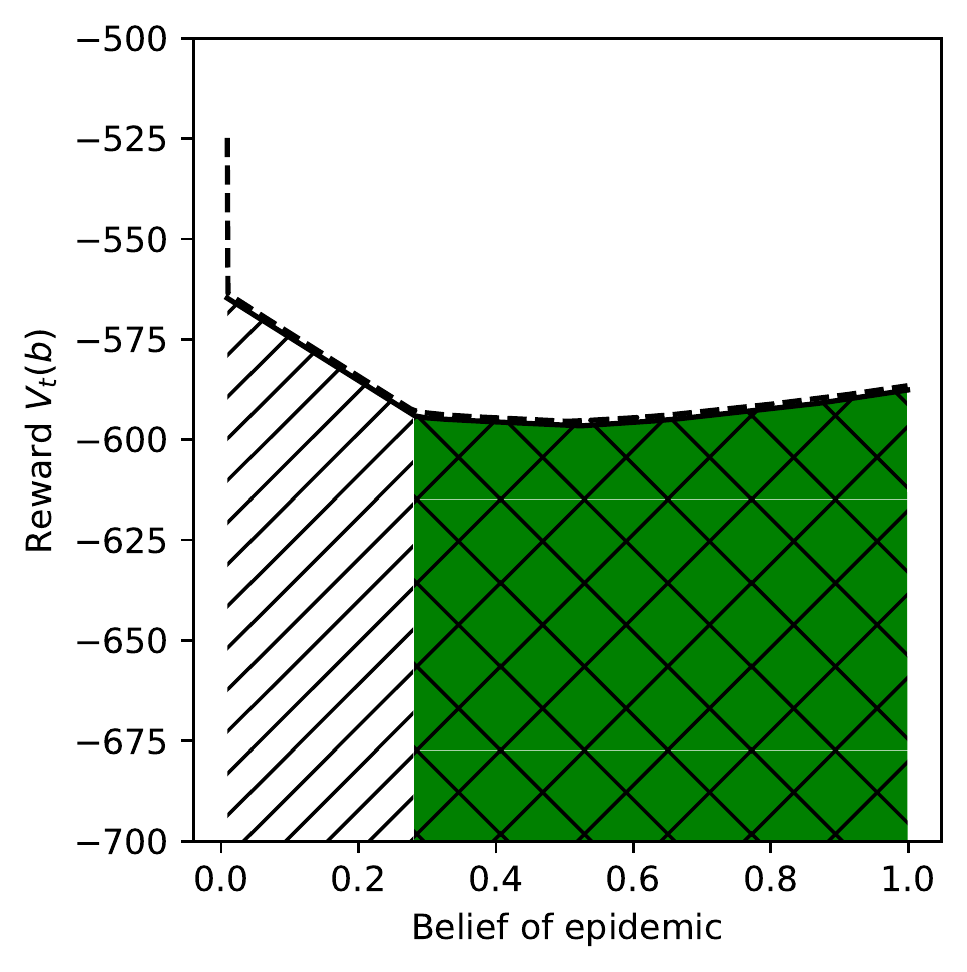}
   \caption{\{(Level 0, E), (Level 0, N), (Level 1, E), (Level 1, N)\}}
   \label{fig:num4}
\end{subfigure}
\hfill
\begin{subfigure}[b]{0.49\textwidth}
   \includegraphics[width=\textwidth]{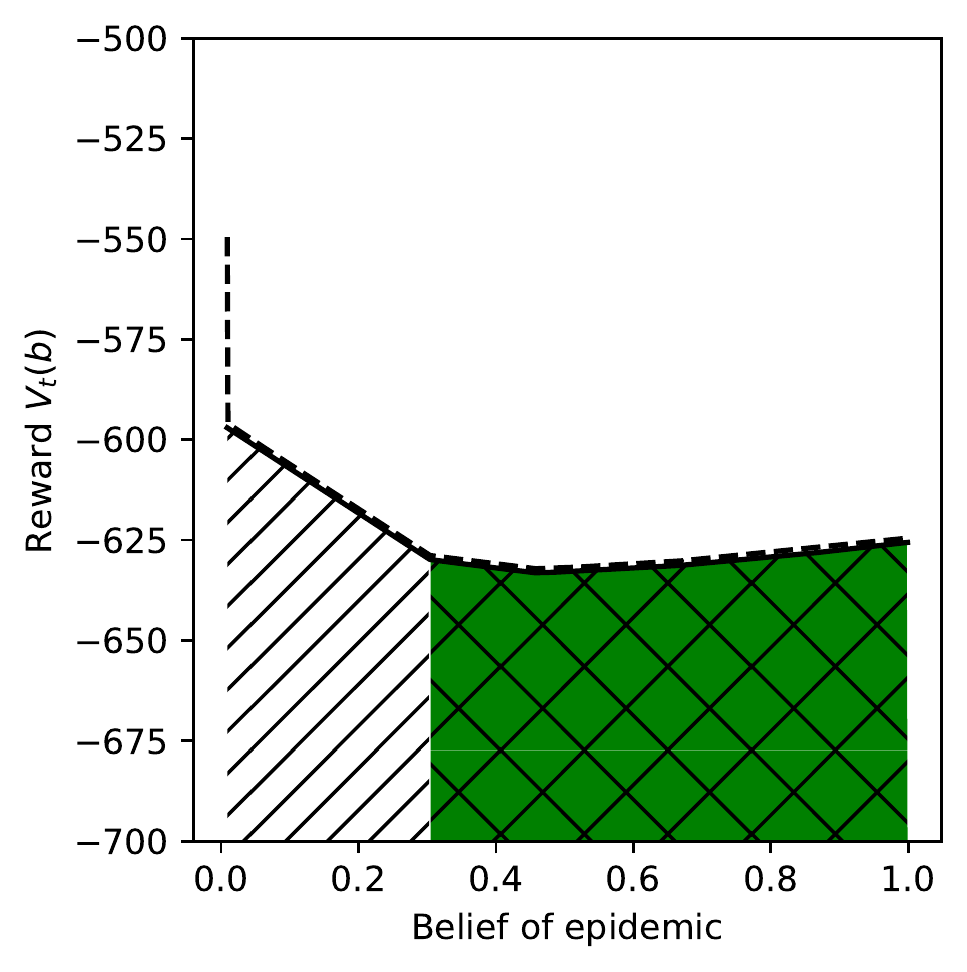}
   \caption{\{(Level 0, E), (Level 0, N), (Level 1, E), (Level 1, N), (Level 2, E)\}}
   \label{fig:num5}
\end{subfigure}
\caption{Value functions for increasing number of action-state pairs with distributional ambiguity. Solid line: lower bound, dashed line: upper bound. Corresponding actions:  Level 0 -- (red, horizontal), Level 1 -- (blue, dot), Level 2 -- (green, cross),  Inspection -- (white, diagonal)}
\label{fig:AmbiguityNumber}
\end{figure}

We observe that the reward becomes smaller as we increase the number of action-state pairs with distributional ambiguity. This is because the worst-case scenario is considered jointly for all action-state pairs and the DR-POMDP policy aims to achieve a conservative reward outcome. Moreover, the belief range where Level 1 action is taken becomes smaller as we consider the distributional ambiguity in the transition-observation probabilities associated with $a=1$. The `Inspection' action also replaces the Level 0 action as we increase the number of ambiguity sources.

\subsubsection{Solution Robustness under Different Ambiguity Sets}
\label{sec:sensitivity1}

We simulate the DR-POMDP policies on instances with an initial state `E' chosen with probability 50\%. We use different sizes of ambiguity sets for the nature to choose the worst-case distributions in the in-sample computation. Specifically, we consider $c=0.03, 0.06, 0.09$ to compute DR-POMDP policies using the ambiguity setting in Section \ref{sec:varyingASsize} and then vary $c' = 0.00, 0.03, 0.06, 0.09$ to change the nature's ambiguity set size for testing each DR-POMDP policy. 

\begin{figure}[htbp]
\centering
\begin{subfigure}[b]{0.6\textwidth}
   \includegraphics[width=\textwidth]{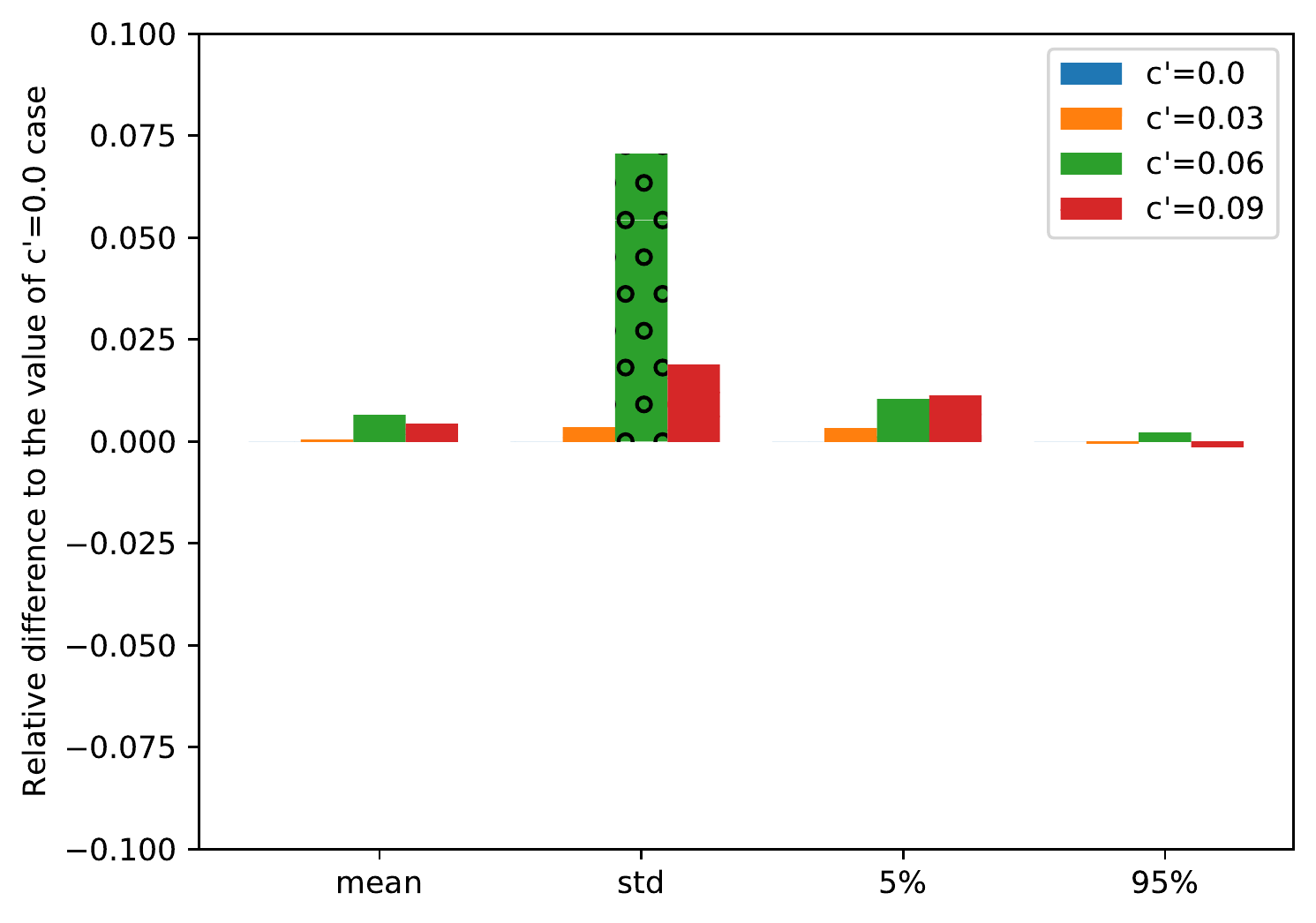}
   \caption{DR-POMDP ($c=0.03$)}
   \label{fig:Ic0} 
\end{subfigure}

\begin{subfigure}[b]{0.6\textwidth}
   \includegraphics[width=\textwidth]{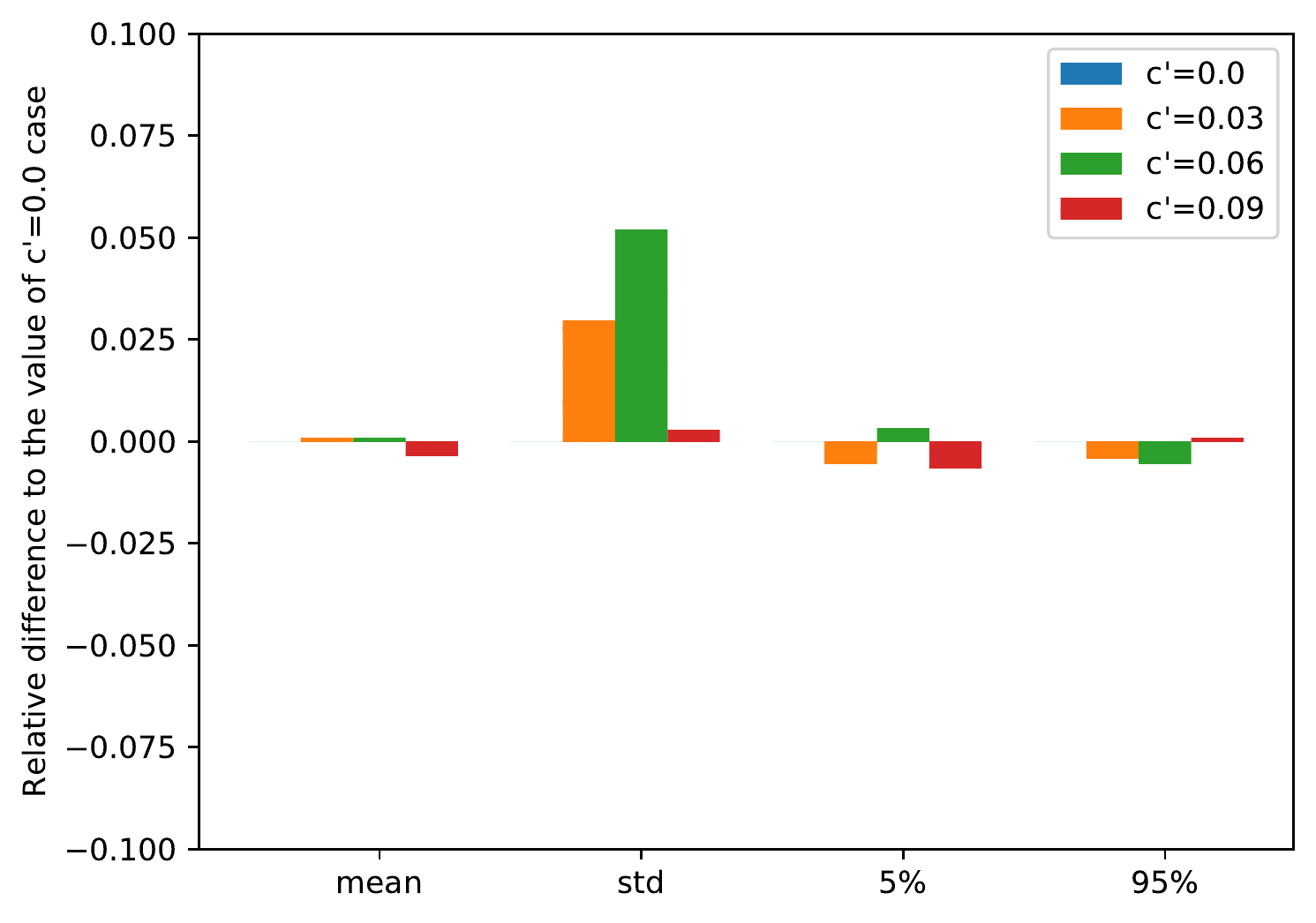}
   \caption{DR-POMDP ($c=0.06$)}
   \label{fig:Ic3}
\end{subfigure}

\begin{subfigure}[b]{0.6\textwidth}
   \includegraphics[width=\textwidth]{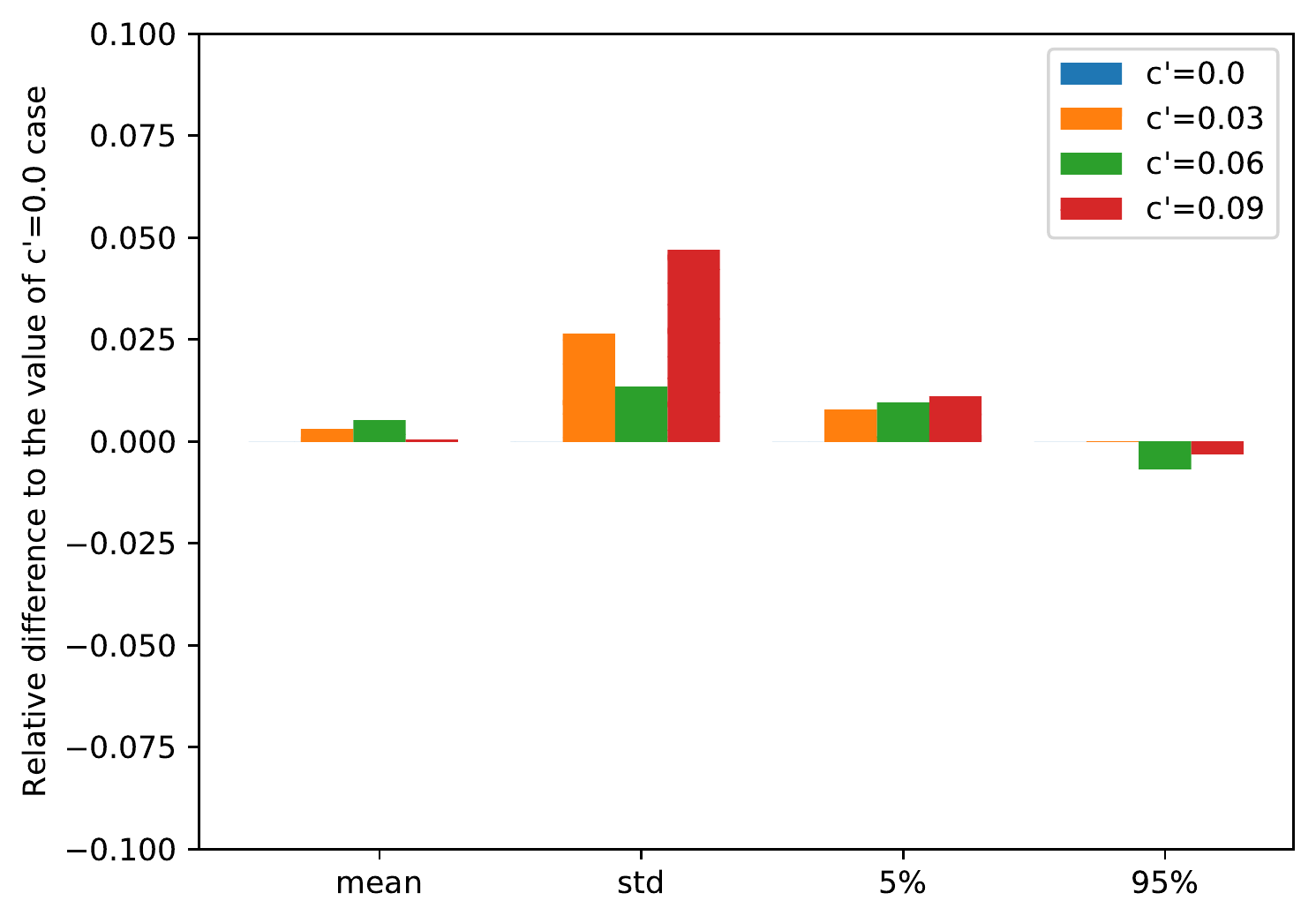}
   \caption{DR-POMDP ($c=0.09$)}
   \label{fig:Ic6}
\end{subfigure}

\caption{Statistics of the reward (mean, standard deviation, 5-percentile, 95-percentile) obtained by implementing DR-POMDP policies in in-sample tests under different ambiguity sets used by the nature.}
\label{fig:insample1}
\end{figure}

Figure \ref{fig:insample1} presents the statistics of the reward, including mean, standard deviation, 5-percentile and 95-percentile values, by implementing the DR-POMDP policies in in-sample tests when the nature uses different sizes of ambiguity sets to choose the worst-case distribution for the transition-observation probabilities. We observe that DR-POMDP policies are robust and not sensitive to the ambiguity set size change, especially in the mean, worst and best reward values.

\subsubsection{Solution Sensitivity under Noise Added to the Realized Transition-Observation Probabilities}
\label{sec:sensitivity2}

We argue that our assumption about the true transition-observation probabilities being accessible at the end of each decision period is relatively weak, by testing the DR-POMDP policies in out-of-sample scenarios while adding noise to the $\p$-value obtained at the end of each period. Specifically, when the DM takes Level 0 action, the transition probability of switching from an epidemic state to a non-epidemic state follows $p_{0}(N|E)=0.99+e\cdot x$, where $e\in\{0.0,0.1,0.2,0.3\}$, and $x$ follows a standard Normal distribution. (We ensure that $0\leq p_{0}(N|E) \leq 1$ and re-sample if not.) 

\begin{figure}[htbp]
\centering
\begin{subfigure}[b]{0.6\textwidth}
   \includegraphics[width=\textwidth]{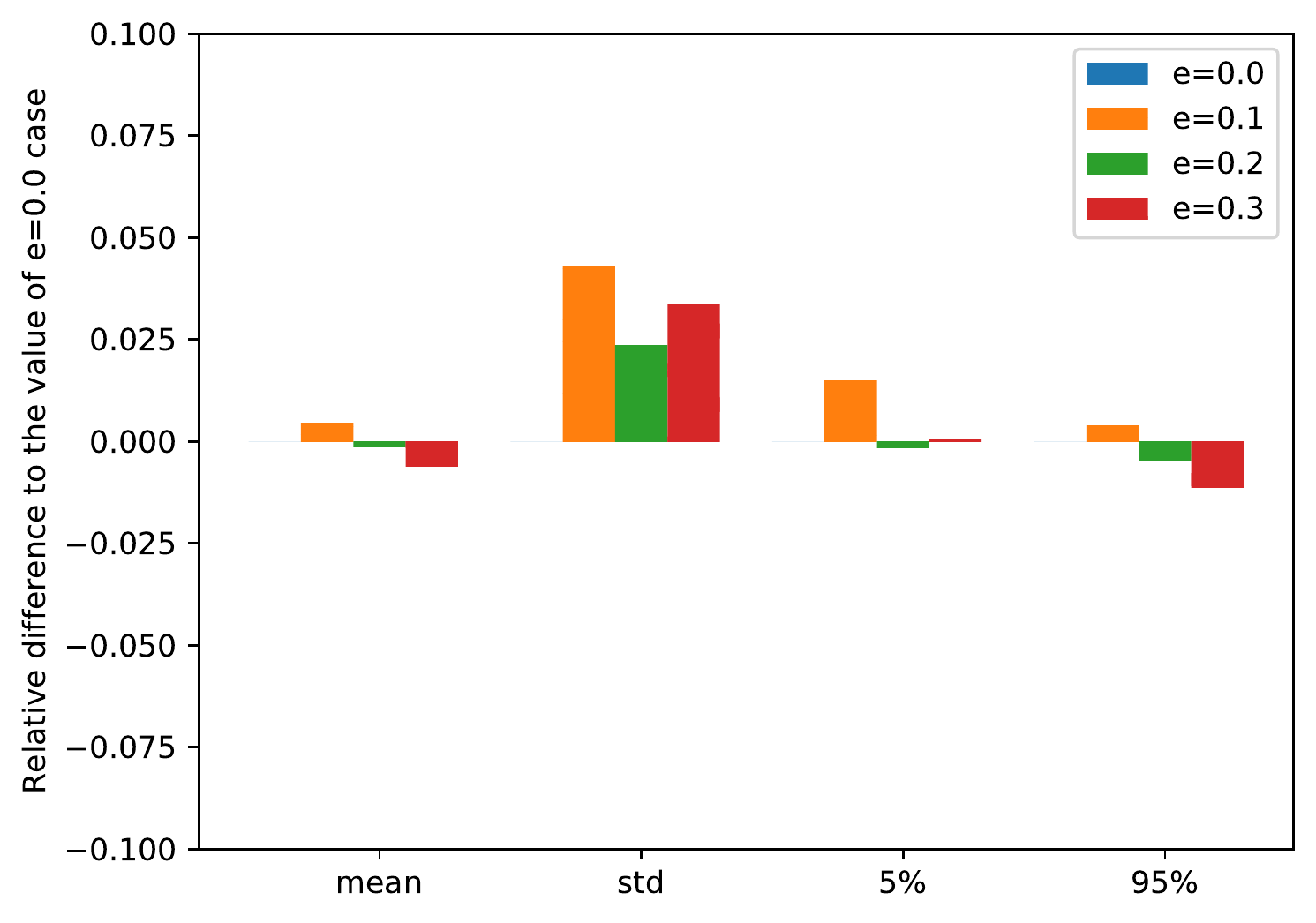}
   \caption{DR-POMDP ($c=0.03$)}
   \label{fig:Oc0} 
\end{subfigure}

\begin{subfigure}[b]{0.6\textwidth}
   \includegraphics[width=\textwidth]{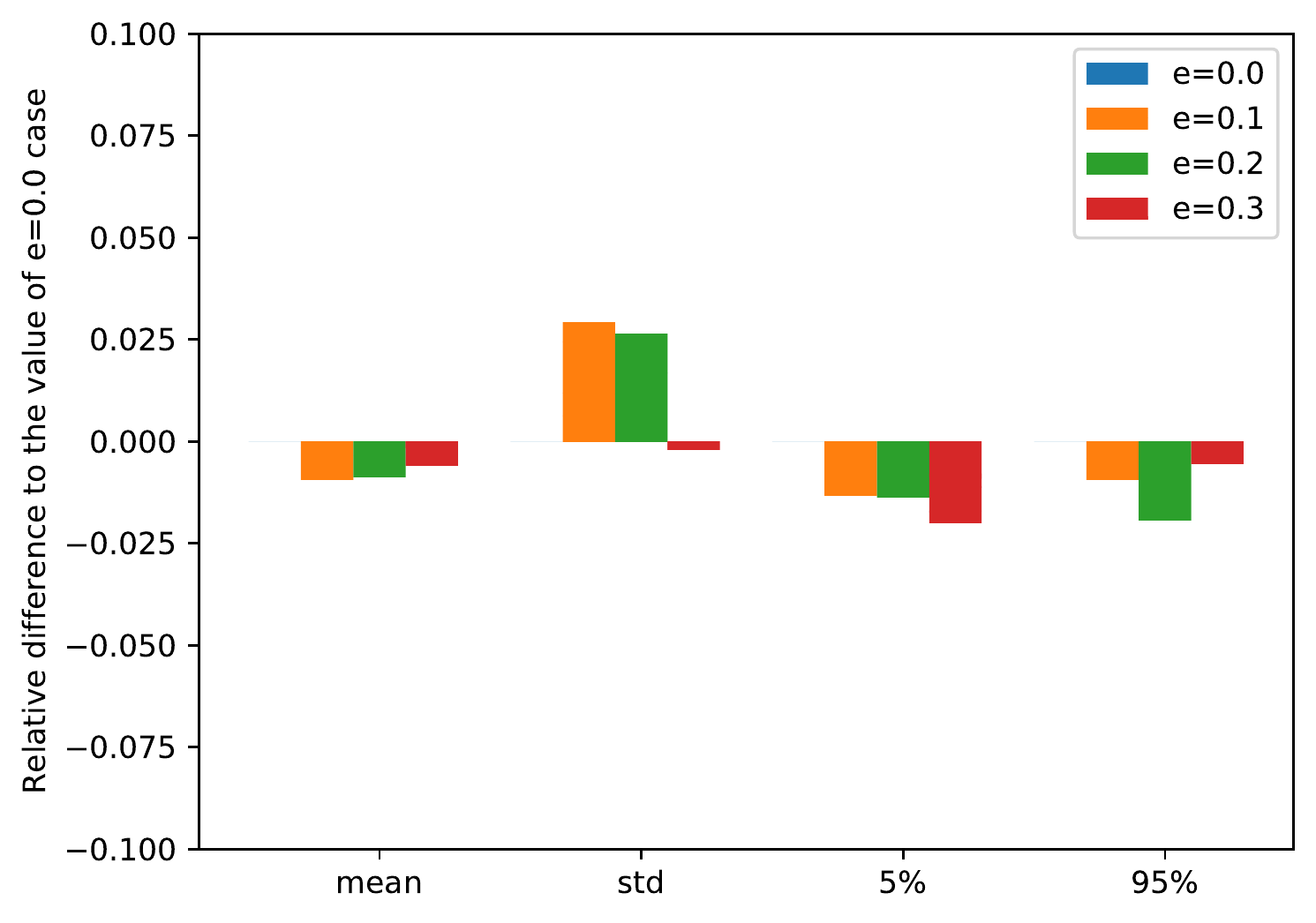}
   \caption{DR-POMDP ($c=0.06$)}
   \label{fig:Oc3}
\end{subfigure}

\begin{subfigure}[b]{0.6\textwidth}
   \includegraphics[width=\textwidth]{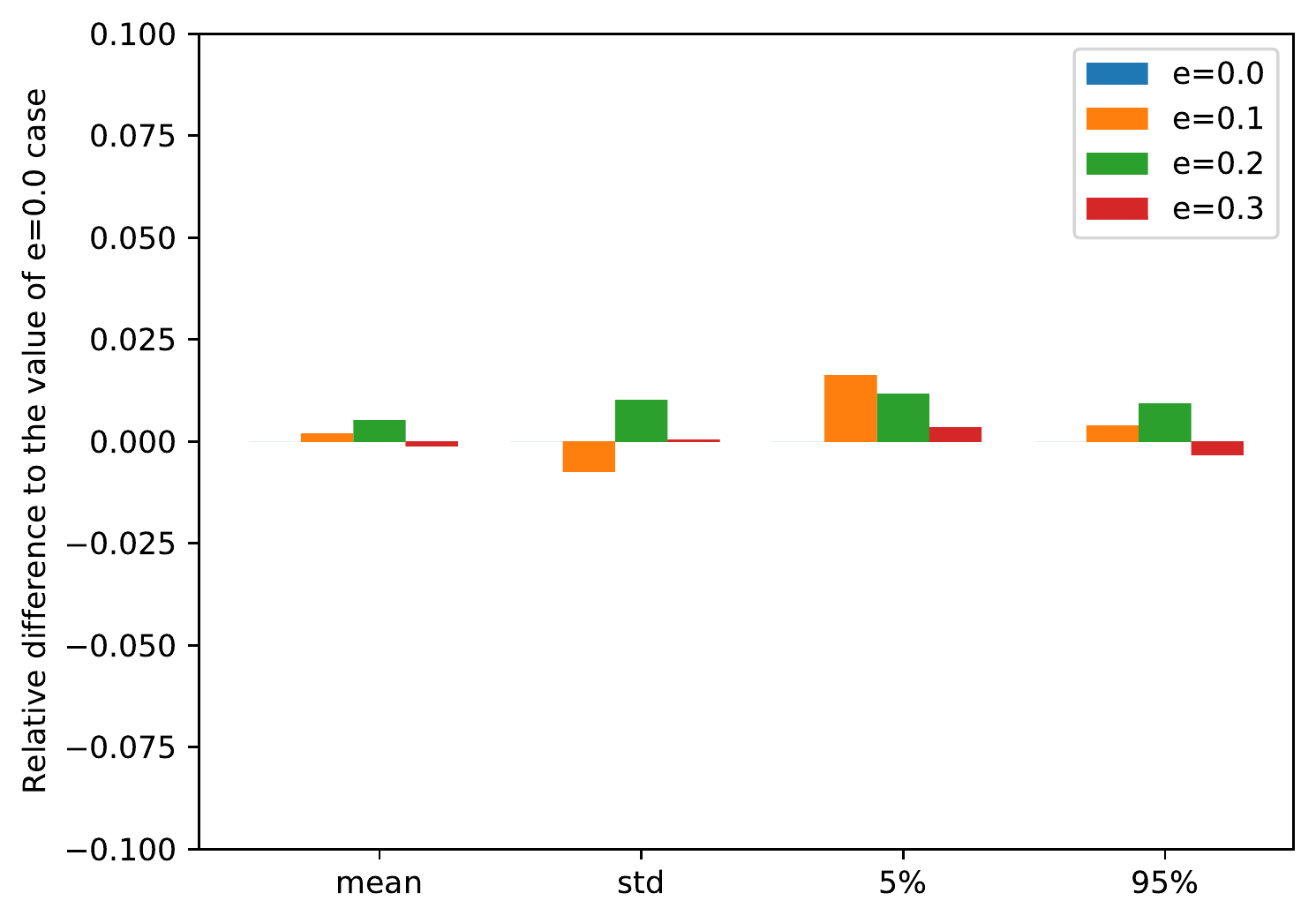}
   \caption{DR-POMDP ($c=0.09$)}
   \label{fig:Oc6}
\end{subfigure}
\caption{Statistics of the reward (mean, standard deviation, 5-percentile, 95-percentile) obtained by performing DR-POMDP policies in out-of-sample tests with noisy $\p$-values.}
\label{fig:outsample1}
\end{figure}

Figure \ref{fig:outsample1} presents the statistics of the reward, including mean, standard deviation, 5-percentile and 95-percentile values, by implementing the DR-POMDP policies in out-of-sample scenarios under varying $\p$-values obtained at the end of each decision period. 
Similar to the previous section, we compare the reward statistics with the case when $e=0.0$, i.e., the case when the DM can fully access the true $\p$-value at the end of each period.  For different ambiguity sets ($c=0.03, 0.06, 0.09$), the DR-POMDP solutions are not sensitive to the perturbation of $\p$-values obtained at the end of each period as we increase the noise. Moreover, all the statistics are within less than $2.5\%$ differences from the results of $e=0.0$, indicating that our assumption  about the necessity of using side information to obtain the true $\p$-value at the end of each period is not strong. 

\subsection{Large-scale Dynamic Epidemic Control Problem}\label{subsec:dynamic}
We demonstrate the algorithmic convergence and compare the computational-time difference for larger-sized instances when applying the HSVI algorithm. We increase the problem size and instance diversity by extending the previous two-state model. Specifically, we consider people who are susceptible to infection and people who have recovered, so that we can model the variation and dynamics in the infection rate. We utilize the SIR compartmental model in epidemiology  \citep[see][]{hethcote2000mathematics, harko2014exact}, where $S$, $I$, $R$ represent the susceptible, infected and recovered population ratios, respectively. These quantities can be modeled using differential equations: 
\begin{align*}
    \frac{dS(t)}{dt} &= -a_1I(t)S(t),\\
    \frac{dI(t)}{dt} &= a_1I(t)S(t)-a_0I(t),\\
    \frac{dR(t)}{dt} &= a_0I(t),
\end{align*}
where $a_0$ is the rate of recovery, and $a_1$ is the average number of contacts per person per time. In this problem setting, we assume that these quantities can be controlled by the DM. 
We discretize the time horizon and consider discretized states $\tilde{S}$, $\tilde{I}$, $\tilde{R}$. Furthermore, we take a first-order approximation and define the transition probabilities such that they satisfy
\begin{align*}
    \E\left[\tilde{S}^{t+1}|\tilde{S}^t\right]&=\tilde{S}^t-a_1\tilde{I}^t\tilde{S}^t dt,\\
    \E\left[\tilde{I}^{t+1}|\tilde{I}^t\right]&=\tilde{I}^t+a_1\tilde{I}^t\tilde{S}^t dt-a_0\tilde{I}^t dt,\\
    \E\left[\tilde{R}^{t+1}|\tilde{R}^t\right]&=\tilde{R}^t+a_0\tilde{I}^t dt.
\end{align*}
We further assume that the states can only transition to its neighboring states, and the quantity of $\tilde{S}$ cannot increase. (Similarly, the quantity of $\tilde{R}$ cannot decrease.) We assume $dt=1$ in the subsequent discussion.

The DM is able to make an imperfect observation of the state $\tilde{I}^t$. The outcome of the observation is typically less than or equal to the true state $\hat{I}$, and the accuracy depends on the quality of the test. We assume that the observation outcome follows a Normal distribution with mean $a_2\times\hat{I}$ (with $a_2$ being a parameter that the DM can control) and standard deviation $0.25\times\hat{I}$, and is further discretized by allocating the probability mass to the closest discrete observation outcome. 

Moreover, the DM can implement certain epidemic control policies to vary $a_1\in[0.1,1.0]$ and $a_2\in[0,1]$, and we fix $a_0 = 0.25$. Choosing a low value of $a_1$ results in high cost due to its economic impact for a strict measure, and choosing a high value of $a_2$ results in high cost due to operating an expensive test process. We set the goal to minimize the number of infected people and preventing it from exceeding the treatment capacity, which is set as 0.2\% of the overall population. Each percentage of population being infected will result in 10 units of cost, while 15 units of cost is incurred when the total infection is more than treatment capacity. Varying one unit of the $a_1$- and $a_2$-values costs 10 and 3 units, respectively. Additionally, when the total infection is more  than 0.5\% of the population, a reward $=20$ will be given for performing the most strict measure in $a_1$. Therefore, 
\begin{align*}
    r_{as} &= 
    \begin{cases}
    - 1000\times\tilde{I} - 10\times(1.0-a_1) - 3\times a_2, & \mbox{if } \hat{I}<0.002\\
    - 2500\times\tilde{I} - 10\times(1.0-a_1) - 3\times a_2, & \mbox{if } \hat{I}\geq0.002,
    \end{cases}\\
    &\quad + 20 \mbox{ if $\hat{I}\geq0.005$ and $a_1$ is the lowest value.}
\end{align*}
where $a \in \{a_1,a_2\}$ and $s \in \{\tilde{S},\tilde{I},\tilde{R}\}$.

\subsubsection{Computational Time for Varying Numbers of States}\label{sec:state-time}
Let $\tilde{I}=0.001$ and $0.005$, representing the `Non-epidemic' state and `Epidemic' state, respectively. We consider the following discretization schemes for the states $\tilde{S}$: $\{0.90,0.95\}$, $\{0.50,0.70,0.90,0.95\}$, and $\{0.30,0.40,0.50,0.60,0.70,0.80,0.90,0.95\}$.

In the numerical experiment, we only consider ambiguities in the action $a_1=1.0$, corresponding to implementing the least strict control policy for reducing the infection rate. We set the radius of the ambiguity set as $c=0.02$. Thus, the different problem sizes are $(s4,a4,z3,u8)$, $(s8,a4,z3,u16)$, and $(s16,a4,z3,u32)$. We set the initial belief to be totally in the non-epidemic state, and allow a tolerance $\epsilon=1.0$. The computational time limit is 3600 seconds. 
\begin{figure}[htbp]
\centering
\begin{subfigure}[b]{0.49\textwidth}
   \includegraphics[width=\textwidth]{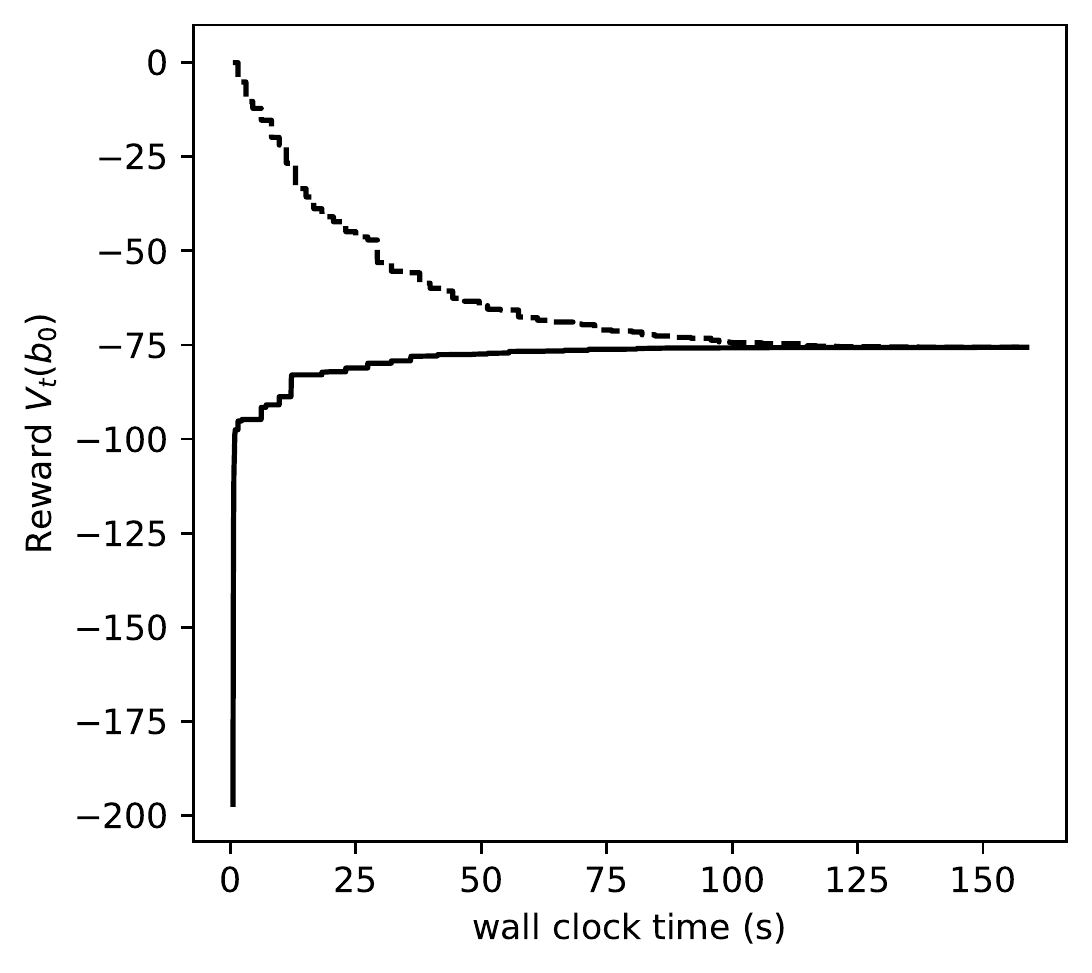}
   \caption{POMDP ($c=0.00$)}
\end{subfigure}
\hfill
\begin{subfigure}[b]{0.49\textwidth}
   \includegraphics[width=\textwidth]{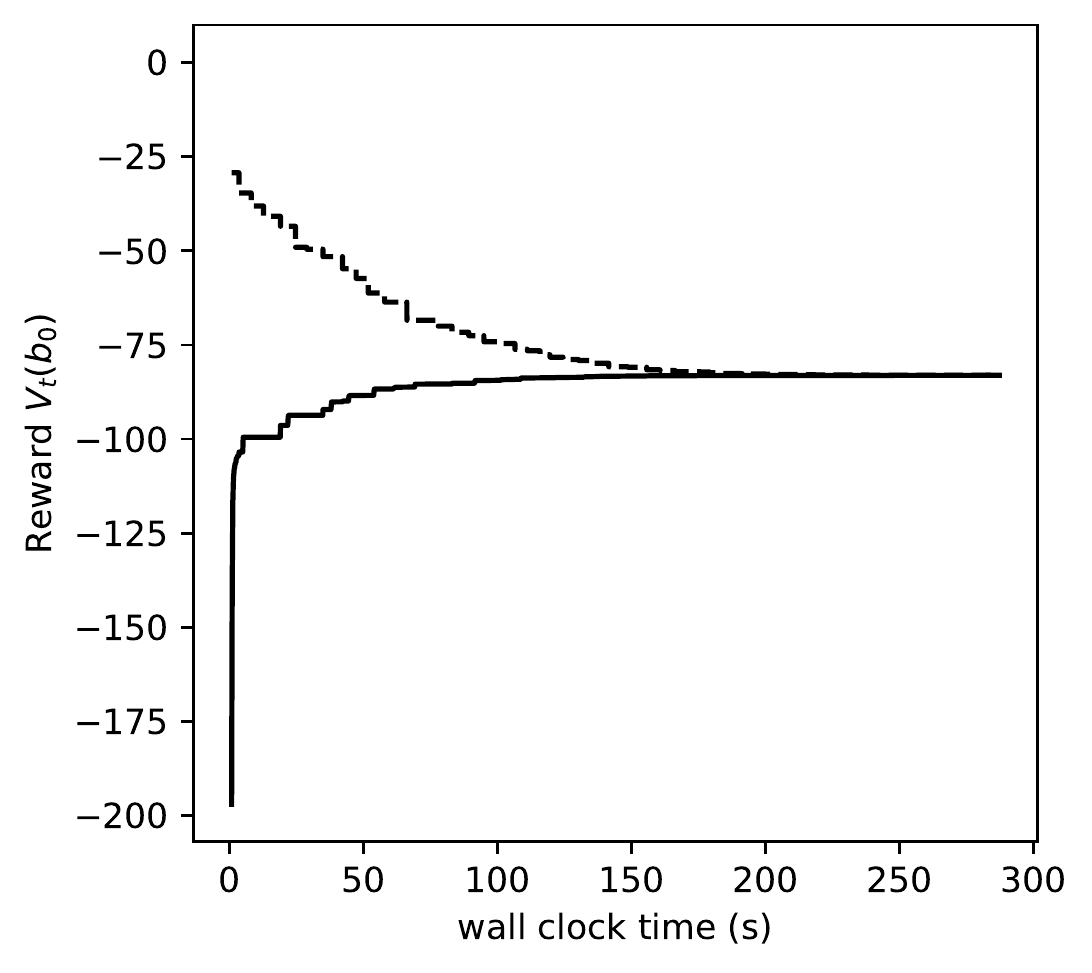}
   \caption{DR-POMDP ($c=0.02$)}
\end{subfigure}
\caption{Dynamic epidemic control problem instance $(s4,a4,z3,u8)$. Solid line: lower bound, dashed line: upper bound}
\label{fig:n4}
\end{figure}

\begin{figure}[htbp]
\centering
\begin{subfigure}[b]{0.49\textwidth}
   \includegraphics[width=\textwidth]{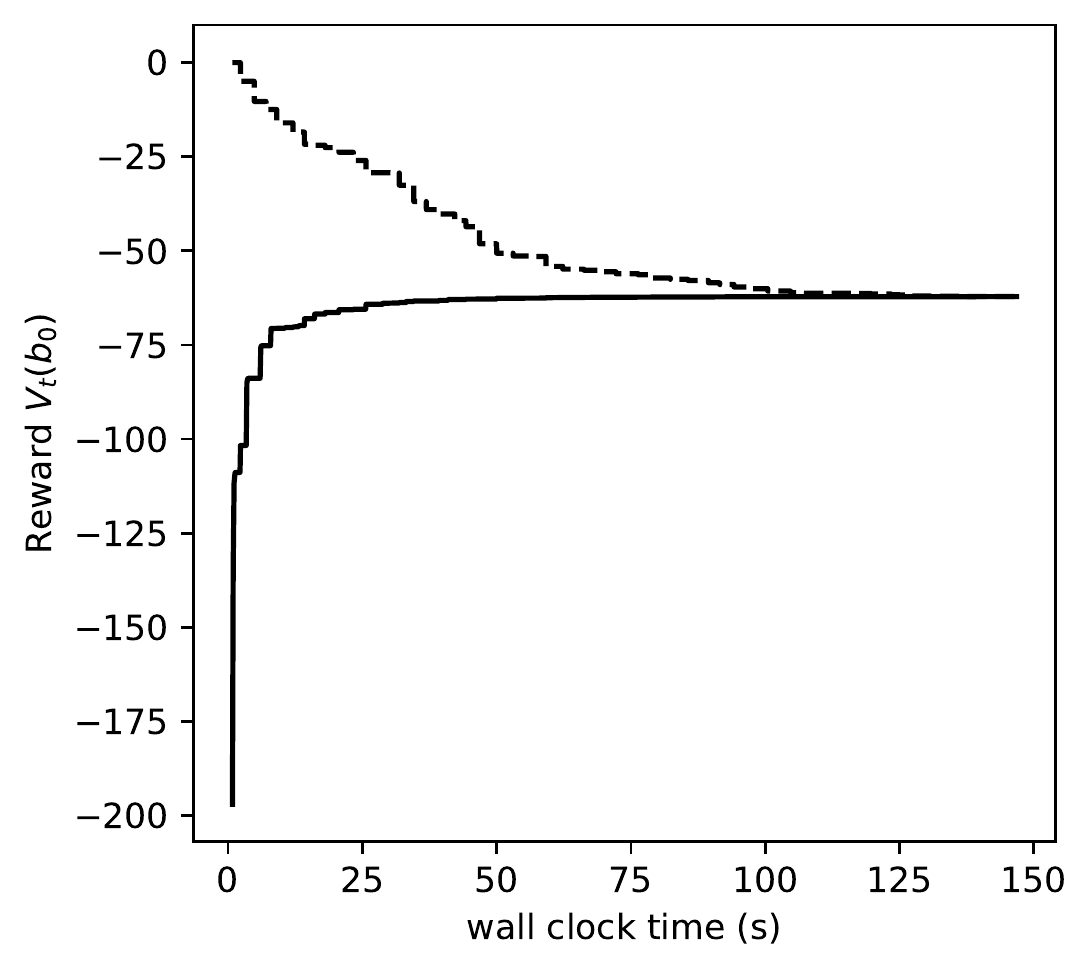}
   \caption{POMDP ($c=0.00$)}
\end{subfigure}
\hfill
\begin{subfigure}[b]{0.49\textwidth}
   \includegraphics[width=\textwidth]{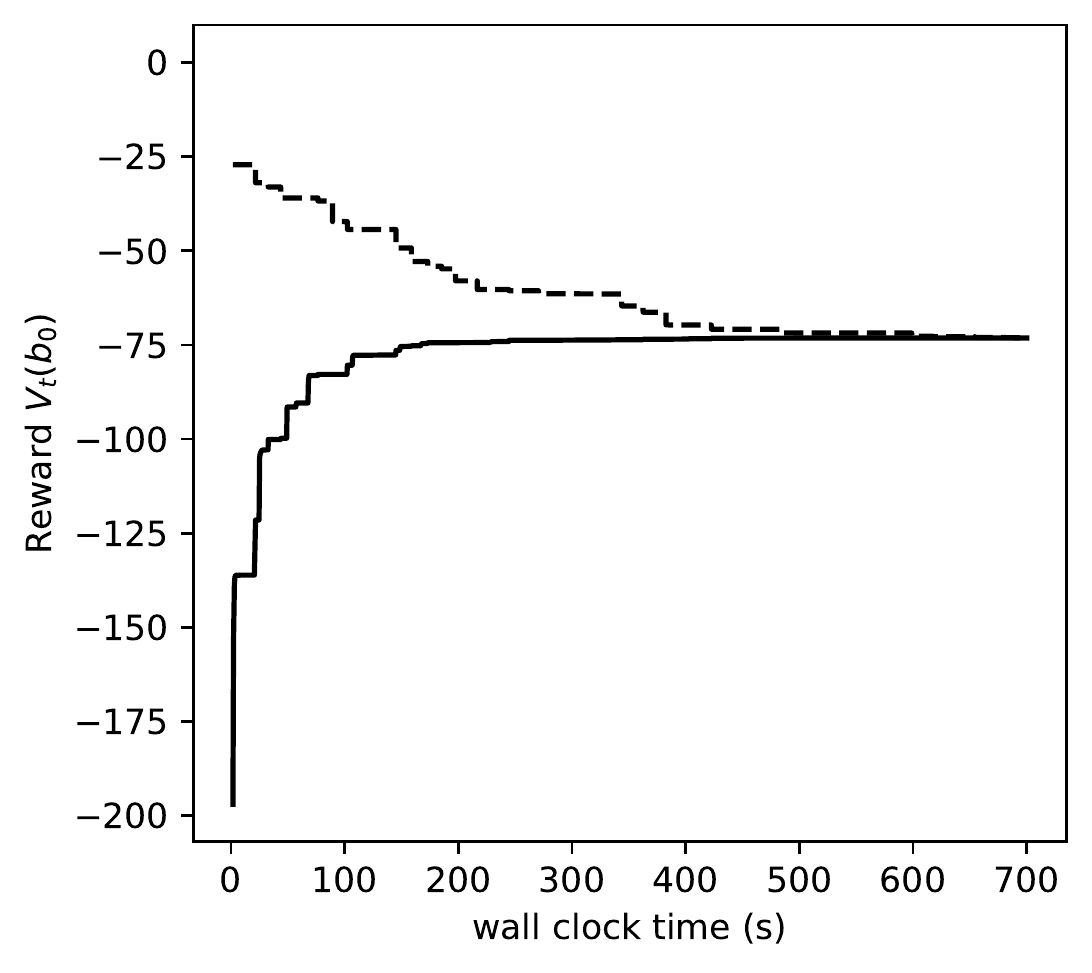}
   \caption{DR-POMDP ($c=0.02$)}
\end{subfigure}
\caption{Dynamic epidemic control problem instance $(s8,a4,z3,u16)$. Solid line: lower bound, dashed line: upper bound}
\label{fig:n8}
\end{figure}

\begin{figure}[htbp]
\centering
\begin{subfigure}[b]{0.49\textwidth}
   \includegraphics[width=\textwidth]{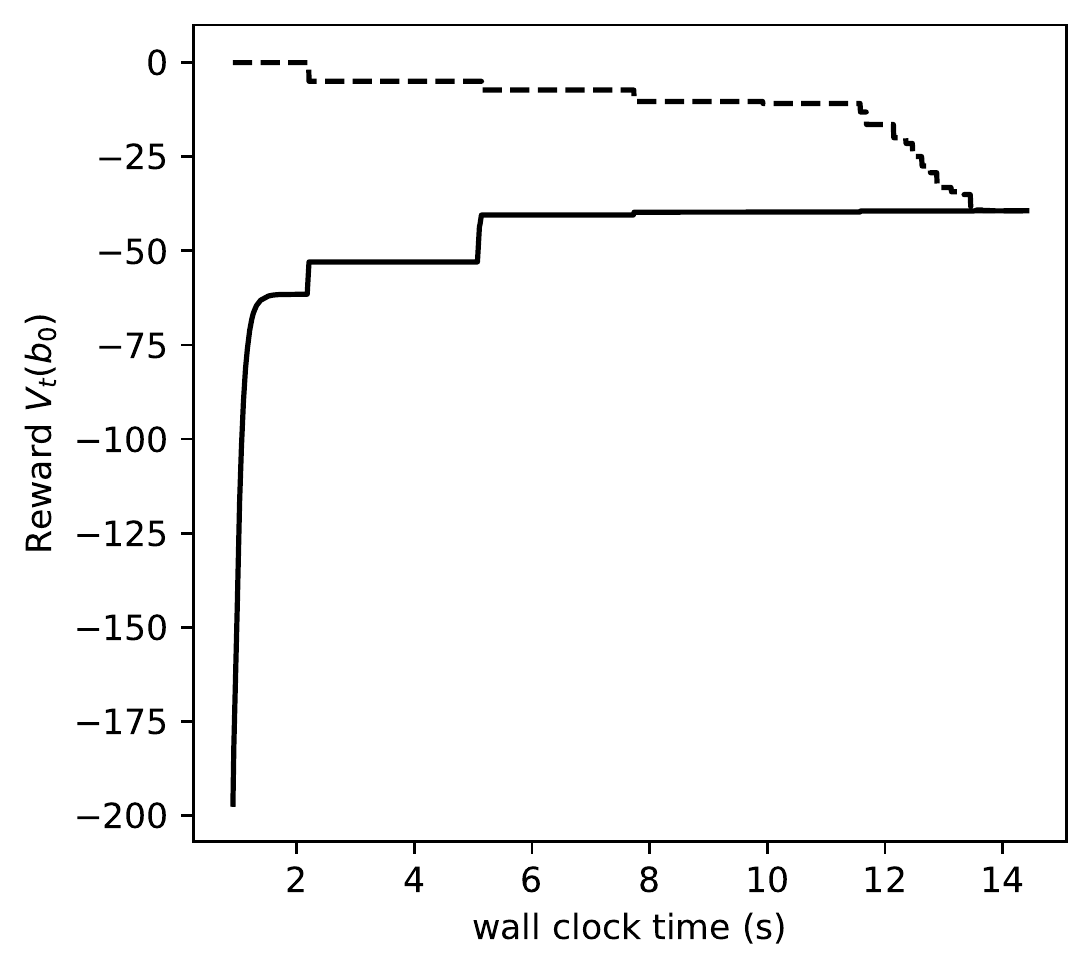}
   \caption{POMDP ($c=0.00$)}
\end{subfigure}
\hfill
\begin{subfigure}[b]{0.49\textwidth}
   \includegraphics[width=\textwidth]{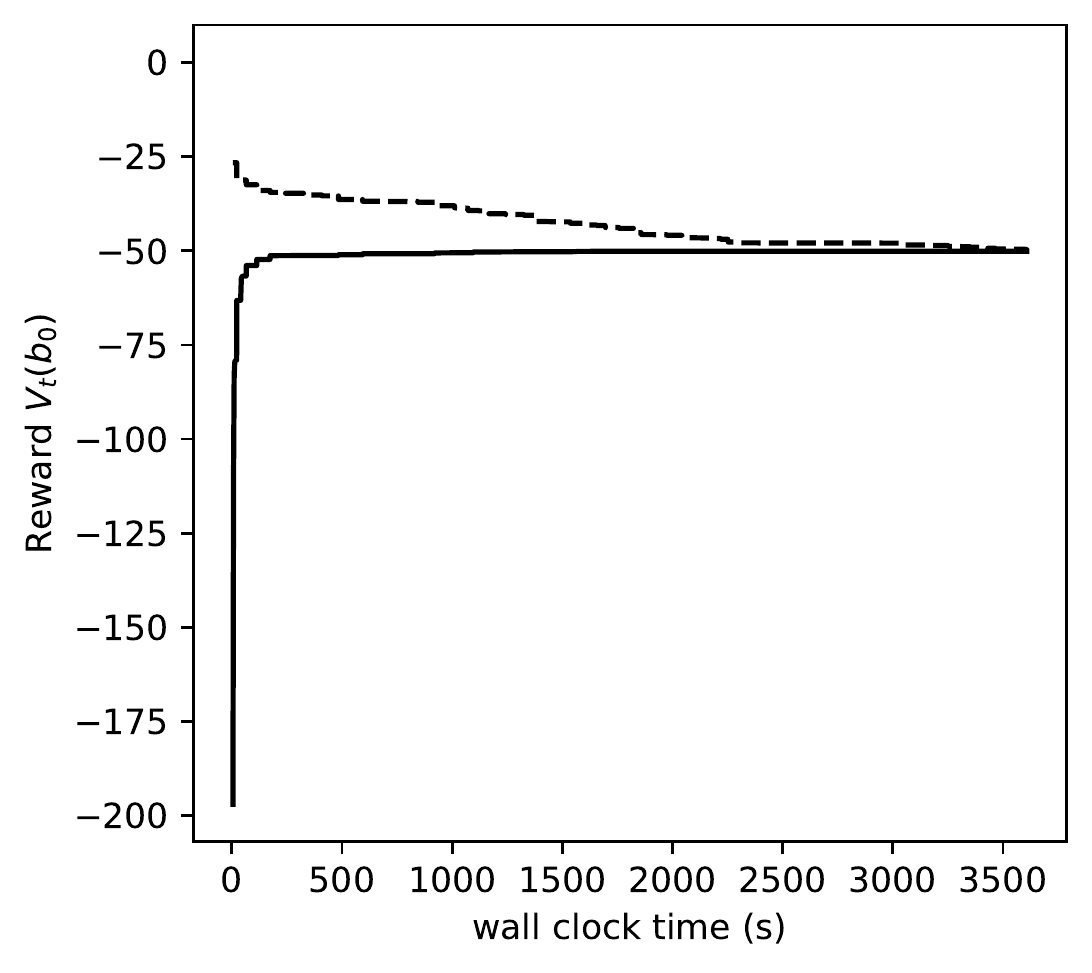}
   \caption{DR-POMDP ($c=0.02$)}
\end{subfigure}
\caption{Dynamic epidemic control problem instance $(s16,a4,z3,u32)$. Solid line: lower bound, dashed line: upper bound}
\label{fig:n16}
\end{figure}

In Figures \ref{fig:n4}, \ref{fig:n8}, \ref{fig:n16}, we depict how the upper bound and lower bound of POMDP ($c=0.00$) and DR-POMDP ($c=0.02$) policies converge as functions of time for the above three problem sizes, respectively. We observe that the computational time for POMDP does not correlate with the number of states. When the number of states are 4 and 8, the corresponding instances take about 150 seconds to converge, as compared to the instances having 16 states take about 14 seconds to converge. On the other hand, the computational time for DR-POMDP increases as the number of states and ambiguity sets increase. We also point out that the value function for DR-POMDP evaluated at $b_0$ is lower than that of POMDP, which is expected since DR-POMDP is more conservative. 

\subsubsection{Computation Time for Varying Uncertainty Sizes}\label{sec:uncertainty-time}
We change the number of ambiguity sets and compare their solutions and computation time. The states are $\tilde{S}\in\{0.50,0.70,0.90,0.95\}$ and $\tilde{I}\in\{0.001,0.005\}$, and actions are $(a_1,a_2)\in\{(0.1,0.1),(0.1,1.0),(1.0,0.1),(1.0,1.0)\}$. We increase the number of actions that are associated with ambiguity sets from 1 to 4. Since there are 8 states in total, the number of ambiguity sets are 8, 16, 32, and 64, respectively. The results are shown in Figure \ref{fig:u}. The solution time are 614, 625, 1012, 1497 seconds, respectively and increase as the number of ambiguity sets increases. The optimal objective values are $-62.64$, $-64.58$, $-71.71$, $-72.99$, respectively, and decrease monotonically.
\begin{figure}[htbp]
\centering
\begin{subfigure}[b]{0.49\textwidth}
   \includegraphics[width=\textwidth]{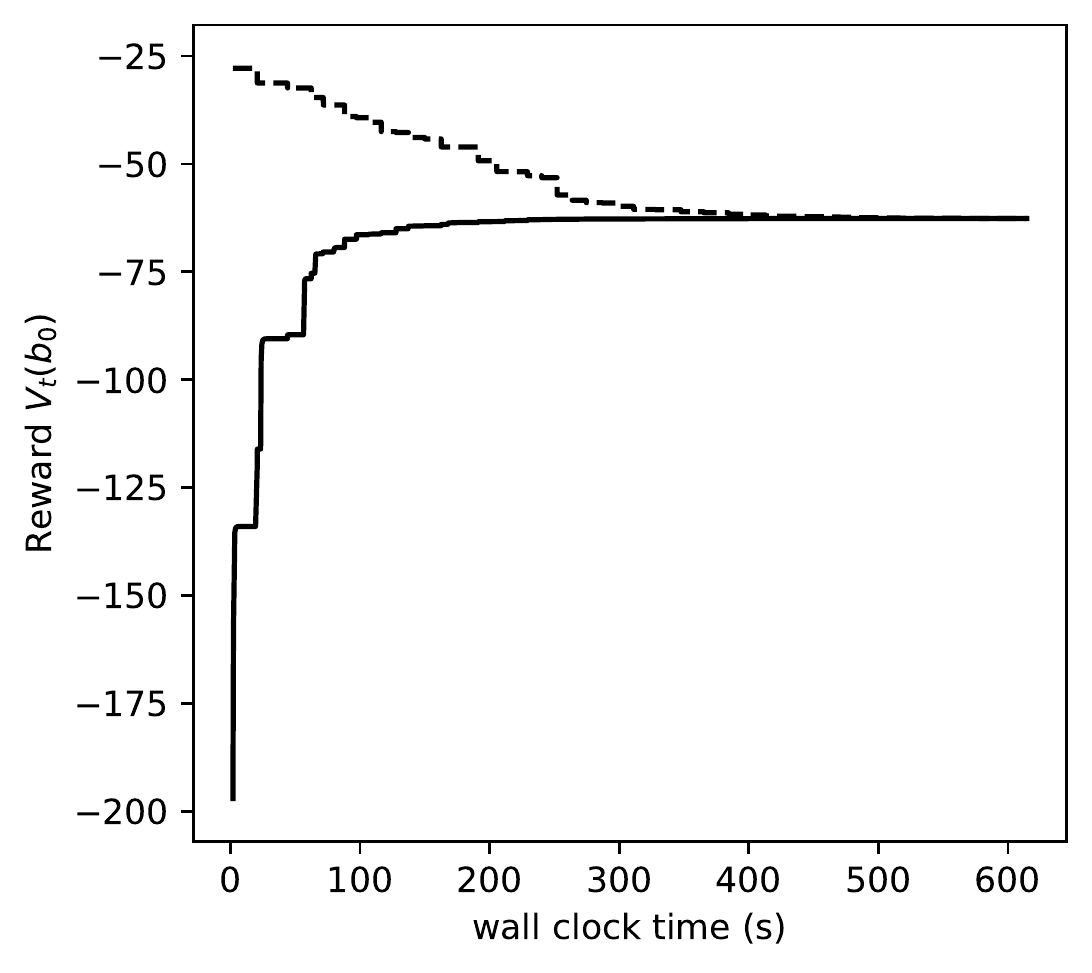}
   \caption{DR-POMDP $(s8,a4,z3,u8)$}
\end{subfigure}
\hfill
\begin{subfigure}[b]{0.49\textwidth}
   \includegraphics[width=\textwidth]{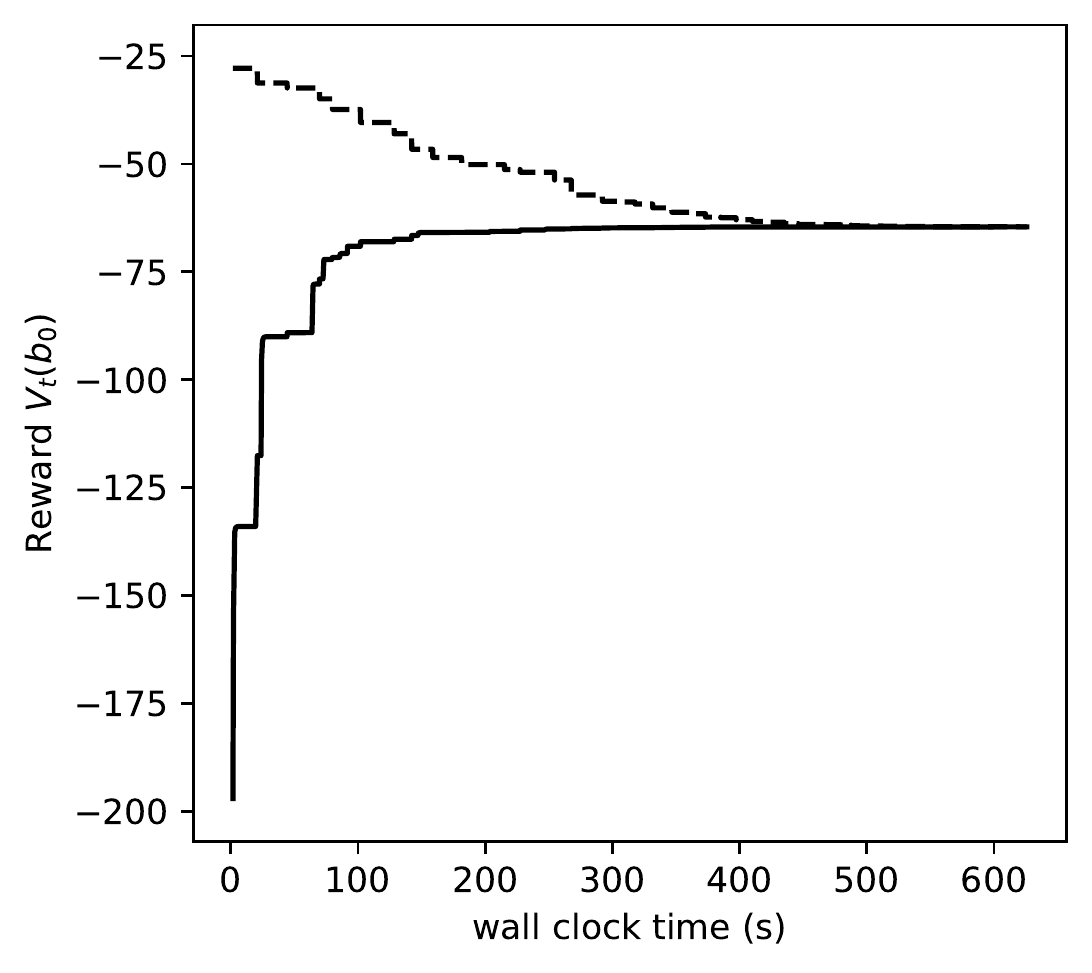}
   \caption{DR-POMDP $(s8,a4,z3,u16)$}
\end{subfigure}
\\
\begin{subfigure}[b]{0.49\textwidth}
   \includegraphics[width=\textwidth]{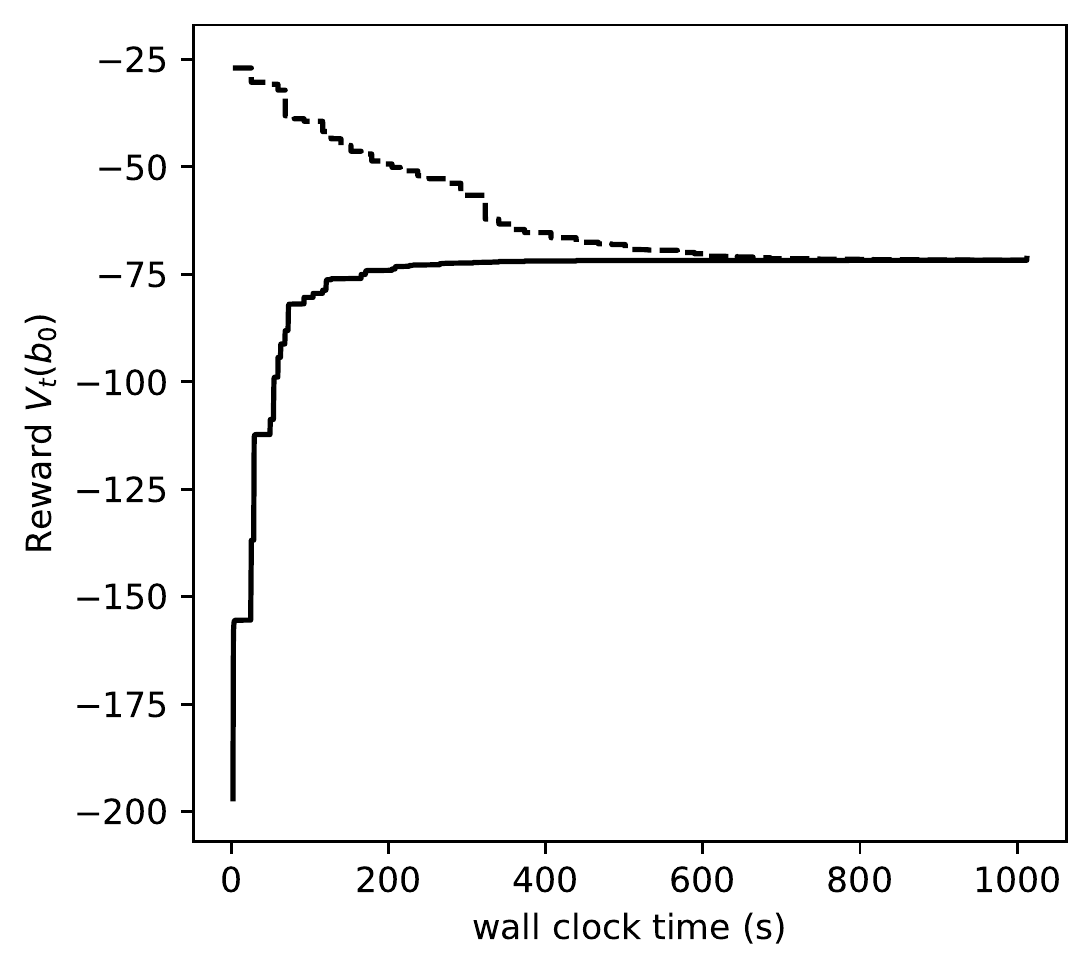}
   \caption{POMDP $(s8,a4,z3,u32)$}
\end{subfigure}
\hfill
\begin{subfigure}[b]{0.49\textwidth}
   \includegraphics[width=\textwidth]{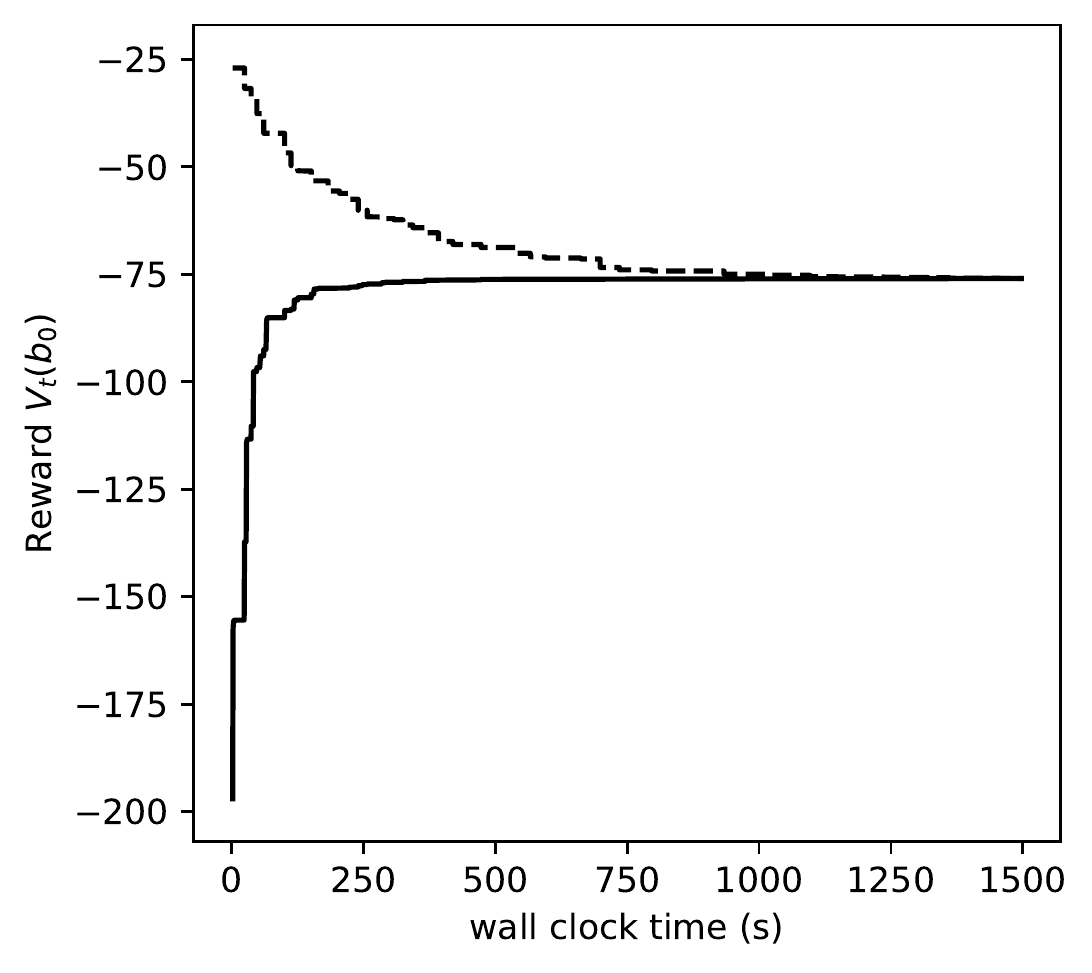}
   \caption{DR-POMDP $(s8,a4,z3,u64)$}
\end{subfigure}
\caption{Dynamic epidemic control problem instances with varying number of ambiguity sets. Solid line: lower bound, dashed line: upper bound}
\label{fig:u}
\end{figure}

\section{Conclusion}
\label{sec:conclusion}
In this paper, we developed new models and algorithms for POMDP when the transition probability and the observation probability are uncertain, and the probability distribution is not perfectly known. We presented a scalable approximation algorithm and numerically compared DR-POMDP optimal policies with the ones of  the standard POMDP and robust POMDP, in both in-sample and out-of-sample tests. Although due to the more complicated model and problem settings, DR-POMDP is much harder to solve, it produces more conservative and robust results than POMDP. It is also not sensitive to the misspecified ambiguity set and true transition-observation probability values obtained at the end of each decision period. 

In the future research, we aim to solve DR-POMDP when the outcomes of the transition-observation probabilities are not observable to the DM at the end of each time. In such a case, the value function is dependent on a set of belief states, where the characterization of the value function becomes much more challenging. We are also interested in designing randomized policy or time-dependent policy for DR-POMDP when we relax the condition that the nature is able to perfectly observe the DM's action, or when the nature is not completely adversarial. We will compare the performance of different types of policies on diverse instances.

\section*{Acknowledgments}
The authors thank the referees and the Associate Editor for their constructive comments and helpful suggestions. The authors gratefully acknowledge the support from the U.S. Department of Engineering (DoE) grant \# DE-SC0018018 and National Science Foundation (NSF) grant \# CMMI-1727618. 


\begin{thebibliography}{}

\bibitem[Abbad and Filar, 1992]{abbad1992perturbation}
Abbad, M. and Filar, J.~A. (1992).
\newblock Perturbation and stability theory for {Markov} control problems.
\newblock {\em IEEE Transactions on Automatic Control}, 37(9):1415--1420.

\bibitem[Abbad et~al., 1990]{abbad1990algorithms}
Abbad, M., Filar, J.~A., and Bielecki, T.~R. (1990).
\newblock Algorithms for singularly perturbed limiting average {Markov} control
  problems.
\newblock In {\em Decision and Control, 1990., Proceedings of the 29th IEEE
  Conference on}, pages 1402--1407. IEEE.

\bibitem[Ben-Tal et~al., 2013]{ben2013robust}
Ben-Tal, A., Den~Hertog, D., De~Waegenaere, A., Melenberg, B., and Rennen, G.
  (2013).
\newblock Robust solutions of optimization problems affected by uncertain
  probabilities.
\newblock {\em Management Science}, 59(2):341--357.

\bibitem[Cassandra, 1998]{cassandra1998survey}
Cassandra, A.~R. (1998).
\newblock A survey of {POMDP} applications.
\newblock In {\em Working notes of AAAI 1998 Fall Symposium on planning with
  partially observable Markov decision processes}, pages 17--24.

\bibitem[Delage and Mannor, 2010]{delage2010percentile}
Delage, E. and Mannor, S. (2010).
\newblock Percentile optimization for {Markov} decision processes with
  parameter uncertainty.
\newblock {\em Operations Research}, 58(1):203--213.

\bibitem[Delage and Ye, 2010]{delage2010distributionally}
Delage, E. and Ye, Y. (2010).
\newblock Distributionally robust optimization under moment uncertainty with
  application to data-driven problems.
\newblock {\em Operations Research}, 58(3):595--612.

\bibitem[Du and Pardalos, 2013]{du2013minimax}
Du, D.-Z. and Pardalos, P.~M. (2013).
\newblock {\em Minimax and Applications}, volume~4.
\newblock Springer Science \& Business Media.

\bibitem[Du et~al., 2017]{du2017evolution}
Du, X., King, A.~A., Woods, R.~J., and Pascual, M. (2017).
\newblock Evolution-informed forecasting of seasonal influenza a (h3n2).
\newblock {\em Science translational medicine}, 9(413):eaan5325.

\bibitem[Esfahani and Kuhn, 2018]{esfahani2015data}
Esfahani, P.~M. and Kuhn, D. (2018).
\newblock Data-driven distributionally robust optimization using the
  {Wasserstein} metric: Performance guarantees and tractable reformulations.
\newblock {\em Mathematical Programming}, 171(1--2):115--166.

\bibitem[Gao and Kleywegt, 2016]{gao2016distributionally}
Gao, R. and Kleywegt, A.~J. (2016).
\newblock Distributionally robust stochastic optimization with {Wasserstein}
  distance.
\newblock {\em arXiv preprint arXiv:1604.02199}.

\bibitem[Harko et~al., 2014]{harko2014exact}
Harko, T., Lobo, F.~S., and Mak, M. (2014).
\newblock Exact analytical solutions of the {Susceptible-Infected-Recovered
  (SIR)} epidemic model and of the sir model with equal death and birth rates.
\newblock {\em Applied Mathematics and Computation}, 236:184--194.

\bibitem[Harrell and Davis, 1982]{harrell1982new}
Harrell, F.~E. and Davis, C. (1982).
\newblock A new distribution-free quantile estimator.
\newblock {\em Biometrika}, 69(3):635--640.

\bibitem[Hauskrecht and Fraser, 2000]{hauskrecht2000planning}
Hauskrecht, M. and Fraser, H. (2000).
\newblock Planning treatment of ischemic heart disease with partially
  observable {Markov} decision processes.
\newblock {\em Artificial Intelligence in Medicine}, 18(3):221--244.

\bibitem[Hethcote, 2000]{hethcote2000mathematics}
Hethcote, H.~W. (2000).
\newblock The mathematics of infectious diseases.
\newblock {\em SIAM review}, 42(4):599--653.

\bibitem[Iyengar, 2005]{iyengar2005robust}
Iyengar, G.~N. (2005).
\newblock Robust dynamic programming.
\newblock {\em Mathematics of Operations Research}, 30(2):257--280.

\bibitem[Jiang and Guan, 2016]{jiang2016data}
Jiang, R. and Guan, Y. (2016).
\newblock Data-driven chance constrained stochastic program.
\newblock {\em Mathematical Programming}, 158(1-2):291--327.

\bibitem[Kumar and Varaiya, 2015]{kumar2015stochastic}
Kumar, P.~R. and Varaiya, P. (2015).
\newblock {\em Stochastic Systems: Estimation, Identification, and Adaptive
  Control}, volume~75.
\newblock SIAM.

\bibitem[Le~Strat and Carrat, 1999]{le1999monitoring}
Le~Strat, Y. and Carrat, F. (1999).
\newblock Monitoring epidemiologic surveillance data using hidden markov
  models.
\newblock {\em Statistics in medicine}, 18(24):3463--3478.

\bibitem[Mannor et~al., 2016]{mannor2016robust}
Mannor, S., Mebel, O., and Xu, H. (2016).
\newblock Robust {MDPs} with k-rectangular uncertainty.
\newblock {\em Mathematics of Operations Research}, 41(4):1484--1509.

\bibitem[{Nakao, Hideaki and Jiang, Ruiwei and Shen, Siqian}, 2020]{online-sup}
{Nakao, Hideaki and Jiang, Ruiwei and Shen, Siqian} (2020).
\newblock {Online Supplement for ``Distributionally Robust Partially Observable
  Markov Decision Process with Moment-based Ambiguity''}.
\newblock \url{http://www-personal.umich.edu/~siqian/dataset.html}.

\bibitem[Nilim and El~Ghaoui, 2005]{nilim2005robust}
Nilim, A. and El~Ghaoui, L. (2005).
\newblock Robust control of {Markov} decision processes with uncertain
  transition matrices.
\newblock {\em Operations Research}, 53(5):780--798.

\bibitem[Osogami, 2015]{osogami2015robust}
Osogami, T. (2015).
\newblock Robust partially observable {Markov} decision process.
\newblock In {\em International Conference on Machine Learning (ICML)}, pages
  106--115.

\bibitem[Pineau et~al., 2003]{pineau2003point}
Pineau, J., Gordon, G., and Thrun, S. (2003).
\newblock Point-based value iteration: An anytime algorithm for {POMDPs}.
\newblock In {\em The Proceedings of International Joint Conferences on
  Artificial Intelligence (IJCAI)}, volume~3, pages 1025--1032.

\bibitem[Puterman, 2014]{puterman2014markov}
Puterman, M.~L. (2014).
\newblock {\em Markov Decision Processes: Discrete Stochastic Dynamic
  Programming}.
\newblock John Wiley \& Sons.

\bibitem[Rasouli and Saghafian, 2018]{rasouli2018robust}
Rasouli, M. and Saghafian, S. (2018).
\newblock Robust partially observable {Markov} decision processes.
\newblock {\em Working paper}.

\bibitem[Rath et~al., 2003]{rath2003automated}
Rath, T.~M., Carreras, M., and Sebastiani, P. (2003).
\newblock Automated detection of influenza epidemics with hidden markov models.
\newblock In {\em International Symposium on Intelligent Data Analysis}, pages
  521--532. Springer.

\bibitem[Saghafian, 2018]{saghafian2018ambiguous}
Saghafian, S. (2018).
\newblock Ambiguous partially observable {Markov} decision processes:
  Structural results and applications.
\newblock {\em Journal of Economic Theory}, 178.

\bibitem[Smallwood and Sondik, 1973]{smallwood1973optimal}
Smallwood, R.~D. and Sondik, E.~J. (1973).
\newblock The optimal control of partially observable {Markov} processes over a
  finite horizon.
\newblock {\em Operations Research}, 21(5):1071--1088.

\bibitem[Smith and Simmons, 2004]{smith2004heuristic}
Smith, T. and Simmons, R. (2004).
\newblock Heuristic search value iteration for {POMDPs}.
\newblock In {\em Proceedings of the 20th conference on Uncertainty in
  artificial intelligence}, pages 520--527. UAI Press.

\bibitem[Treharne and Sox, 2002]{treharne2002adaptive}
Treharne, J.~T. and Sox, C.~R. (2002).
\newblock Adaptive inventory control for nonstationary demand and partial
  information.
\newblock {\em Management Science}, 48(5):607--624.

\bibitem[Wiesemann et~al., 2013]{wiesemann2013robust}
Wiesemann, W., Kuhn, D., and Rustem, B. (2013).
\newblock Robust {Markov} decision processes.
\newblock {\em Mathematics of Operations Research}, 38(1):153--183.

\bibitem[Wiesemann et~al., 2014]{wiesemann2014distributionally}
Wiesemann, W., Kuhn, D., and Sim, M. (2014).
\newblock Distributionally robust convex optimization.
\newblock {\em Operations Research}, 62(6):1358--1376.

\bibitem[Xu and Mannor, 2012]{xu2012distributionally}
Xu, H. and Mannor, S. (2012).
\newblock Distributionally robust {Markov} decision processes.
\newblock {\em Mathematics of Operations Research}, 37(2):288--300.

\bibitem[Yang, 2017]{yang2017convex}
Yang, I. (2017).
\newblock A convex optimization approach to distributionally robust {Markov}
  decision processes with {Wasserstein} distance.
\newblock {\em IEEE Control Systems Letters}, 1(1):164--169.

\bibitem[Yu and Xu, 2016]{yu2016distributionally}
Yu, P. and Xu, H. (2016).
\newblock Distributionally robust counterpart in {Markov} decision processes.
\newblock {\em IEEE Transactions on Automatic Control}, 61(9):2538--2543.

\bibitem[Zhang and Denton, 2018]{zhang2018partially}
Zhang, J. and Denton, B.~T. (2018).
\newblock Partially observable markov decision processes for prostate cancer
  screening, surveillance, and treatment: A budgeted sampling approximation
  method.
\newblock {\em Decision Analytics and Optimization in Disease Prevention and
  Treatment}, pages 201--222.

\bibitem[Zymler et~al., 2013]{zymler2013distributionally}
Zymler, S., Kuhn, D., and Rustem, B. (2013).
\newblock Distributionally robust joint chance constraints with second-order
  moment information.
\newblock {\em Mathematical Programming}, 137(1-2):167--198.

\end{thebibliography}

\appendix
\section{Relaxation of $a$-rectangularity}
\label{sec:arectangular}

In this section, we investigate a variant of DR-POMDP where we relax the rectangularity condition of the ambiguity set in the actions. So far, we have only considered the setting where the ambiguity set is rectangular in terms of the states in $\mathcal{S}$ and the actions in $\mathcal{A}$. This is known as $(s,a)$-rectangular set in the literature of \cite{wiesemann2013robust}, who defined the term in the context of robust MDP. Ref.\ \cite{wiesemann2013robust} also considered $s$-rectangular set in robust POMDP, which is only rectangular in terms of the states $\mathcal{S}$. This setting has randomized policy as the optimal policy. We take a similar approach and formulate the Bellman equation:

\begin{small}
\begin{align}\label{eq:vf-v1}
V^t(\bm{b})=\max_{\bm{\phi}\in\Delta(\mathcal{A})}\min_{\mu\in\mathcal{D}}\E_{P\sim\mu}\left[\sum_{a\in\mathcal{A}}\phi_a\sum_{s\in\mathcal{S}}b_s\left(r_{as}+\beta\sum_{z\in\mathcal{Z}}J_z\p_{as} V^{t+1}\left(\bm{f}\left(\bm{b},a,\p_a,z\right)\right)\right)\right],
\end{align}
\end{small}
where $\phi_a$ is the probability for selecting action $a$. We define the ambiguity set to be
\begin{align}
\tilde{\mathcal{D}}_{s}=\left\{\tilde{\mu}_{s}
\begin{pmatrix}
\p_{s} \\ \rb_{s} \\ \ut_{s}
\end{pmatrix}
\middle\vert
\begin{array}{ll}
\E_{(\p_{s},\rb_{s},\ut_{s})\sim\tilde{\mu}_{s}}\left[F_{s}\p_{s}+G_{s}\rb_{s}+H_{s}\ut_{s}\right]=\bm{c}_{s},\\
\tilde{\mu}_{s}\left(\mathcal{X}_{s}\right)=1
\end{array}
\right\},
\end{align}
where $\ut_{s}\in\mathbb{R}^{Q}$ is a vector of auxiliary variables, and 
\begin{align}
\mathcal{X}_{s}=\left\{
\begin{pmatrix}
\p_{s} \\ \rb_{s} \\ \ut_{s}
\end{pmatrix}\in
\begin{matrix}
\mathbb{R}^{|\mathcal{A}|\times|\mathcal{S}|\times|\mathcal{Z}|}\\
\mathbb{R}^{|\mathcal{A}|}\\
\mathbb{R}^{L}
\end{matrix}
\middle\vert\ 
B_{s}\p_{s}+C_{s}r_{s}+E_{s}\ut_{s}\preceq_{K_{s}}\bm{d}_{s}
\right\}.
\end{align}
Here, $F_{s} \in\mathbb{R}^{k\times(|\mathcal{A}|\times|\mathcal{S}|\times|\mathcal{Z}|)}$, $G_{s}\in\mathbb{R}^{k\times |\mathcal{A}|}$, $H_{s}\in\mathbb{R}^{k\times L}$, $\bm{c}_{s}\in\mathbb{R}^{k}$, $B_{s}\in\mathbb{R}^{\ell\times(|\mathcal{A}|\times|\mathcal{S}|\times|\mathcal{Z}|)}$, $C_{s}\in\mathbb{R}^{\ell\times |\mathcal{A}|}$, $E_{s}\in\mathbb{R}^{\ell\times L}$, and $\bm{d}_{s}\in\mathbb{R}^{\ell}$. 

The value function is also convex in the form \eqref{eq:pwlc2}, since for $t<T$,
\begin{align*}
V^t(\bm{b})=\max_{\bm{\phi}\in\Delta(\mathcal{A})}\max_{\substack{\bm{\alpha}_{az}\in\mbox{Conv}\left(\Lambda^{t+1}\right)\\ \forall a\in\mathcal{A},\ z\in\mathcal{Z}}}\sum_{s\in\mathcal{S}}b_s\min_{\left(\hat{\p}_{s},\hat{\rb}_{s},\hat{\ut}_{s}\right)}\quad&\mathrlap{\bm{\phi}^{\top}\left(\beta\sum_{z\in\mathcal{Z}}\left[\left(\bm{\alpha}_{az}^{\top}J_{az}\right)^{\top},\ a\in\mathcal{A}\right]^{\top}\hat{\p}_{s}+\hat{\rb}_{s}\right)}\\
\mbox{s.t.}\quad& F_{s}\hat{\p}_{s}+G_{s}\hat{\rb}_{s}+H_{s}\hat{\ut}_{s}=\bm{c}_{s},&&\forall s\in\mathcal{S}\\
&B_{s}\hat{\p}_{s}+C_{s}\hat{\rb}_{s}+E_{s}\hat{\ut}_{s}\preceq_{K_{s}}\bm{d}_{s},&&\forall s\in\mathcal{S}
\end{align*}
where $J_{az}\in\mathbb{R}^{|\mathcal{S}|\times(|\mathcal{A}|\times|\mathcal{S}|\times|\mathcal{Z}|)}$ is a matrix of zeros and ones that maps $\p_s$ to $\p_{asz}$. For an exact algorithm, we solve the inner minimization problem for all $\phi\in\Delta(\mathcal{A})$, $\bm{\alpha}_{az}\in\mbox{Conv}(\Lambda^{t+1}),\ \forall z\in\mathcal{Z},\ a\in\mathcal{A}$. The optimal objective is used for constructing the set $\Lambda^{t}$, at each time step $t$.

\section{General Ambiguity Set}
\label{app:generalAS}
In this section, we provide a general form of the ambiguity set where the mean values are on an affine manifold, and the supports are conic representable. For all $a\in\mathcal{A}$ and $s\in\mathcal{S}$, we define a non-empty ambiguity set
\begin{align}\label{eq:ax_ambiguityset}
\tilde{\mathcal{D}}_{as}=\left\{\tilde{\mu}_{as}
\begin{pmatrix}
\p_{as} \\ r_{as} \\ \ut_{as}
\end{pmatrix}
\middle\vert
\begin{array}{ll}
\E_{(\p_{as},r_{as},\ut_{as})\sim\tilde{\mu}_{as}}\left[F_{as}\p_{as}+G_{as}r_{as}+H_{as}\ut_{as}\right]=\bm{c}_{as},\\
\tilde{\mu}_{as}\left(\mathcal{X}_{as}\right)=1
\end{array}
\right\},
\end{align}
where $\ut_{as}\in\mathbb{R}^{L}$ is a vector of auxiliary variables, and a support with a non-empty relative interior
\begin{align}\label{eq:ax_support}
\mathcal{X}_{as}=\left\{
\begin{pmatrix}
\p_{as} \\ r_{as} \\ \ut_{as}
\end{pmatrix}\in
\begin{matrix}
\mathbb{R}^{|\mathcal{S}|\times|\mathcal{Z}|}\\
\mathbb{R}\\
\mathbb{R}^{L}
\end{matrix}
\middle\vert\ 
B_{as}\p_{as}+C_{as}r_{as}+E_{as}\ut_{as}\preceq_{K_{as}}\bm{d}_{as}
\right\}.
\end{align}
Here, $F_{as} \in\mathbb{R}^{k\times(|\mathcal{S}|\times|\mathcal{Z}|)}$, $G_{as}\in\mathbb{R}^{k\times 1}$, $H_{as}\in\mathbb{R}^{k\times L}$, $\bm{c}_{as}\in\mathbb{R}^{k}$, $B_{as}\in\mathbb{R}^{\ell\times(|\mathcal{S}|\times|\mathcal{Z}|)}$, $C_{as}\in\mathbb{R}^{\ell\times 1}$, $E_{as}\in\mathbb{R}^{\ell\times L}$, and $\bm{d}_{as}\in\mathbb{R}^{\ell}$. The symbol $\preceq_{K_{as}}$ represents a generalized inequality with respect to a proper cone $K_{as}$. We denote the marginal distribution by $\mu_{as}=\prod_{(\p_{as},r_{as})}\tilde{\mu}_{as}$, and also extend the definition to the ambiguity set so that $\mathcal{D}_{as}=\prod_{(\p_{as},r_{as})}\tilde{\mathcal{D}}_{as}=\bigcup_{\tilde{\mu}_{as}\in\tilde{D}_{as}}\prod_{(\p_{as},r_{as})}\tilde{\mu}_{as}$. The auxiliary variables  $\ut_{as}$ are used for ``lifting" techniques, enabling the representation of nonlinear constraints to linear ones.


\section{Proofs of Theorems \ref{thm:pwl} and \ref{thm:fixedpoint2}}
\label{sec:proofs}
~\\
First, we provide a detailed proof for Theorem \ref{thm:pwl} below. 
\begin{proof}
We show the result by induction. When $t=T$, $V^T(\bm{b})=0$ satisfies \eqref{eq:pwlc2}.
For $t<T$, the inner problem $Q^t(\bm{b},a)$ described in \eqref{eq:Q} becomes \begin{subequations}\label{prob:original}
\begin{align}
\min_{\tilde{\mu}_a\in\mathcal{P}\left({\tilde{\mathcal{X}}}_a\right)}\quad&{\mathrlap{\E_{\left(\p_a,\ut_a\right)\sim\tilde{\mu}_{a}}\left[\sum_{s\in\mathcal{S}}b_s\left(r_{as}+\beta\sum_{z\in\mathcal{Z}}\bm{1}^{\top}{\bm{J}_z}\p_{as}V^{t+1}\left(\bm{f}(\bm{b},a,\p_a,z)\right)\right)\right]}}\\
\label{eq:constraint1}\mbox{s.t.}\quad&{\E_{\left(\p_a,\ut_a\right)\sim\tilde{\mu}_{a}}\left[\ut_{as}\right]=\bm{c}_{as}}, && \forall s\in\mathcal{S}\\
\label{eq:constraint2}&{\E_{\left(\p_a,\ut_a\right)\sim\tilde{\mu}_{a}}\left[I\left(\left(\p_{as},\ut_{as}\right)\in\tilde{\mathcal{X}}_{as}\right)\right]}=1, && \forall s\in\mathcal{S}
\end{align}
\end{subequations}
for all $a\in\mathcal{A}$. Here $I(\cdot)$ is an indicator function, such that if event $\cdot$ is true, it returns value 1 and 0 otherwise. Associating the dual variables $\bm{\rho}_{as}$ and $\omega_{as}$ with constraints \eqref{eq:constraint1} and \eqref{eq:constraint2}, respectively, we formulate the dual of \eqref{prob:original} as
\footnotesize
\begin{subequations}
\begin{align}
\label{eq:dualmu}\max_{\bm{\rho}_a,\bm{\omega}_a}\quad&\sum_{s\in\mathcal{S}}\bm{c}_{as}^{\top}\bm{\rho}_{as}+\sum_{s\in\mathcal{S}}\omega_{as}\\
\label{eq:dualmu1}\mbox{s.t.}\quad&\mathrlap{\sum_{s\in\mathcal{S}}{\ut_{as}}^{\top}\bm{\rho}_{as}+\sum_{s\in\mathcal{S}}\omega_{as}}\\
&\quad \leq \sum_{s\in\mathcal{S}}b_s\left(r_{as}+\beta\sum_{z\in\mathcal{Z}}\bm{1}^{\top}{\bm{J}_z}\p_{as}V^{t+1}\left(\bm{f}(\bm{b},a,\p_a,z)\right)\right) &&\forall {\left(\p_{a},\ut_{a}\right)\in\tilde{\mathcal{X}}_{a}}\nonumber\\
&\bm{\rho}_{as}\in\mathbb{R}^{{|\mathcal{S}|\times|\mathcal{Z}|}},\ \omega_{as}\in\mathbb{R}&&\forall s\in\mathcal{S}.
\end{align}
\end{subequations}
\normalsize
Constraints \eqref{eq:dualmu1} are further equivalent to the following inequality with a minimization problem on the right-hand side (RHS).  
\small
\begin{subequations}\label{eq:mindualmu}
\begin{align}
\sum_{s\in\mathcal{S}}\omega_{as}\leq \\
\min_{\left(\p_a,\ut_a\right)}\quad&\sum_{s\in\mathcal{S}}b_s\left(r_{as}+\beta\sum_{z\in\mathcal{Z}}\bm{1}^{\top}{\bm{J}_z}\p_{as}V^{t+1}\left(\bm{f}(\bm{b},a,\p_a,z)\right)\right) \mathrlap{-\sum_{s\in\mathcal{S}}\ut_{as}^{\top}\bm{\rho}_{as}}\nonumber\\
\label{eq:supportconstraints1}\mbox{s.t.}\quad&
\ut_{as}\geq\p_{as}-\bar{\p}_{as}&&\forall s\in\mathcal{S}\\
\label{eq:supportconstraints2}&\ut_{as}\geq\bar{\p}_{as}-\p_{as}&&\forall s\in\mathcal{S}\\
\label{eq:supportconstraints3}&\bm{1}^{\top}\p_{as}=1&&\forall s\in\mathcal{S}\\
\label{eq:supportconstraints4}&\p_{as}\geq 0 &&\forall s\in\mathcal{S}.
\end{align}
\end{subequations}

\normalsize
Substituting \eqref{eq:pwlc2} for $V^{t+1}$ and \eqref{eq:bayesian} for $\bm{f}(\bm{b},a,\p_a,z)$, we obtain 
\small
\begin{align}
\label{eq:mindualmu2}
\mbox{RHS of }\eqref{eq:mindualmu}=\min_{{\left(\p_a,\ut_a\right)}}\quad&\sum_{s\in\mathcal{S}}b_s r_{as}+\beta\sum_{z\in\mathcal{Z}}\max_{\bm{\alpha}_{az}\in\Lambda^{t+1}}\left[\bm{\alpha}_{az}^{\top}\sum_{s\in\mathcal{S}}{\bm{J}_z}\p_{as}b_{s}\right] \mathrlap{-\sum_{s\in\mathcal{S}}{\ut_{as}}^{\top}\bm{\rho}_{as}}\\
\mbox{s.t.}\quad&{\eqref{eq:supportconstraints1}\mbox{--}\eqref{eq:supportconstraints4}}. \nonumber
\end{align}
\normalsize
Since the objective of the maximization problem is linear in terms of $\bm{\alpha}_{az}, \forall z\in\mathcal{Z}$, the optimal objective value does not change by taking the convex hull of $\Lambda^{t+1}$, denoted as $\mbox{Conv}\left(\Lambda^{t+1}\right)$. Bringing the maximization to the front, we have
\small
\begin{align}
\label{eq:mindualmu3}
\eqref{eq:mindualmu2}=\min_{{\left(\p_a,\ut_a\right)}}\quad&\max_{\substack{\bm{\alpha}_{az}\in\mbox{Conv}\left(\Lambda^{t+1}\right) \\ \forall z\in\mathcal{Z}}}\Biggl[\sum_{s\in\mathcal{S}}b_s r_{as}+\beta\sum_{z\in\mathcal{Z}}\bm{\alpha}_{az}^{\top}\sum_{s\in\mathcal{S}}{\bm{J}_z}\p_{as}b_{s} \mathrlap{-\sum_{s\in\mathcal{S}}{\ut_{as}}^{\top}\bm{\rho}_{as}\Biggr]}\\
\mbox{s.t.}\quad&{\eqref{eq:supportconstraints1}\mbox{--}\eqref{eq:supportconstraints4}}\nonumber
\end{align}
\normalsize
The expression in the bracket is convex (linear) in $\left(\p_a,\ut_a\right)$ for fixed $\bm{\alpha}_{az},\ z\in\mathcal{Z}$, and concave (affine) in $\bm{\alpha}_{az},\ z\in\mathcal{Z}$ given fixed values of $\left(\p_a,\ut_a\right)$. Moreover, \eqref{eq:supportconstraints1}\mbox{--}\eqref{eq:supportconstraints4} and $\mbox{Conv}\left(\Lambda^{t+1}\right)$ are convex sets. The minimax theorem (see, e.g., \cite{osogami2015robust}, \cite{du2013minimax}) ensures that the problem is equivalent to
\small
\begin{align}
\label{eq:mindualmu4}
\eqref{eq:mindualmu3}=\max_{\substack{\bm{\alpha}_{az}\in\mbox{Conv}\left(\Lambda^{t+1}\right) \\ \forall z\in\mathcal{Z}}}\min_{{\left(\p_a,\ut_a\right)}}\quad&\sum_{s\in\mathcal{S}}b_s r_{as}+\beta\sum_{z\in\mathcal{Z}}\bm{\alpha}_{az}^{\top}\sum_{s\in\mathcal{S}}{\bm{J}_z}\p_{as}b_{s} -\sum_{s\in\mathcal{S}}{\ut_{as}}^{\top}\bm{\rho}_{as}\\
\mbox{s.t.}\quad&\eqref{eq:supportconstraints1}\mbox{--}\eqref{eq:supportconstraints4}\nonumber
\end{align}
\normalsize
{We take the dual of the inner minimization by associating dual variables $\bm{\kappa}_{as}^1$, $\bm{\kappa}_{as}^2$, $\sigma_{as}$ with constraints \eqref{eq:supportconstraints1}--\eqref{eq:supportconstraints3}, respectively. We thus have the following equivalence: 
\footnotesize
\begin{subequations}\label{eq:mindualmu6}
\begin{align}
\eqref{eq:mindualmu4}=\max_{\substack{\bm{\alpha}_{az}\in\mbox{Conv}\left(\Lambda^{t+1}\right)\\ \forall z\in\mathcal{Z}}}\max_{\bm{\kappa}_{a}^1,\bm{\kappa}_{a}^2,\bm{\sigma}_{a}}\quad&\mathrlap{\sum_{s\in\mathcal{S}}b_s r_{as}+\sum_{s\in\mathcal{S}}\left(-\bar{p}_{as}^\top\bm{\kappa}_{as}^1+\bar{p}_{as}^\top\bm{\kappa}_{as}^2+\sigma_{as}\right)}\\
\label{eq:dualp}\mbox{s.t.}\quad&\beta b_{s}\sum_{z\in\mathcal{Z}}{\bm{J}_z}^{\top}\bm{\alpha}_{az}+\bm{\kappa}_{as}^1-\bm{\kappa}_{as}^2-\bm{1}\sigma_{as}\geq0,&&\forall s\in\mathcal{S}\\
\label{eq:dualut}&\bm{\kappa}_{as}^1+\bm{\kappa}_{as}^2+\bm{\rho}_{as}=0,&& \forall s\in\mathcal{S}\\
\label{eq:dualvar}&\bm{\kappa}_{as}^1,\bm{\kappa}_{as}^2\in \mathbb{R}_+^{|\mathcal{S}|\times|\mathcal{Z}|},\sigma_{as}\in\mathbb{R},&& \forall s\in\mathcal{S},
\end{align}
\end{subequations}
\normalsize
Due to \eqref{eq:mindualmu}, we substitute $\sum_{s\in\mathcal{S}}\omega_{as}$ in the objective function \eqref{eq:dualmu} with  \eqref{eq:mindualmu6}.} As a result, the value function \eqref{eq:vf2} is equivalent to
\footnotesize
\begin{subequations}\label{eq:maxproblem}
\begin{align}
V^t(\bm{b})=\max_{a\in\mathcal{A}}\max_{\substack{\bm{\alpha}_{az}\in\mbox{Conv}\left(\Lambda^{t+1}\right)\\ \forall z\in\mathcal{Z}}}&\\
\nonumber\max_{\bm{\rho}_a, \bm{\kappa}_{a}^1,\bm{\kappa}_{a}^2,\bm{\sigma}_{a}}\quad&\sum_{s\in\mathcal{S}}\bm{c}_{as}^{\top}\bm{\rho}_{as}+\sum_{s\in\mathcal{S}}b_s r_{as}+\sum_{s\in\mathcal{S}}\left(-\bar{p}_{as}^\top\bm{\kappa}_{as}^1+\bar{p}_{as}^\top\bm{\kappa}_{as}^2+\sigma_{as}\right)\\
\nonumber\mbox{s.t.}\quad&\eqref{eq:dualp}\mbox{--}\eqref{eq:dualvar}\\
\label{eq:dualrho}&\bm{\rho}_{as}\in\mathbb{R}^{|\mathcal{S}|\times|\mathcal{Z}|}\ \forall s\in\mathcal{S},
\end{align}
\end{subequations}
\normalsize
and after taking the dual of the most inner maximization problem, we have 
\begin{align}
\label{eq:valuefunc}V^t(\bm{b})=\max_{a\in\mathcal{A}}\max_{\substack{\bm{\alpha}_{az}\in\mbox{Conv}\left(\Lambda^{t+1}\right)\\ \forall z\in\mathcal{Z}}}\sum_{s\in\mathcal{S}}b_s\times\Xi(a,\bm{\alpha}_{az}\ \forall z\in\mathcal{Z},s),
\end{align}
where 
\begin{subequations}
\begin{align}
\Xi(a,\bm{\alpha}_{az}\ \forall z\in\mathcal{Z},s)=\min_{\left(\p_{as},\ut_{as}\right)}\quad&\beta\sum_{z\in\mathcal{Z}}\bm{\alpha}_{az}^{\top} \bm{J}_z \p_{as}+r_{as}\\
\mbox{s.t.}\quad&
\bm{c}_{as}\geq\p_{as}-\bar{\p}_{as}\\
&\bm{c}_{as}\geq\bar{\p}_{as}-\p_{as}\\
&\bm{1}^{\top}\p_{as}=1\\
&\p_{as}\geq 0.
\end{align}
\end{subequations}
Defining set $\Lambda^t$ as
\begin{align*}
\left\{\ \left(
\Xi(a,\bm{\alpha}_{az}\ \forall z\in\mathcal{Z},s),\ s\in\mathcal{S}
\right)^{\top}
\middle\vert\ 
\begin{array}{c}
\forall a\in\mathcal{A},\\
\forall \bm{\alpha}_{az}\in\mbox{Conv}\left(\Lambda^{t+1}\right),\ \forall z\in\mathcal{Z}
\end{array}
\right\},
\end{align*}
it follows that the above value function in \eqref{eq:valuefunc} is of the form \eqref{eq:pwlc2}. Furthermore, by induction, this is true for all $t$. This completes the proof. 
\end{proof}


The proof of Theorem \ref{thm:fixedpoint2} is given as follows. 
\begin{proof}
Consider two arbitrary value functions $V_1$ and $V_2$. Given belief state $\bm{b}$, let
\begin{align*}
a_i^{\star}=\argmax_{a\in\mathcal{A}}\min_{\mu_{a}\in\tilde{\mathcal{D}}_{a}}&\E_{(\p_a,\rb_a)\sim\mu_{a}}\left[\sum_{s\in\mathcal{S}}b_s\left(r_{as}+\beta\sum_{z\in\mathcal{Z}}\bm{1}^{\top}{\bm{J}_z}\p_{as}V_i\left(\bm{f}(\bm{b},a,\p_a,z)\right)\right)\right],
\end{align*}
for $i=1, 2$, and for all actions $a\in\mathcal{A}$, denote 
\begin{align*}
\mu_{a,i}^{\star}=\argmin_{\mu_{a}\in\tilde{\mathcal{D}}_{a}}&\E_{(\p_a,\rb_a)\sim\mu_{a}}\left[\sum_{s\in\mathcal{S}}b_s\left(r_{as}+\beta\sum_{z\in\mathcal{Z}}\bm{1}^{\top}{\bm{J}_z}\p_{as}V_i\left(\bm{f}(\bm{b},a,\p_a,z)\right)\right)\right]
\end{align*}
for $i=1,2$. First, suppose that $\mathcal{L}V_1(\bm{b})\geq\mathcal{L}V_2(\bm{b})$. Then,
\scriptsize
\begin{align}
0&\leq\mathcal{L}V_1(\bm{b})-\mathcal{L}V_2(\bm{b})\nonumber\\
&=\E_{(\p_{a_1^{\star}},\rb_{a_1^{\star}})\sim\mu_{a_1^{\star},1}^{\star}}\left[\sum_{s\in\mathcal{S}}b_s\left(r_{a_1^{\star}s}+\beta\sum_{z\in\mathcal{Z}}\bm{1}^{\top}{\bm{J}_z}\p_{a_1^{\star}s}V_1\left(\bm{f}(\bm{b},a_1^{\star},\p_{a_1^{\star}},z)\right)\right)\right]\nonumber\\
&\quad-\E_{(\p_{a_2^{\star}},\rb_{a_2^{\star}})\sim\mu_{a_2^{\star},2}^{\star}}\left[\sum_{s\in\mathcal{S}}b_s\left(r_{a_2^{\star}s}+\beta\sum_{z\in\mathcal{Z}}\bm{1}^{\top}{\bm{J}_z}\p_{a_2^{\star}s}V_2\left(\bm{f}(\bm{b},a_2^{\star},\p_{a_2^{\star}},z)\right)\right)\right]\nonumber\\
&\leq\E_{(\p_{a_1^{\star}},\rb_{a_1^{\star}})\sim\mu_{a_1^{\star},2}^{\star}}\left[\sum_{s\in\mathcal{S}}b_s\left(r_{a_1^{\star}s}+\beta\sum_{z\in\mathcal{Z}}\bm{1}^{\top}{\bm{J}_z}\p_{a_1^{\star}s}V_1\left(\bm{f}(\bm{b},a_1^{\star},\p_{a_1^{\star}},z)\right)\right)\right]\nonumber\\
&\quad-\E_{(\p_{a_1^{\star}},\rb_{a_1^{\star}})\sim\mu_{a_1^{\star},2}^{\star}}\left[\sum_{s\in\mathcal{S}}b_s\left(r_{a_1^{\star}s}+\beta\sum_{z\in\mathcal{Z}}\bm{1}^{\top}{\bm{J}_z}\p_{a_1^{\star}s}V_2\left(\bm{f}(\bm{b},a_1^{\star},\p_{a_1^{\star}},z)\right)\right)\right]\nonumber\\
\label{eq:contraction1}&=\beta\E_{(\p_{a_1^{\star}},\rb_{a_1^{\star}})\sim\mu_{a_1^{\star},2}^{\star}}\Biggl[\sum_{s\in\mathcal{S}}b_s\sum_{z\in\mathcal{Z}}\bm{1}^{\top}{\bm{J}_z}\p_{a_1^{\star}s} \times\left(V_1\left(\bm{f}(\bm{b},a_1^{\star},z,\p_{a_1^{\star}})\right)-V_2\left(\bm{f}(\bm{b},a_1^{\star},\p_{a_1^{\star}},z)\right)\right)\Biggr].
\end{align}
\normalsize
The inequality follows that we replace the nature's optimal decision $\mu_{a_1^{\star},1}^{\star}$ for $V_1$ by $\mu_{a_1^{\star},2}^{\star}$, and replace the DM's optimal solution $a_2^{\star}$ for $V_2$ by $a_1^{\star}$. Then, by changing the difference between $V_1$ and $V_2$ to the absolute value of the difference, we have 
\scriptsize
\begin{align*}
\eqref{eq:contraction1}&\leq\beta\E_{(\p_{a_1^{\star}},\rb_{a_1^{\star}})\sim\mu_{a_1^{\star},2}^{\star}}\Biggl[\sum_{s\in\mathcal{S}}b_s\sum_{z\in\mathcal{Z}}\bm{1}^{\top}{\bm{J}_z}\p_{a_1^{\star}s}\times\left|V_1\left(\bm{f}(\bm{b},a_1^{\star},\p_{a_1^{\star}},z)\right)-V_2\left(\bm{f}(\bm{b},a_1^{\star},z,\p_{a_1^{\star}})\right)\right|\Biggr]\\
&\leq\beta\E_{(\p_{a_1^{\star}},\rb_{a_1^{\star}})\sim\mu_{a_1^{\star},2}^{\star}}\left[\sum_{s\in\mathcal{S}}b_s\sum_{z\in\mathcal{Z}}\bm{1}^{\top}{\bm{J}_z}\p_{a_1^{\star}s}\sup_{\bm{b}' \in\Delta(\mathcal{S})}\left|V_1(\bm{b}')-V_2(\bm{b}')\right|\right]\\
&=\beta\sup_{\bm{b}' \in\Delta(\mathcal{S})}\left|V_1(\bm{b}')-V_2(\bm{b}')\right|.
\end{align*}
\normalsize
The second inequality follows that we take the supremum for all belief states $\bm{b}'\in\Delta(\mathcal{S})$, and the last equality is because $\E_{(\p_{a_1^{\star}},\rb_{a_1^{\star}})\sim\mu_{a_1^{\star},2}^{\star}}\left[\sum_{s\in\mathcal{S}}b_s\sum_{z\in\mathcal{Z}}\bm{1}^{\top}{\bm{J}_z}\p_{a_1^{\star}s}\right]=1$.

The same result holds for the case where $\mathcal{L}V_1(\bm{b})<\mathcal{L}V_2(\bm{b})$. Thus, for any belief state value $\bm{b}$, it follows that
\begin{align*}
\left|\mathcal{L}V_1(\bm{b})-\mathcal{L}V_2(\bm{b})\right|\leq\beta\sup_{\bm{b}' \in\Delta(\mathcal{S})}\left|V_1(\bm{b}')-V_2(\bm{b}')\right|,
\end{align*}
and therefore,
\begin{align*}
\sup_{\bm{b} \in\Delta(\mathcal{S})}\left|\mathcal{L}V_1(\bm{b})-\mathcal{L}V_2(\bm{b})\right|\leq\beta\sup_{\bm{b}' \in\Delta(\mathcal{S})}\left|V_1(\bm{b}')-V_2(\bm{b}')\right|,
\end{align*}
yielding that $\mathcal{L}$ is a contraction under $0<\beta<1$. This completes the proof. 
\end{proof}

\end{document}